\documentclass[11pt,letterpaper]{article}

\usepackage[margin=1in]{geometry}

\usepackage[utf8]{inputenc} 
\usepackage{enumitem}        
\usepackage[T1]{fontenc}    
\usepackage{hyperref}       
\hypersetup{
    colorlinks=true,
    citecolor=blue,
    linkcolor=blue,
    urlcolor=blue,
    breaklinks=true
}
\usepackage{url}            
\usepackage{booktabs}       
\usepackage{amsfonts}       
\usepackage{nicefrac}       
\usepackage{microtype}      
\usepackage{xcolor}         
\usepackage{amsmath, amssymb, amsthm}
\usepackage{mathtools}
\usepackage{graphicx}
\usepackage{algorithm}
\usepackage{algpseudocode}
\usepackage{todonotes}

\newtheorem{theorem}{Theorem}[section]

\newtheorem{lemma}[theorem]{Lemma}
\newtheorem{proposition}[theorem]{Proposition}
\newtheorem{definition}{Definition}[section]
\newtheorem{assumption}{Assumption}[section]
\theoremstyle{remark}
\newtheorem{remark}{Remark}[section]

\DeclareMathOperator*{\argmax}{arg\,max}
\DeclareMathOperator*{\argmin}{arg\,min}
\DeclarePairedDelimiter{\norm}{\lVert}{\rVert}

\usepackage{subcaption}

\usepackage[numbers,sort&compress]{natbib}

\title{Stochastic Dynamic Barrier Perturbed Gradient Methods for Nonconvex Simple Bilevel Optimization}

\author{
\parbox{\textwidth}{
\centering
Mohammad Mahdi Ahmadi\thanks{Department of Systems and Industrial Engineering, The University of Arizona, Tucson, AZ, USA.\\
Emails:
\texttt{ahmadi@arizona.edu, erfany@arizona.edu}}
\quad
Jincheng Cao\thanks{Department of Electrical and Computer Engineering, The University of Texas at Austin, Austin, TX, USA.\\
Emails: \texttt{jinchengcao@utexas.edu, mokhtari@austin.utexas.edu}}
\quad
Aryan Mokhtari$^{\thefootnote}$
\quad\\
Erfan Yazdandoost Hamedani$^{*}$
}
}
\date{}

\begin{document}

\maketitle

\begin{abstract}
We study stochastic simple bilevel optimization with smooth, possibly nonconvex upper- and lower-level objectives accessed only through stochastic gradient oracles. A key challenge is that the dual multiplier induced by the lower-level constraint may become unbounded near lower-level stationary points, invalidating bounded-dual analyses and destabilizing stochastic gradient estimates. To address this, we propose \emph{Stochastic Dynamic Barrier Perturbed Gradient} (SDBPG), a single-loop method that adaptively perturbs the dual formulation to regularize this degeneracy. The perturbation stabilizes the multiplier and yields controlled bias and variance even near the lower-level stationarity region. Under a rare-visit assumption governed by a parameter $\delta \in (0, \tfrac{1}{2}]$, SDBPG finds an $(\epsilon, \epsilon)$-stationary point in $\mathcal{O}(\epsilon^{-1/\delta})$ iterations, with sample gradient complexities $\mathcal{O}(\epsilon^{-2/\delta})$ and $\mathcal{O}(\epsilon^{-3/\delta})$ for the upper- and lower-level objectives, where larger $\delta$ corresponds to rarer visits to the bad region describing the negative alignment between the two objectives when the lower-level gradient is small. 
We further develop PR-SDBPG, a penalty-regularized variant that eliminates the rare-visit assumption, and VR-PR-SDBPG, which improves the resulting sample complexities entirely through variance reduction. To our knowledge, these are the first explicit $(\epsilon_f,\epsilon_g)$-stationarity guarantees for stochastic nonconvex-nonconvex simple bilevel optimization. 
\end{abstract}

\section{Introduction}

Bilevel optimization has emerged as a fundamental framework for modeling hierarchical decision-making problems. A particularly important subclass is \emph{simple bilevel optimization}, where the goal is to minimize an upper-level objective over the solution set of a lower-level problem~\citep{dempe2010optimality,dutta2020algorithms,jiang2023conditional,cao2024accelerated}. This class of problems has recently gained significant attention in machine learning and artificial intelligence due to its wide range of applications, including continual learning~\citep{borsos2020coresets,hao2023bilevel}, hyperparameter optimization~\citep{franceschi2018bilevel,shaban2019truncated}, meta-learning~\citep{rajeswaran2019meta,bertinetto2018meta}, reinforcement learning~\citep{hong2023two,zhang2020bi}, and machine unlearning~\citep{sekhari2021remember,liu2024towards,reisizadeh2026blur}.

In this work, we focus on the stochastic variant of the simple bilevel problem, motivated by large-scale learning settings where exact evaluations are infeasible. Both the upper- and lower-level objectives are expressed as expectations accessible only through sampled realizations, leading to the stochastic nonconvex bilevel problem
\begin{align}\label{eq:mainn1}
\min_{\mathbf{x} \in \mathbb{R}^n} \; f(\mathbf{x}) = \mathbb{E}\big[\tilde f(\mathbf{x}, \xi_f)\big] \quad
\text{s.t.} \quad \mathbf{x} \in \mathcal{X}_g^* := \argmin_{\mathbf{z} \in \mathbb{R}^n} \; g(\mathbf{z}) = \mathbb{E}\big[\tilde g(\mathbf{z}, \xi_g)\big],
\end{align}
where $\tilde f, \tilde g$ are continuously differentiable (possibly nonconvex) functions, and $\xi_f$, $\xi_g$ are independent random variables drawn from unknown distributions. {A prominent application of the nonconvex bilevel setting in~\eqref{eq:mainn1} is large language model (LLM) unlearning~\citep{reisizadeh2026blur}, which can be formulated as a simple bilevel problem whose objectives are generally nonconvex due to the neural network structure.}

A central difficulty in Problem~\eqref{eq:mainn1} is that the feasible set $\mathcal{X}_g^*$ is implicit and it is defined through the solution set of the lower-level problem rather than by explicit constraints. This implicit structure precludes several standard algorithmic approaches, including projection-based methods, since projecting onto $\mathcal{X}_g^*$ is essentially as hard as solving the lower-level problem. Regularization schemes, such as Tikhonov regularization~\citep{tikhonov1977solutions,solodov2007explicit}, typically require convexity of the lower-level objective and therefore do not apply in our setting. Moreover, the value function reformulation approach $\min f(\mathbf{x})$ subject to $g(\mathbf{x}) \leq g^*$, where $g^*$ denotes the optimal lower-level value, is also intractable as $g$ is nonconvex, hence, $g^*$ is generally difficult to compute, and the associated dual variables may be unbounded~\citep{cartis2014complexity,facchinei2021ghost}.

Among first-order methods, Dynamic Barrier Gradient Descent (DBGD)~\citep{gong2021bi} provides a natural approach for this problem class. Rather than projecting onto or estimating $\mathcal{X}_g^*$, it constructs search directions by adaptively combining $\nabla f$ and $\nabla g$ through a dual multiplier $\lambda_k$ that regulates the interaction between the two levels. A discrete-time variant analyzed by~\citet{cao2025complexity} introduced $(\epsilon_f,\epsilon_g)$-stationarity (Definition~\ref{def:stationary}) and established iteration complexity $\mathcal{O}(\max(\epsilon_f^{-2},\epsilon_g^{-2}))$ in the deterministic nonconvex-nonconvex setting. However, DBGD faces a fundamental challenge in which the multiplier scales as $\lambda_k \propto \|\nabla g(\mathbf{x}_k)\|^{-1}$ and diverges near lower-level stationarity. While manageable in deterministic settings due to exact gradients and controlled divergence, this degeneracy is exacerbated under stochastic gradients, where noise
amplifies the instability. In particular, the stochastic multiplier $\tilde{\lambda}_k$, defined as a ratio of correlated mini-batch random variables, is generally biased and lacks bounded variance.

These challenges have motivated recent work summarized in Table~\ref{tab:comparison}. 
An alternative deterministic approach by \citet{reisizadeh2026blur} avoids the DBGD dual variable via a hierarchical gradient method with temporal averaging, achieving complexities $\mathcal{O}(\epsilon_f^{-4})$ and $\mathcal{O}(\epsilon_g^{-2})$, but requiring exact gradients. 
In the stochastic nonconvex-nonconvex setting, the only existing method is BLOOP~\citep{hsieh2024careful}, which projects $\nabla f$ onto the subspace orthogonal to $\nabla g$, thereby avoiding the unbounded dual issue, but achieves only $\mathcal{O}(\epsilon_g^{-4})$ for the lower level with no guarantees on upper-level stationarity. When $g$ is convex, stronger results are possible~\citep{cao2023projection,giang2025conditional}, but these rely on convexity and do not apply here.

\textbf{Gap.} No existing method provides joint $(\epsilon_f,\epsilon_g)$-stationarity guarantees with explicit sample complexity bounds in the stochastic nonconvex-nonconvex simple bilevel setting.

\paragraph{Contributions.}
We propose two main single-loop, first-order algorithms for stochastic nonconvex simple bilevel optimization, both extending the deterministic DBGD framework~\citep{cao2025complexity} to the stochastic setting. The key contributions are:

\begin{itemize}[leftmargin=1.5em]
    \item \textbf{SDBPG: stochastic DBGD with dual perturbation (Section~\ref{sec:sdbpg}).} We introduce a perturbation $\gamma_k$ into the dual variable denominator, replacing $\|\nabla g\|^2$ with $\|\nabla g\|^2 + \gamma_k$. In the stochastic setting, the dual variable involves a ratio of \emph{correlated random variables}, both numerator and denominator share the same mini-batch, making the estimator neither unbiased nor bounded-variance. The perturbation regularizes this ratio, enabling a controlled bias-variance tradeoff. 
    Under a ``rare-visit'' assumption on the iterates, SDBPG achieves $(\epsilon,\epsilon)$-stationarity with sample complexities $\mathcal{O}(\epsilon^{-2/\delta})$ and $\mathcal{O}(\epsilon^{-3/\delta})$ for the upper- and lower-level objectives, respectively, where $\delta\in(0,\tfrac12]$ measures the rarity of visits to regions of gradient misalignment near lower-level stationary points; larger $\delta$ corresponds to rarer visits to the bad region, with $\delta=\tfrac12$ representing the case of uniformly bounded cumulative expected visits; in particular, an empty bad region is a sufficient special case (Theorem~\ref{thm:theorem2}).

    \item \textbf{PR-SDBPG: penalty-regularized SDBPG (Section~\ref{sec:pr-sdbpg}).} To eliminate the rare-visit assumption entirely, we redesign the subproblem using a penalty-regularized unconstrained objective whose denominator is always bounded below by $\gamma > 0$. This yields a multiplier that is uniformly bounded regardless of $\|\nabla g(\mathbf{x}_k)\|$, allowing the batch sizes to be \emph{fixed constants} set a priori. PR-SDBPG achieves the same $(\epsilon,\epsilon)$-stationarity \emph{without} the rare-visit assumption, at the cost of higher sample complexities: $\mathcal{O}(\epsilon^{-8})$ and $\mathcal{O}(\epsilon^{-10})$ for the upper and lower levels, respectively (Theorem~\ref{thm:pr-sdbpg}). Moreover, incorporating STORM-type variance reduction yields a \emph{variance-reduced PR-SDBPG} (VR-PR-SDBPG) that attains $(\epsilon, \epsilon)$-stationarity with improved sample complexities $\mathcal{O}(\epsilon^{-6})$ and $\mathcal{O}(\epsilon^{-8})$ for the upper and lower levels, respectively (Theorem~\ref{rem:vr-pr-sdbpg}).

    \item \textbf{Novel joint Lipschitz analysis.} The stochastic descent direction $\nabla\tilde{f} + \tilde{\lambda}_{\gamma,k}\nabla\tilde{g}$ cannot be analyzed by separately controlling $\tilde{\lambda}_{\gamma,k}$ and $\nabla\tilde{g}$, since they share mini-batch randomness. To address this, we introduce a joint Lipschitz analysis of the map $T_\gamma(u,v) = \frac{\langle u,v\rangle}{\|v\|^2+\gamma}v$ that directly bounds the bias and variance of the stochastic direction without requiring independence, which is a key ingredient in the analysis of both algorithms.

    \item \textbf{Comparison with inexact alternatives.} We also analyze SDBPG variants that employ stochastic primal-dual inner solvers (ConEx and SRPD) and show that the direct perturbation approach yields strictly better sample complexity than these alternatives.
\end{itemize}

\begin{table}[t!]
\centering
\caption{Comparison of results for simple bilevel optimization. S/D denote stochastic/deterministic settings; C/NC denote convex/nonconvex objectives. Sample complexity refers to the total number of stochastic gradient evaluations.}
\label{tab:comparison}
\small
\begin{tabular}{@{}l c c c c c l@{}}
\toprule
\textbf{Reference} & \textbf{Type} & \textbf{$f$} & \textbf{$g$} & \textbf{Upper} & \textbf{Lower} & \textbf{Extra assumption} \\
\midrule
\citet{jalilzadeh2024stochastic} & S & C & C & \multicolumn{2}{c}{$\mathcal{O}(\epsilon^{-4})$} & Convexity \\
\citet{cao2023projection} & S & NC & C & \multicolumn{2}{c}{$\tilde{\mathcal{O}}(\epsilon^{-3})$} & $g$ convex \\
\citet{giang2025conditional} & S & NC & C & \multicolumn{2}{c}{$\tilde{\mathcal{O}}(\epsilon^{-7})$} & $g$ convex \\
\citet{hsieh2024careful} & S & NC & NC & -- & $\mathcal{O}(\epsilon_g^{-4})$ & Lower-level only \\
\citet{reisizadeh2026blur} & D & NC & NC & \multicolumn{2}{c}{N/A (deterministic)} & -- \\
\citet{cao2025complexity} & D & NC & NC & \multicolumn{2}{c}{N/A (deterministic)} & -- \\
\midrule
\textbf{SDBPG} (Thm.~\ref{thm:theorem2}) & \textbf{S} & \textbf{NC} & \textbf{NC} & $\mathbf{\mathcal{O}(\epsilon^{-2/\delta})}$ & $\mathbf{\mathcal{O}(\epsilon^{-3/\delta})}$ & \textbf{Rare-visit $\delta\in (0,\tfrac12]$} \\
\textbf{PR-SDBPG} (Thm.~\ref{thm:pr-sdbpg}) & \textbf{S} & \textbf{NC} & \textbf{NC} & $\mathbf{\mathcal{O}(\epsilon^{-8})}$ & $\mathbf{\mathcal{O}(\epsilon^{-10})}$ & \textbf{None} \\
\textbf{VR-PR-SDBPG} (Thm.~\ref{rem:vr-pr-sdbpg}) & \textbf{S} & \textbf{NC} & \textbf{NC} & $\mathbf{\mathcal{O}(\epsilon^{-6})}$ & $\mathbf{\mathcal{O}(\epsilon^{-8})}$ & \textbf{None} \\
\bottomrule
\end{tabular}
\end{table}

\subsection{Additional related work}

\paragraph{Nonconvex constrained optimization.}
Problem~\eqref{eq:mainn1} can be reformulated as $\min f(\mathbf{x})$ s.t.\ $g(\mathbf{x}) \leq g^*$, suggesting a connection to nonconvex constrained optimization. Several works provide complexity guarantees in this setting~\citep{cartis2014complexity,facchinei2021ghost,cartis2017corrigendum,cartis2019evaluation,birgin2016evaluation,lin2019inexact}. However, these methods typically require knowledge of $g^*$ or impose constraint qualifications (e.g., MFCQ) that fail in the bilevel setting. Stochastic extensions~\citep{boob2023stochastic,shi2026momentum,yang2026stochastic,shen2025sequential,cui2025two} face similar limitations and often require second-order oracle information to characterize lower-level optimality, making them less suitable for large-scale problems.

\paragraph{General bilevel optimization.}
A broader class of bilevel problems considers formulations where the upper-level objective depends on an auxiliary variable $\mathbf{y}$ appearing in both levels. Several works address nonconvex variants of this general formulation~\citep{liu2022bome,xiao2023generalized,chen2024finding,ahmadi2025single,kwon2023penalty,khanduri2021near,chen2021tighter,abolfazli2025perturbed,sharifi2025sequential}. However, many of these impose the Polyak--\L{}ojasiewicz (PL) condition on the lower-level objective~\citep{liu2022bome,xiao2023generalized,chen2024finding,kwon2023penalty}, which restricts the problem class. Moreover, the stationarity notions used in these works (e.g., hyper-gradient norm) are defined with respect to upper-level variables that are absent in the simple bilevel structure, making the guarantees not directly comparable to ours.

\section{Preliminaries}\label{Sec:prelim}
\subsection{Assumptions and definitions}
\begin{assumption}\label{assump:gradf-g-lip}
We assume the following conditions hold:
\begin{itemize}
    \item[(i)] The function $f$ is continuously differentiable, and its gradient is Lipschitz continuous with constant $L_f$, i.e.,
    $\|\nabla f(\mathbf{x}) - \nabla f(\mathbf{y})\| \leq L_f \|\mathbf{x} - \mathbf{y}\|, \quad \forall \mathbf{x}, \mathbf{y} \in \mathbb{R}^n$.
    
    
    \item[(ii)] The function $g$ is continuously differentiable, and its gradient is Lipschitz continuous with constant $L_g$, i.e.,
    $\|\nabla g(\mathbf{x}) - \nabla g(\mathbf{y})\| \leq L_g \|\mathbf{x} - \mathbf{y}\|, \quad \forall \mathbf{x}, \mathbf{y} \in \mathbb{R}^n$.
\end{itemize}
In addition, we assume that both objective functions are bounded below, i.e., $f^* > -\infty$ and $g^* > -\infty$, where $f^*$ and $g^*$ denote the infimum values of $f$ and $g$, respectively.
\end{assumption}

\begin{assumption}\label{assum:stoch_grad}
We have access to independent stochastic gradients
$\nabla \tilde f(\mathbf{x},\xi_f) = \nabla f(\mathbf{x}) + \delta_{f}$ and 
$\nabla \tilde g(\mathbf{x},\xi_g) = \nabla g(\mathbf{x}) + \delta_{g}$,
where the noise terms are unbiased ($\mathbb{E}[\delta_f]=\mathbb{E}[\delta_g]=0$) with bounded variance
$\mathbb{E}[\|\delta_f\|^2]\le\nu_f^2$ and 
$\mathbb{E}[\|\delta_g\|^2]\le\nu_g^2$.
\end{assumption}


\begin{assumption}\label{assum:bounded-stoch-g}
The gradients of $f$ and $g$ are uniformly bounded; that is, there exist constants $0 \le G_f, C_g < \infty$ such that, for all $\mathbf{x} \in \mathbb{R}^n$, $\|\nabla f(\mathbf{x})\| \le G_f$ and $\|\nabla g(\mathbf{x})\| \le C_g$.
\end{assumption}

\begin{definition}
Let $\{\mathbf{x}_k\}_{k\ge 0}$ be the sequence of iterates generated by the algorithm, and define the filtration
$\mathcal{F}_k \triangleq \sigma(\mathbf{x}_0,\mathbf{x}_1,\ldots,\mathbf{x}_k)$,
which represents the $\sigma$-algebra generated by the history of iterates up to iteration $k$. 
For any random variable $Z$ defined on the underlying probability space, we define
$\mathbb{E}_k[Z] \;\triangleq\; \mathbb{E}\!\left[\,Z \,\middle|\, \mathcal{F}_k \right]$,
that is, the conditional expectation of $Z$ given the information available up to iteration $k$.
\end{definition}
\paragraph{Stationarity metric.}
We adopt the stationarity metric for nonconvex simple bilevel optimization introduced in~\citep{cao2025complexity}. Informally, we seek a point where (i) the lower-level objective is approximately stationary, and (ii) the upper-level gradient has no significant component along directions that could further decrease $g$.

\begin{definition}\label{def:stationary}
Let $\epsilon_f \geq 0$ and $\epsilon_g \geq 0$. A point $\hat{\mathbf{x}} \in \mathbb{R}^n$ is called an $(\epsilon_f,\epsilon_g)$-stationary point of Problem~\eqref{eq:mainn1} if there exists a scalar $\lambda \geq 0$ such that
\[
\|\nabla g(\hat{\mathbf{x}})\|^2 \leq \epsilon_g,
\quad \text{and} \quad
\|\nabla f(\hat{\mathbf{x}}) + \lambda \nabla g(\hat{\mathbf{x}})\|^2 \leq \epsilon_f.
\]
\end{definition}

The first condition requires the lower-level gradient to be small, ensuring approximate first-order stationarity of $g$. The second condition states that $\nabla f(\hat{\mathbf{x}})$ is approximately expressible as $-\lambda \nabla g(\hat{\mathbf{x}})$ for some $\lambda \geq 0$; in other words, the component of $\nabla f$ orthogonal to $\nabla g$ is small, which is a first-order necessary condition for bilevel optimality (see~\citep{cao2025complexity} for a detailed discussion).

\section{Stochastic DBGD with dual perturbation}\label{sec:sdbpg}
We build on the dynamic barrier gradient descent (DBGD) framework~\citep{gong2021bi}, which constructs a descent direction $\mathbf{d}_k$ that reduces $f$ while ensuring sufficient descent on $g$. At each iteration, $\mathbf{d}_k$ solves the quadratic program
\begin{equation}\label{eq:QP}
    \mathbf{d}_k = \argmin_{\mathbf{d} \in \mathbb{R}^n} 
    \;\frac{1}{2}\|\nabla f(\mathbf{x}_k) - \mathbf{d}\|^2
    \quad \text{s.t.} \quad 
    \nabla g(\mathbf{x}_k)^\top \mathbf{d} \ge \phi(\mathbf{x}_k),
\end{equation}
where $\phi(\mathbf{x}_k) \geq 0$ controls the required alignment with $\nabla g(\mathbf{x}_k)$. Following~\citep{reisizadeh2026blur,gong2021bi,cao2025complexity}, we use $\phi(\mathbf{x}_k)=\beta_k\|\nabla g(\mathbf{x}_k)\|^2$. The equivalent saddle-point formulation is
\begin{align}\label{eq:saddle-point}
\min_{\mathbf d\in\mathbb R^n}~\max_{\lambda\geq 0}~\mathcal{L}_k(\mathbf d, \lambda) \triangleq \frac{1}{2}\|\nabla f(\mathbf x_k)-\mathbf d\|^2-\lambda \left(\nabla g(\mathbf x_k)^\top \mathbf d - \phi(\mathbf x_k) \right),
\end{align}%
which admits the closed-form solution $\mathbf{d}_k = \nabla f(\mathbf{x}_k) + \lambda_k \nabla g(\mathbf{x}_k)$ with dual multiplier
\begin{align}\label{eq:tild-lambda1}
    \lambda_k = \max\left\{
    \frac{\phi(\mathbf{x}_k) - \nabla f(\mathbf{x}_k)^\top \nabla g(\mathbf{x}_k)}
    {\|\nabla g(\mathbf{x}_k)\|^2}, \; 0
    \right\}.
\end{align}
The iterate is then updated as
\begin{align}\label{eq:x-update2}
    \mathbf{x}_{k+1} 
    &= \mathbf{x}_k - \eta_k \big(\nabla f(\mathbf{x}_k) + \lambda_k \nabla g(\mathbf{x}_k)\big).
\end{align}

\paragraph{Challenges in the stochastic setting.}
In the stochastic setting, the dual variable $\lambda_k$ in~\eqref{eq:tild-lambda1} involves ratios of stochastic quantities, and $\lambda_k = \mathcal{O}(\|\nabla g(\mathbf x_k)\|^{-1})$ near lower-level stationarity (Lemma~\ref{lem:lambda-ineq}). Solving the subproblem~\eqref{eq:saddle-point} inexactly via a stochastic primal-dual method~\citep{zhao2022accelerated,boob2023stochastic} amplifies this ill-conditioning through a multi-loop structure, resulting in significantly larger sample complexity (see Section~\ref{section:inexact}). We instead pursue a direct single-loop stochastic approximation of $\mathbf d_k$ and $\lambda_k$ via a perturbed dual reformulation.

\subsection{Proposed method: dual perturbation and stochastic approximation}
To stabilize the dual variable, we introduce a perturbation $\gamma_k \geq 0$ that regularizes the denominator in~\eqref{eq:tild-lambda1}, replacing $\|\nabla g_k\|^2$ with $\|\nabla g_k\|^2 + \gamma_k$. The perturbed dual problem is
\begin{align}\label{eq:argmax-lameps}
\lambda_{\gamma,k} \equiv 
\argmax_{\lambda \ge 0}\;
\lambda \beta_k \left(\|\nabla g(\mathbf{x}_k)\|^2 + \gamma_k\right)
- \lambda \nabla g(\mathbf{x}_k)^\top \nabla f(\mathbf{x}_k)
-\tfrac{1}{2}\lambda^2 \left(\|\nabla g(\mathbf{x}_k)\|^2 + \gamma_k\right).
\end{align}
The solution of~\eqref{eq:argmax-lameps} admits the closed form
\begin{align}\label{eq:lambda-eps1}
\lambda_{\gamma,k}
= \max\left\{
\beta_k - 
\frac{\nabla g(\mathbf{x}_k)^\top \nabla f(\mathbf{x}_k)}
{\|\nabla g(\mathbf{x}_k)\|^2 + \gamma_k},
0
\right\}.
\end{align}

Since exact gradients are unavailable, we construct a stochastic estimator by replacing the gradients in~\eqref{eq:lambda-eps1} with mini-batch samples:

\begin{align}\label{eq:tild-lambda-eps-2}
    \tilde\lambda_{\gamma,k} = \max\left\{\beta_k - \frac{\nabla \tilde g(\mathbf x_k,\xi_{g}^k)^\top \nabla\tilde f(\mathbf{x}_k,\xi_f^k)}{\|\nabla \tilde g(\mathbf x_k,\xi_{g}^k)\|^2+\gamma_k}, 0\right\}.
\end{align}
In \eqref{eq:tild-lambda-eps-2}, the same mini-batch stochastic gradient $\nabla \tilde g(\mathbf x_k, \xi_g^k)$ is used in both the numerator and denominator, which induces a dependency between these terms. This dependence is inherent to ratio-type estimators and implies that $\tilde \lambda_{\gamma,k}$ is generally biased. Rather than requiring unbiasedness of this estimator, our analysis directly focuses on the stochastic descent direction and establishes bounds on its bias and second moment, as developed in Section \ref{Sec:conv-anlys}.

The complete steps of the proposed method are summarized in Algorithm \ref{alg:alg1}.
\begin{algorithm}[t!]
\caption{Stochastic Dynamic Barrier Perturbed Gradient (SDBPG) Method}
\label{alg:alg1}
\begin{algorithmic}[1]
\State \textbf{Input:} step sizes $\{\eta_k\}_{k\ge 0}$, barrier parameters $\{\beta_k\}_{k\ge 0}$, perturbation sequence $\{\gamma_k\}_{k\ge 0}$
\State \textbf{Initialization:} $\mathbf{x}_0 \in \mathbb R^n$
\For{$k = 0,1,2,\dots,K-1$}
    \State Draw independent mini-batch samples $\xi_f^k,\xi_{g}^k$ with size $N_f$ and $N_g$, respectively. 
    
    \State Compute $\nabla \tilde{f}_k = \frac{1}{N_f}\sum_{i=1}^{N_f} \nabla \tilde f(\mathbf{x}_k, \xi_{f,i}^k)$, $\nabla \tilde{g}_k = \frac{1}{N_g}\sum_{i=1}^{N_g} \nabla \tilde g(\mathbf{x}_k, \xi_{g,i}^k)$. 

    \State $\tilde\lambda_{\gamma,k} \gets \max\big\{\beta_k - \tfrac{\nabla \tilde g_k^\top \nabla\tilde f_k}{\|\nabla \tilde g_k\|^2+\gamma_k}, 0\big\}$ 

    \State 
    $\tilde{\mathbf d}_k
    \gets
    \nabla \tilde f_k
    +
    \tilde{\lambda}_{\gamma,k}
    \nabla \tilde g_k$

    \State 
    $\mathbf{x}_{k+1}
    \gets
    \mathbf{x}_k
    -
    \eta_k \tilde{\mathbf d}_k$
\EndFor
\end{algorithmic}
\end{algorithm}

\subsection{Convergence analysis}\label{Sec:conv-anlys}
In this section, we analyze the convergence behavior of the proposed stochastic method. 
The main challenge is to understand how the stochastic quantities used in the algorithm approximate their deterministic counterparts and how this affects the descent direction.

A natural approach would be to compare the stochastic dual variable $\tilde{\lambda}_{\gamma,k}$ with its deterministic counterpart $\lambda_k$. 
However, this is not the most suitable object for analysis. 
Due to stochastic noise and the perturbation in the denominator, the difference 
$|\tilde{\lambda}_{\gamma,k} - \lambda_k|$ may not vanish and does not directly reflect the behavior of the update rule. 
Moreover, even in the deterministic setting, $\lambda_k$ can become large when $\|\nabla g(\mathbf x_k)\|$ is small, making such comparisons unstable.

Instead, we focus on the quantity that directly appears in the descent direction, namely $\lambda_{\gamma,k}\nabla g(\mathbf x_k)$. 
Although $\lambda_{\gamma,k}$ itself may be large, the product $\lambda_{\gamma,k}\nabla g(\mathbf x_k)$ remains well-behaved under our assumptions. 
Therefore, the key step in our analysis is to control the discrepancy between the deterministic quantity $\lambda_{\gamma,k}\nabla g(\mathbf x_k)$ and its stochastic approximation 
$\tilde{\lambda}_{\gamma,k}\nabla \tilde g(\mathbf x_k,\xi_{g}^k)$.

\paragraph{Error decomposition and role of perturbation.}
The stochastic descent direction can be written as $\tilde {\mathbf d}_k = \nabla \tilde f_k + \tilde{\lambda}_{\gamma,k}\nabla \tilde g_k$. We define the descent-direction error as $e_d \triangleq \tilde {\mathbf d}_k - \mathbf d_k$. This error captures both the stochasticity in the gradients and the approximation error in the dual variable. To understand the structure of the error, we decompose
\begin{align*}
    e_d 
    &=\big(\nabla \tilde f_k-\nabla f(\mathbf x_k)\big)
+
\big(
\tilde{\lambda}_{\gamma,k}\nabla \tilde g_k
-
\lambda_{\gamma,k}\nabla g(\mathbf x_k)
\big)
+
\left(\lambda_{\gamma,k}-\lambda_k\right)\nabla g(\mathbf x_k).
\end{align*}
The first two terms capture the stochastic approximation errors in the gradients and in the dual-gradient term, respectively. The last term arises from the perturbation $\gamma_k$ and introduces an additional deterministic error.
This decomposition explains the appearance of the term $\frac{\gamma_k \|\nabla f(\mathbf x_k)\|}{\|\nabla g(\mathbf x_k)\|^2 + \gamma_k}$ in the bias bound, and its squared counterpart in the second-moment bound of the stochastic descent direction in Lemma~\ref{lem:bias-ed-gradg}. This term originates from bounding the difference between $\lambda_{\gamma,k}$ and $\lambda_k$, and reflects the effect of perturbing the denominator. 
While the perturbation stabilizes the dual variable, it introduces a deterministic bias that must be controlled through an appropriate choice of $\gamma_k$.


\begin{lemma}[Stochastic approximation error]\label{lem:bias-ed-gradg}
Suppose Assumptions~\ref{assum:stoch_grad}-~\ref{assum:bounded-stoch-g} hold and let 
$e_d \triangleq \tilde{\mathbf d}_k - \mathbf d_k$.
Then, for any $k\geq 0$:
\begin{enumerate}[label=(\alph*)]
\item 
\textbf{Bias bound.} The conditional bias satisfies \label{lem:bias-part}
\begin{align}\label{eq:exp-ed1-gradg}
\norm*{\mathbb E_k\left[e_d\right]}
\leq \left(\beta_k+\frac{2G_f}{\sqrt{\gamma_k}}\right)\frac{\nu_g}{\sqrt{N_g}} + \frac{\nu_f}{\sqrt{N_f}} + \frac{
\gamma_k G_f}{\|\nabla g(\mathbf x_k)\|^2+\gamma_k}.
\end{align}
\item \textbf{Second-moment bound.} The conditional second moment satisfies \label{lem:second-moment-ed-gradg}
\begin{align}\label{eq:exp-ed2-gradg}
\mathbb E_k\left[\|e_d\|^2\right]\leq \frac{9\nu_f^2}{N_f} 
+6\left(\beta_k+\frac{2G_f}{\sqrt{\gamma_k}}\right)^2\frac{\nu_g^2}{N_g} 
+ \frac{3\gamma_k^2G_f^2
}{\left(\|\nabla g(\mathbf x_k)\|^2+\gamma_k\right)^2}.
\end{align}
\end{enumerate}
\end{lemma}

Leveraging the smoothness properties of the objective functions $f$ and $g$ stated in Assumption~\ref{assump:gradf-g-lip}, we establish a descent-type result. In particular, this result yields upper bounds on $\|\mathbf d_k\|^2$ and $\|\nabla g(\mathbf x_k)\|^2$ at every iteration, which are presented in the following.

Suppose Assumption~\ref{assump:gradf-g-lip} holds. Let the sequence
$\{\mathbf x_k\}_{k\geq 0}$ be generated by Algorithm \ref{alg:alg1}, using a fixed step-size $\eta_k \equiv \eta$ and a constant
parameter $\beta_k \equiv \beta$.
Let
$\Delta f_k \triangleq \mathbb E_k \left[f(\mathbf x_k)-f(\mathbf x_{k+1})\right]$ and
$\Delta g_k \triangleq \mathbb E_k \left[g(\mathbf x_k)-g(\mathbf x_{k+1})\right]$.
Then, 
\begin{align*}
    &\left(1-L_f \eta \right)\norm*{\mathbf d_k}^2 \leq \tfrac{1}{\eta} \Delta f_k+ \lambda_k \beta \norm*{\nabla g(\mathbf x_k)}^2 + G_f \norm*{\mathbb E_k[e_d]} + \tfrac{L_f\eta}{2} \big(\norm*{\mathbb E_k[e_d]}^2 + \mathbb E_k\big[\norm*{e_d}^2\big] \big),\\
    &\beta\norm*{\nabla g(\mathbf x_k)}^2 \leq \tfrac{2}{\eta}\Delta g_k +2L_g \eta \norm*{\mathbf d_k}^2 + (\tfrac{1}{\beta}+L_g\eta)\norm*{\mathbb E_k\left[e_d \right]}^2 + L_g\eta \mathbb E_k\big[\norm*{e_d}^2 \big].
\end{align*}

Together with the bounds on $e_d$, these inequalities link the stochastic error control to a sufficient decrease in the objectives. In particular, by combining these bounds with the previous lemmas, we now present the main convergence guarantee of the proposed algorithm.

\begin{proposition}\label{thm:theorem1}
    Assume that Assumptions~\ref{assump:gradf-g-lip}--\ref{assum:bounded-stoch-g} hold, and let the sequence $\{\mathbf{x}_k\}_{k\geq 0}$ be generated by Algorithm \ref{alg:alg1} using a constant step-size $\eta_k \equiv \eta=K^{-\frac{1}{4}}$, parameters $\beta_k \equiv \beta = K^{-\frac{1}{4}}$ and $\gamma_k=K^{-\frac{1}{2}}\|\nabla g(\mathbf x_k)\|^2$, and mini-batch sizes of $N_f=K$ and $N_g=K\gamma_k^{-1}$.
Let $\Delta_f \triangleq f(\mathbf{x}_0)-\inf f$ and $\Delta_g \triangleq g(\mathbf{x}_0)-g^*$, then
\begin{align*}
    &\frac{1}{K}\sum_{k=0}^{K-1}\mathbb E\left[\norm*{\mathbf d_{k}}^2\right] =\mathcal O\left(K^{-\frac{1}{2}}\right)
    ,\quad \frac{1}{K}\sum_{k=0}^{K-1}\mathbb E\left[\norm*{\nabla g(\mathbf x_{k})}^2\right] = \mathcal{O}\left(K^{-\frac{1}{2}}\right).
\end{align*}

\end{proposition}

While Proposition~\ref{thm:theorem1} establishes an $\mathcal{O}(K^{-1/2})$ convergence rate matching the deterministic counterpart in \cite{cao2025complexity}, it does not yield a clean total sample complexity bound. The root cause is that the parameter choices $\gamma_k = K^{-1/2}\|\nabla g(\mathbf{x}_k)\|^2$ and $N_g = K\gamma_k^{-1} = K^{3/2}\|\nabla g(\mathbf{x}_k)\|^{-2}$ depend on the \emph{true} gradient norm $\|\nabla g(\mathbf{x}_k)\|$, which is unavailable in the stochastic setting. This creates two fundamental difficulties. First, the algorithm \emph{cannot be implemented as stated}: the practitioner only has access to stochastic estimates $\nabla \tilde{g}(\mathbf{x}_k, \xi_g^k)$ and cannot set $N_g$ or $\gamma_k$ correctly. Second, the total lower-level cost $\sum_{k=0}^{K-1} N_g^{(k)} = K^{3/2}\sum_{k=0}^{K-1}\|\nabla g(\mathbf{x}_k)\|^{-2}$ is \emph{random and trajectory-dependent}---it can become arbitrarily large if the iterates visit near-stationary regions of $g$ where $\|\nabla g(\mathbf{x}_k)\|$ is small. This is not an artifact of the analysis: as the iterates approach a lower-level stationary point, the dual problem in~\eqref{eq:saddle-point} becomes degenerate and stochastic estimation of $\tilde\lambda_{\gamma,k}$ becomes increasingly unstable, requiring more samples to control the bias.

To obtain a concrete complexity bound, we need to control $\lambda_{\gamma,k}$, and its stochastic counterpart $\tilde\lambda_{\gamma,k}$, near lower-level stationary points. From~\eqref{eq:lambda-eps1}, $\lambda_{\gamma,k}$ may become large when $\|\nabla g(\mathbf x_k)\|$ is small, and the two gradients are negatively aligned, namely when $\nabla g(\mathbf x_k)^\top\nabla f(\mathbf x_k)<0$. Let us define such ``bad region'' by $\mathcal R_k(\tau,\theta)\triangleq \{\mathbf x_k\mid \norm{\nabla g(\mathbf x_k)}^2\le \tau,~\nabla g(\mathbf x_k)^\top\nabla f(\mathbf x_k)< -\theta\norm{\nabla g(\mathbf x_k)}^2\}$, for some parameters $\tau,\theta>0$, and its observable stochastic counterpart by $\widehat{\mathcal R}_k(\tau,\theta) \triangleq \{\mathbf x_k\mid \norm{\nabla \hat g_k}^2\le \tau,~\nabla \hat g_k^\top\nabla \hat f_k< -\theta\norm{\nabla \hat g_k}^2\}$ using some independent mini-batche gradients $\nabla \hat f_k,\nabla \hat g_k$. We modify the method by setting $\tilde\lambda_{k,\gamma}=0$ whenever
$\mathbf x_k\in \widehat{\mathcal R}_k(\tau/2,\theta)$. Thus, in the observable bad region, the update direction becomes $\tilde{\mathbf d}_k=\nabla \tilde f_k$, avoiding the unstable multiplier. The detector threshold is set strictly below $\tau$ so that a mismatch between the observable and the true region forces a gradient estimation error of at least $(1-\frac{1}{\sqrt 2})\sqrt{\tau}$, which is precisely what makes the frequency of such mismatches controllable through $N_g$. Next, we impose an assumption to control the frequency of visits to the bad region $\mathcal R_k(\tau,\theta)$, which is defined through the true gradients only. The mismatch events, where the detector is active although the iterate lies outside $\mathcal R_k(\tau,\theta)$, are not required to be rare and are instead treated as direction errors in the analysis.
\begin{assumption}
\label{assump:rare-visit}
Fix a horizon $K\ge 1$. There exists $\varsigma\in (0,1]$ such that 
$\sum_{k=0}^{K-1}\mathbb P\left(\mathbf x_k\in \mathcal{R}_k(\tau,\theta)\right)=\mathcal O(K^{1-\varsigma})$.
\end{assumption}
\begin{assumption}
\label{assump:stoch-bound}
There exist constants $0 \le G_f, C_g < \infty$ such that, for all $\mathbf{x} \in \mathbb{R}^n$,\\
$\mathbb E\left[\|\nabla \tilde f(\mathbf{x},\xi_f)\|\right] \le G_f$ and $\mathbb E\left[\|\nabla \tilde g(\mathbf{x},\xi_g)\| \right]\le C_g$.
\end{assumption}
\begin{remark}
Assumption~\ref{assump:rare-visit} is consistent with the modified update. When $\mathbf x_k\in \widehat{\mathcal R}_k(\tau/2,\theta)$, the method sets
$\tilde\lambda_{k,\gamma}=0$ and follows the stochastic upper-level gradient step. Because the defining condition of the bad region corresponds to negative alignment between the upper- and lower-level gradients, such steps are expected to increase the lower-level function and drive the iterates out of the bad region after a few iterations. The assumption only requires that the total expected number of visits to the bad region $\mathcal R_k(\tau,\theta)$ is sublinear in $K$, and it imposes no condition on the observable region $\widehat{\mathcal R}_k(\tau/2,\theta)$.
\end{remark}%

We emphasize that the perturbation in~\eqref{eq:tild-lambda-eps-2} is still necessary. Assumption~\ref{assump:rare-visit} controls the frequency of bad-region visits, but it does not uniformly lower bound the stochastic norm
$\|\nabla \tilde g(\mathbf x_k,\xi_g^k)\|$. Hence, without perturbation, $\tilde\lambda_{\gamma,k}$ may still become unstable even outside the true bad region.

These two assumptions resolve the difficulties identified in Proposition~\ref{thm:theorem1} as follows. On the majority of iterations, those where $\mathbf x_k\notin \mathcal{R}_k(\tau,\theta)$ and the detector is inactive, we can show that the perturbation error $|\lambda_{\gamma,k}-\lambda_k|\|\nabla g(\mathbf x_k)\|$ is controllable which depends on parameters $\gamma_k,\beta_k$ and thresholds $\tau,\theta$ (see Lemma \ref{lem:lambda-bad-region-bound}). On the iterations where $\mathbf x_k\notin \mathcal{R}_k(\tau,\theta)$ while the detector is active, either the true lower-level gradient is also small, in which case the induced error is bounded by the same quantities up to an absolute constant, or the detector reports a false alarm, whose probability is at most $\mathcal O(\nu_g^2N_g^{-1}\tau^{-1})$ by Chebyshev inequality. Note that only the gradient norm part of the detector enters this argument, hence no relation between the alignment thresholds of $\widehat{\mathcal R}_k(\tau/2,\theta)$ and $\mathcal R_k(\tau,\theta)$ is needed, while the alignment condition itself prevents the multiplier from being zeroed at points where it is already well behaved. On the remaining rare iterations where $\mathbf{x}_k \in \mathcal{R}_k(\tau,\theta)$, the error can be large but Assumption~\ref{assump:stoch-bound} gives a uniform bound of order
$\mathcal O(\beta_k+G_f)$, while Assumption~\ref{assump:rare-visit} makes their accumulated contribution sublinear. This four-case argument, detailed in Appendix~\ref{Appndx10}, allows fixed choices of $\gamma_k$ and $N_g$ that are independent of
$\|\nabla g(\mathbf x_k)\|^{-1}$ with explicit complexity guarantee.

\begin{theorem}\label{thm:theorem2}
    Suppose Assumptions~\ref{assump:gradf-g-lip}, \ref{assum:stoch_grad}, \ref{assump:rare-visit}, and \ref{assump:stoch-bound} hold, and let the sequence $\{\mathbf{x}_k\}_{k\geq 0}$ be generated by Algorithm \ref{alg:alg1} where line 6 is modified as $\tilde\lambda_{\gamma,k}\gets 0$ if $\mathbf x_k\in \widehat{\mathcal R}_k(\tau/2,\theta)$, and $\tilde\lambda_{\gamma,k}\gets \eqref{eq:tild-lambda-eps-2}$ otherwise, using a constant step-size $\eta_k \equiv \eta=K^{-\frac{1}{4}}$, parameters $\theta=K^{-\frac{1}{4}}$, $\tau=K^{-\frac{1}{2}}$, $\beta_k \equiv \beta = K^{-\frac{1}{4}}$ and $\gamma_k\equiv \gamma=K^{-1}$, and mini-batch sizes of $N_f=K$ and $N_g=K^{2}$.
Let $\Delta_f \triangleq f(\mathbf{x}_0)-\inf f$ and $\Delta_g \triangleq g(\mathbf{x}_0)-g^*$, then for any $K\geq 1$, there
exists an index $k^*\in\{0,\hdots,K-1\}$ such that
\begin{align*}
    &\mathbb E\big[\norm*{\mathbf d_{k^*}}^2\big] \leq \mathcal{O}\big(K^{-\frac{1}{2}} +K^{\frac{1}{2}-\varsigma}\big),\quad \mathbb E\big[\norm*{\nabla g(\mathbf x_{k^*})}^2\big] \leq \mathcal{O}\big(K^{-\frac{1}{2}} +K^{\frac{1}{2}-\varsigma} \big),
\end{align*}
Moreover, let $\varsigma\in (\frac{1}{2},1]$ and $\delta\triangleq \varsigma-\frac{1}{2}$, then to achieve $(\epsilon,\epsilon)$-stationary solution, the sample gradient complexities of upper-level and lower-level objective functions are $\mathcal O(\epsilon^{-2/\delta})$ and $\mathcal O(\epsilon^{-3/\delta})$, respectively.
In particular, if $\mathcal R_k(\tau,\theta)$ is empty, then Assumption \ref{assump:rare-visit} holds with $\varsigma=1$, so $\delta=\frac{1}{2}$ and the sample complexities are $\mathcal O(\epsilon^{-4})$ and $\mathcal O(\epsilon^{-6})$, respectively, otherwise the rate depends on $\varsigma$.
\end{theorem}

\begin{remark}\label{rem:quadratic}

Assumption~\ref{assump:rare-visit} holds with $\varsigma=1$ for a natural class of problems in which the bad region is empty. Consider the quadratic lower-level objective $g(\mathbf y)=\frac{1}{2}\norm{A\mathbf y-\mathbf b}^2$ and assume that $\nabla f(\mathbf y)$ belongs to the kernel of $A$ for every $\mathbf y$ satisfying $\norm{\nabla g(\mathbf y)}^2\leq \tau$. Since $\nabla g(\mathbf y)=A^\top (A\mathbf y-\mathbf b)$ and $A\nabla f(\mathbf y)=0$, we obtain $\nabla g(\mathbf y)^\top \nabla f(\mathbf y)=(A\mathbf y-\mathbf b)^\top A\nabla f(\mathbf y)=0$. Hence, at every point satisfying $\norm{\nabla g(\mathbf y)}^2\leq \tau$, membership in $\mathcal R_k(\tau,\theta)$ would require $0<-\theta\norm{\nabla g(\mathbf y)}^2$, which is impossible. Therefore, $\mathcal R_k(\tau,\theta)=\emptyset$ for any $k\geq 0$ and any $\theta>0$, and Theorem~\ref{thm:theorem2} yields the rate $\mathcal O(K^{-\frac{1}{2}})$ together with the sample complexities $\mathcal O(\epsilon^{-4})$ and $\mathcal O(\epsilon^{-6})$. Moreover, in this example the uniform bound on $\nabla g$ in Assumption~\ref{assump:stoch-bound} is not restrictive, since the analysis only uses this bound along the generated trajectory. Indeed, $g$ satisfies the Polyak Lojasiewicz condition, and combining the descent inequality of $g$ with the parameter selection of Theorem~\ref{thm:theorem2} gives $\sup_{k\leq K}\mathbb E\left[g(\mathbf x_k)-g^*\right]=\mathcal O(1)$ uniformly in $K$, which in turn implies $\mathbb E\left[\norm{\nabla \tilde g(\mathbf x_k,\xi_g)}\right]\leq \sqrt{2L_g\sup_{k\leq K}\mathbb E\left[g(\mathbf x_k)-g^*\right]}+\nu_g\triangleq C_g<\infty$. Hence, the only global requirement in this example is the boundedness of $\nabla f$.

\end{remark}


\section{Penalty-regularized SDBPG}\label{sec:pr-sdbpg}

The convergence guarantee in Theorem~\ref{thm:theorem2} relies on Assumption~\ref{assump:rare-visit}, which limits visits to the unstable bad region where the lower-level gradient is small and the gradients of the two levels are negatively aligned. Specifically, it requires the cumulative probability of visiting the bad region  to grow slower than the horizon $K$. Although weaker than assuming $\|\nabla g(\mathbf{x}_k)\|$ is uniformly bounded away from zero, this condition still affects the convergence rate through the average bad-region visit frequency, whose value may not be available a priori. In this section, we present an alternative approach that eliminates this assumption entirely by redesigning the subproblem.


\subsection{Regularized subproblem}

Rather than regularizing the denominator of the dual variable post-hoc via $\gamma_k$, we replace the constrained QP~\eqref{eq:QP} with a \emph{penalty-regularized} unconstrained objective:
\begin{align}\label{eq:reg-QP}
    \mathbf{d}_k = \argmin_{\mathbf{d} \in \mathbb{R}^n} \frac{1}{2}\|\mathbf{d} - \nabla f(\mathbf{x}_k)\|^2 + \frac{\mu}{2}  \frac{\left[\beta(\|\nabla g(\mathbf{x}_k)\|^2+\gamma)-\nabla g(\mathbf{x}_k)^\top \mathbf{d} \right]_+^2}{\|\nabla g(\mathbf{x}_k)\|^2 + \gamma},
\end{align}
where $\mu > 0$ is the penalty weight and $\gamma > 0$ is a \emph{fixed} regularization parameter. The penalty term penalizes the \emph{relative} violation of the DBGD constraint $\nabla g(\mathbf{x}_k)^\top \mathbf{d} \geq \beta\|\nabla g(\mathbf{x}_k)\|^2$, with the normalization by $\|\nabla g(\mathbf{x}_k)\|^2 + \gamma$ ensuring scale-invariance across all gradient magnitudes. Setting the gradient of~\eqref{eq:reg-QP} to zero yields a closed-form solution 
\begin{align}\label{eq:d-reg}
    \mathbf{d}_k = \nabla f(\mathbf{x}_k) + \hat\lambda_k \nabla g(\mathbf{x}_k), \qquad \hat\lambda_k = \left[\frac{\mu\left(\beta(\|\nabla g(\mathbf{x}_k)\|^2 + \gamma) - \nabla g(\mathbf{x}_k)^\top \nabla f(\mathbf{x}_k)\right)}{(1+\mu)\|\nabla g(\mathbf{x}_k)\|^2 + \gamma}\right]_+.
\end{align}

The critical structural difference from~\eqref{eq:tild-lambda1} is that the denominator $(1+\mu)\|\nabla g(\mathbf{x}_k)\|^2 + \gamma$ is \emph{always} bounded below by $\gamma > 0$, regardless of $\|\nabla g(\mathbf{x}_k)\|$. This eliminates the degeneracy that necessitated the rare-visit assumption. In particular, we first obtain a bound for $\hat{\lambda}_k$ uniform in $\|\nabla g(\mathbf x_k)\|$. 


\begin{lemma}[Uniform boundedness of the multiplier]\label{lem:pr-lambda-bounded}
Under Assumption~\ref{assum:bounded-stoch-g}, the regularized multiplier satisfies $0\leq \hat{\lambda}_k \leq C_\Lambda \triangleq \mu\beta + \frac{\mu G_f}{2\sqrt{(1+\mu)\gamma}}$ for all $k\geq 0$. 
\end{lemma}

\begin{remark}[Recovery of DBGD behavior]
When $\|\nabla g(\mathbf x_k)\|^2\gg\gamma$, the clipped regularized multiplier satisfies $\hat\lambda_k\approx \frac{\mu}{1+\mu}
    \left[
    \beta-\frac{\nabla g(\mathbf x_k)^\top\nabla f(\mathbf x_k)}
    {\|\nabla g(\mathbf x_k)\|^2}
    \right]_+$, which is a damped version of the DBGD multiplier in~\eqref{eq:tild-lambda1}.
When $\nabla g(\mathbf x_k)=0$, the direction is still
$\mathbf d_k=\nabla f(\mathbf x_k)$ because the correction term
$\hat\lambda_k\nabla g(\mathbf x_k)$ vanishes, while the multiplier remains
finite.
\end{remark}

\begin{algorithm}[t!]
\caption{Penalty-Regularized Stochastic DBGD (PR-SDBPG)}
\label{alg:pr-sdbpg}
\begin{algorithmic}[1]
\State \textbf{Input:} step size $\eta$, barrier $\beta$, penalty $\mu$, regularization $\gamma$, batch sizes $N_f, N_g$
\State \textbf{Initialize:} $\mathbf{x}_0 \in \mathbb{R}^n$
\For{$k = 0, 1, \ldots, K-1$}
    \State Draw independent mini-batches $\xi_f^k, \xi_g^k$ of sizes $N_f$ and $N_g$
    \State Compute $\nabla \tilde{f}_k = \frac{1}{N_f}\sum_{i=1}^{N_f} \nabla \tilde f(\mathbf{x}_k, \xi_{f,i}^k)$, $\nabla \tilde{g}_k = \frac{1}{N_g}\sum_{i=1}^{N_g} \nabla \tilde g(\mathbf{x}_k, \xi_{g,i}^k)$
    \State $\tilde\lambda_k \gets \left[\frac{\mu\left(\beta(\|\nabla \tilde{g}_k\|^2 + \gamma) - \nabla \tilde{g}_k^\top \nabla \tilde{f}_k\right)}{(1+\mu)\|\nabla \tilde{g}_k\|^2 + \gamma}\right]_+$
    \State $\tilde{\mathbf{d}}_k \gets \nabla \tilde{f}_k + \tilde\lambda_k\, \nabla \tilde{g}_k$
    \State $\mathbf{x}_{k+1} \gets \mathbf{x}_k - \eta\,\tilde{\mathbf{d}}_k$
\EndFor
\end{algorithmic}
\end{algorithm}

The stochastic algorithm replaces the exact gradients in~\eqref{eq:d-reg} with mini-batch estimates and is summarized in Algorithm~\ref{alg:pr-sdbpg}. It has the same single-loop structure as Algorithm~\ref{alg:alg1}, with two key differences: (i) the denominator $(1+\mu)\|\nabla \tilde{g}_k\|^2 + \gamma$ is always at least $\gamma > 0$, while in Algorithm~\ref{alg:alg1} the denominator $\|\nabla \tilde{g}_k\|^2 + \gamma_k$ can become arbitrarily small based on the selection $\gamma_k$  in Proposition~\ref{thm:theorem1} which depends on $\|\nabla g(\mathbf x_k)\|^2$; (ii) the batch sizes $N_f, N_g$ are \emph{fixed constants} set a priori while Algorithm~\ref{alg:alg1} requires choosing $\gamma_k$ and $N_g$ as functions of $\|\nabla g(\mathbf{x}_k)\|$ as in Proposition~\ref{thm:theorem1}.

\subsection{Convergence results}

\begin{theorem}[Convergence of PR-SDBPG]\label{thm:pr-sdbpg}
Under Assumptions~\ref{assump:gradf-g-lip}--\ref{assum:bounded-stoch-g} and~\ref{assump:stoch-bound}, let Algorithm~\ref{alg:pr-sdbpg} run with parameters
$\mu = 1$, $\beta = 4G_f^3/\epsilon$, $\gamma = 1$, $\eta = \epsilon/(5G_f^3(L_f + cL_g))$,
where $c = C_\Lambda + 1 = \mathcal{O}(G_f^3/\epsilon)$, and batch sizes
    $N_f = \mathcal{O}\!\left(\frac{\nu_f^2 G_f^6 C_g^2}{\epsilon^4}\right)$, $N_g = \mathcal{O}\!\left(\frac{G_f^{12}\nu_g^2 C_g^2}{\epsilon^6}\right)$. 
Then after $K = \mathcal{O}(G_f^9 L_g / \epsilon^4)$ iterations, there exists an index $k^*\in\{0,\hdots,K-1\}$ such that
\begin{align}\label{eq:pr-conv}
    \mathbb{E}\big[\|\mathbf{d}_{k^*}\|^2\big] \leq \mathcal{O}(\epsilon), \qquad \mathbb{E}\big[\|\nabla g(\mathbf{x}_{k^*})\|^2\big] \leq \mathcal{O}(\epsilon).
\end{align}
Moreover, the upper- and lower-level sample complexities are $\mathcal{O}(\epsilon^{-8})$ and $\mathcal{O}(\epsilon^{-10})$, respectively. 
\end{theorem}

\begin{proof}[Proof sketch]
The proof uses a Lyapunov function $\Phi_k = f(\mathbf{x}_k) + c \: g(\mathbf{x}_k)$ and proceeds in two stages. First, we show that the Lyapunov descent $\Phi_k - \Phi_{k+1}$ controls $\|\mathbf{d}_k\|^2$ up to a neighborhood $\delta^* = \mathcal{O}(\epsilon)$. The key insight is that when $\nabla g(\mathbf{x}_k)^\top \mathbf{d}_k < 0$ (the problematic case), the multiplier $\hat\lambda_k$ is provably large ($> \mu\beta$), making the cross-term $(c - \hat\lambda_k)\nabla g(\mathbf{x}_k)^\top \mathbf{d}_k$ small. Second, a separate $g$-smoothness telescope gives $\|\nabla g(\mathbf{x}_k)\|^2 \leq \mathcal{O}(\epsilon)$. The stochastic error is controlled by the bias and second-moment bounds of $e_d = \tilde{\mathbf{d}}_k - \mathbf{d}_k$, which are uniform in $\|\nabla g(\mathbf{x}_k)\|$ because the denominator of the regularized multiplier is always $\geq \gamma$. Full details are in Appendix~\ref{sec:pr-sdbpg-proof}.
\end{proof}

\begin{remark}[Comparison with Theorem~\ref{thm:theorem2}]
Theorem~\ref{thm:pr-sdbpg} achieves $(\epsilon,\epsilon)$-stationarity \emph{without} the rare-visit Assumption~\ref{assump:rare-visit}, at the cost of a possibly larger sample complexities: $\mathcal{O}(\epsilon^{-8})$ and $\mathcal{O}(\epsilon^{-10})$ vs.\ $\mathcal{O}(\epsilon^{-4})$ and $\mathcal{O}(\epsilon^{-6})$ in Theorem~\ref{thm:theorem2} when $\delta=\tfrac12$. The bottleneck is the Lyapunov weight amplification: ensuring $\delta^* \leq \epsilon$ requires $\beta = \Omega(G_f^3/\epsilon)$, which forces $c = \Omega(G_f^3/\epsilon)$ and the bias requirement $G_\Phi B_{\max} \leq \epsilon$ demands batch sizes $N_g = \mathcal{O}(\epsilon^{-6})$. In exchange, PR-SDBPG uses fixed, deterministic batch sizes without any additional assumption. 
\end{remark}

To address the large mini-batches in PR-SDBPG, we consider implementing STORM-type variance reduction~\citep{cutkosky2019momentum} estimators to track the stochastic gradients. This preserves the convergence guarantee without Assumption~\ref{assump:rare-visit} and improves both sample complexities by factor of \(\epsilon^{-2}\) under an additional mean-square smoothness condition. The algorithm details and proofs are deferred to Appendix~\ref{app:vr_pr_sdbpg}.

\begin{theorem}[Variance-reduced PR-SDBPG]\label{rem:vr-pr-sdbpg}
Suppose the assumptions of Theorem~\ref{thm:pr-sdbpg} hold along with the mean-square smoothness condition in Assumption~\ref{ass:ms_smooth}.
Applying STORM-type variance reduction~\citep{cutkosky2019momentum} to Algorithm~\ref{alg:pr-sdbpg}, replacing i.i.d.\ gradient estimates with recursive momentum-based trackers with the same PR-SDBPG parameters, yields a $(\mathcal{O}(\epsilon), \mathcal{O}(\epsilon))$-stationary point with total sample complexities $\mathcal{O}(\epsilon^{-6})$ and $\mathcal{O}(\epsilon^{-8})$ for the upper and lower-level objectives, respectively.
\end{theorem}

\vspace*{-2mm}
\section{Numerical experiment}\label{Sec:numeric-exp}
\vspace*{-2mm}

We evaluate SDBPG, PR-SDBPG, and VR-PR-SDBPG against the deterministic counterpart (DDBPG) on LLM unlearning task~\cite{sekhari2021remember,liu2024towards,reisizadeh2026blur}, which seeks to remove undesirable data (e.g., sensitive, biased, or private information) from a pre-trained model while preserving its overall utility. The task involves two objectives: a \emph{forget} objective for removing targeted information and a \emph{retain} objective for maintaining performance on the remaining data. The bilevel formulation~\citep{reisizadeh2026blur} is
\begin{align*}
    \min_{\theta \in \Theta} \quad & f(\theta) := \mathbb{E}_{\xi_f}[\ell_{\mathrm{ret}}(y \mid x; \theta)]
    \quad \text{s.t.} \quad 
    \Theta = \arg\min_{\theta} \; g(\theta) := \mathbb{E}_{\xi_g}[\ell_{\mathrm{for}}(y \mid x; \theta)],
\end{align*}
where $\xi_f$ and $\xi_g$ are sampled from the retain and forget datasets, respectively. Here $\ell_{\mathrm{ret}}$ is the cross-entropy loss and $\ell_{\mathrm{for}}$ is the negative preference optimization (NPO) loss. Both objectives are nonconvex due to the neural network parameterization.

We conduct experiments on two benchmark datasets (TOFU and MUSE-News) using instruction-tuned LLMs. Stochastic methods use mini-batch gradients, whereas the deterministic variant uses fixed, large, or full batches. Full experimental details are deferred to Appendix~\ref{app:exp_details}.

We evaluate performance using the normalized metrics
$\frac{1}{K}\sum_{k=0}^{K-1}\|\nabla \tilde g(\mathbf{x}_k, \xi_g^k)\|^2$ and
$\frac{1}{K}\sum_{k=0}^{K-1}\|\tilde{\mathbf d}_k\|^2$,
each normalized by dividing them by their initial value, and plotted versus number of samples, wall-clock time, and iterations.

Figures~\ref{fig:tofu} show that the stochastic methods outperform DDBPG in terms of sample efficiency, achieving faster decay of both stationarity metrics across datasets. Among them, SDBPG exhibits the fastest convergence, followed by VR-PR-SDBPG and PR-SDBPG, aligning with the theoretical guarantees. When measured against wall-clock time, stochastic methods achieve lower stationarity measures with comparable computational budgets, demonstrating practical efficiency and stability.

\begin{figure}[t]
\centering

\begin{subfigure}{0.24\linewidth}
    \centering
    \includegraphics[width=\linewidth]{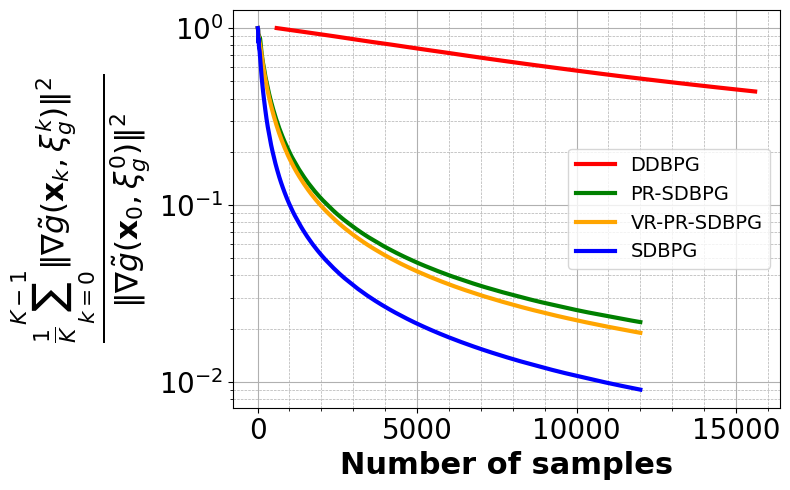}
\end{subfigure}
\hfill
\begin{subfigure}{0.24\linewidth}
    \centering
    \includegraphics[width=\linewidth]{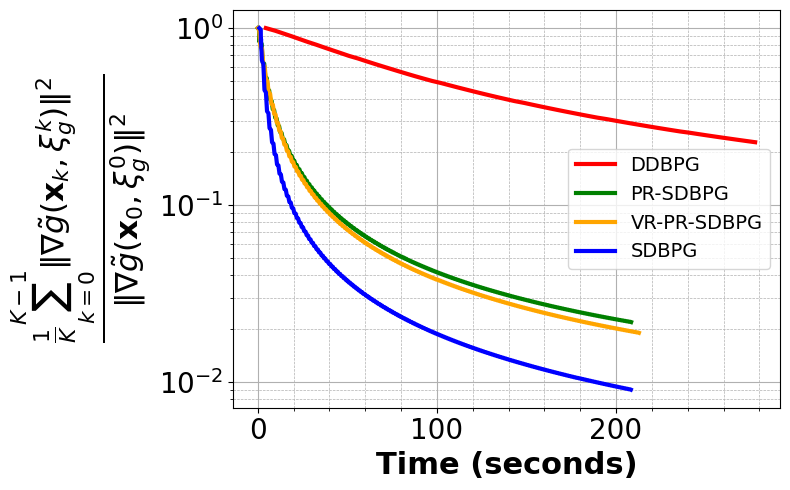}
\end{subfigure}
\begin{subfigure}{0.24\linewidth}
    \centering
    \includegraphics[width=\linewidth]{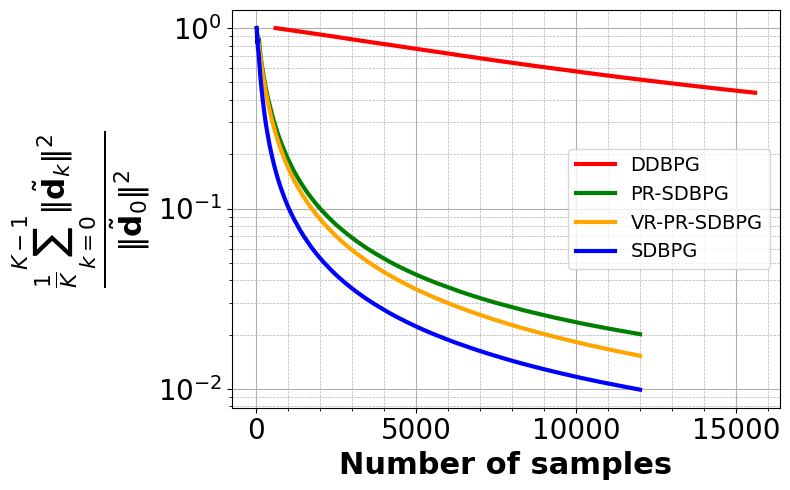}
\end{subfigure}
\hfill
\begin{subfigure}{0.24\linewidth}
    \centering
    \includegraphics[width=\linewidth]{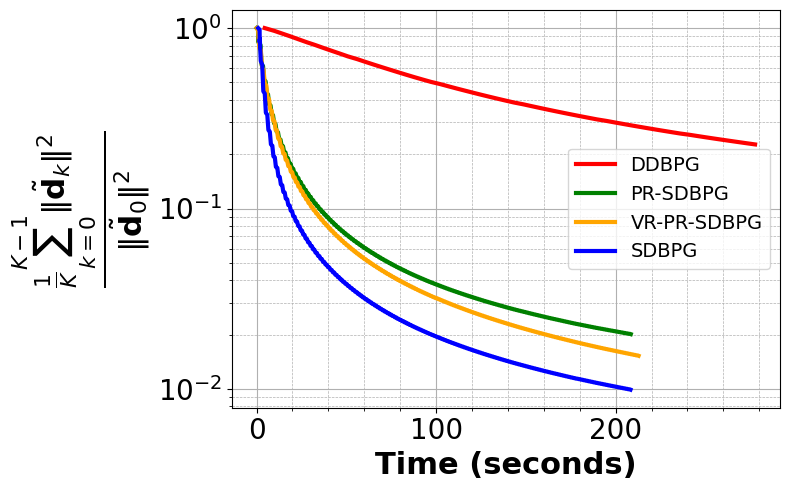}
\end{subfigure}
\begin{subfigure}{0.24\linewidth}
    \centering
    \includegraphics[width=\linewidth]{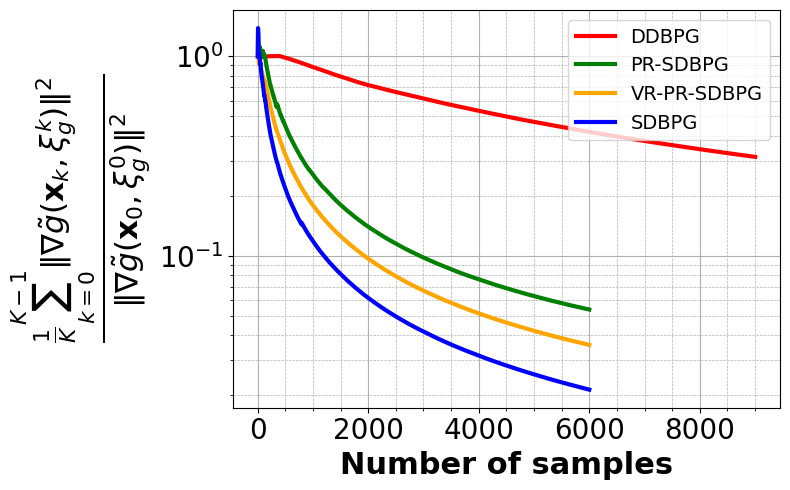}
\end{subfigure}
\hfill
\begin{subfigure}{0.24\linewidth}
    \centering
    \includegraphics[width=\linewidth]{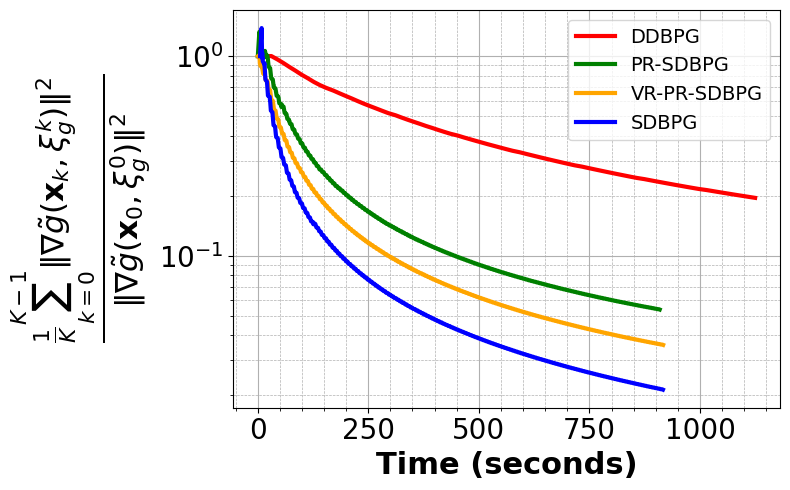}
\end{subfigure}
\begin{subfigure}{0.24\linewidth}
    \centering
    \includegraphics[width=\linewidth]{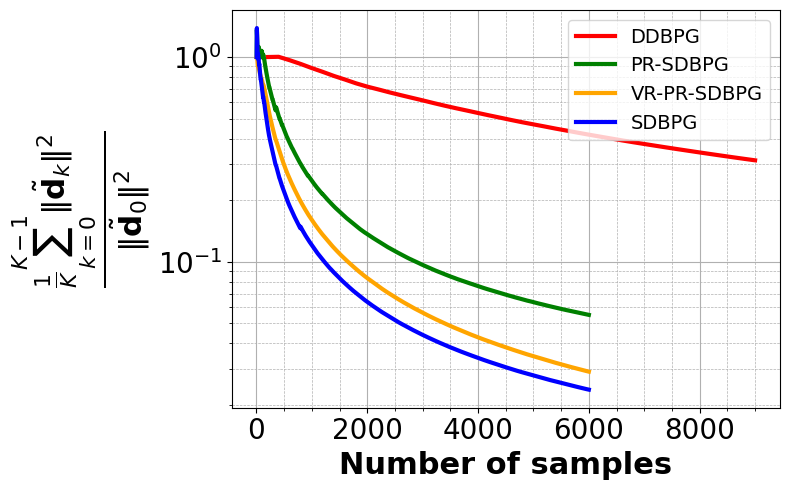}
\end{subfigure}
\hfill
\begin{subfigure}{0.24\linewidth}
    \centering
    \includegraphics[width=\linewidth]{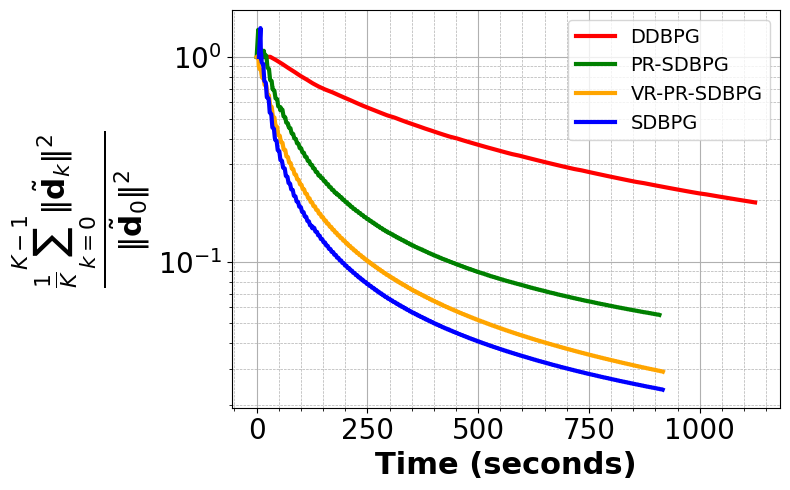}
\end{subfigure}
\caption{
Comparison of SDBPG (blue), PR-SDBPG (green), and VR-PR-SDBPG (orange) with the deterministic counterpart DDBPG (red) versus the number of sample gradients and wall-clock time (seconds). First row: TOFU dataset, second row: MUSE-News dataset.
}
\label{fig:tofu}
\end{figure}



\vspace*{-2mm}
\section{Conclusion}
\vspace*{-2mm}

In this work, we proposed SDBPG for stochastic nonconvex simple bilevel optimization, showing that dual perturbation stabilizes the stochastic DBGD direction. Under a rare-visit assumption with parameter $\delta = \varsigma - \tfrac{1}{2} \in (0, \tfrac{1}{2}]$, it yields $(\epsilon,\epsilon)$-stationarity with sample complexities $\mathcal{O}(\epsilon^{-2/\delta})$ and $\mathcal{O}(\epsilon^{-3/\delta})$ for the upper and lower levels, respectively, recovering $\mathcal{O}(\epsilon^{-4})$ and $\mathcal{O}(\epsilon^{-6})$ when the cumulative expected number of bad-region visits is bounded uniformly in $K$. Furthermore, we developed PR-SDBPG and VR-PR-SDBPG, which remove the rare-visit assumption with complexities
$\mathcal O(\epsilon^{-8}),\mathcal O(\epsilon^{-10})$ and
$\mathcal O(\epsilon^{-6}),\mathcal O(\epsilon^{-8})$, respectively.

\medskip

{
\small
\bibliographystyle{unsrtnat}
\bibliography{reference}

}

\newpage

\appendix

\section*{Technical Appendix and Supplementary Material}

\section{Auxiliary Lemmas}\label{AxilAppndx1}
\begin{lemma}\label{lem:lambda_eps_bound}
Let $\lambda_k$ and $\lambda_{\gamma,k}$ be defined as in
\eqref{eq:tild-lambda1} and \eqref{eq:lambda-eps1}, respectively.
Then, for any $\gamma_k \ge 0$, the following bound holds:
\begin{align}\label{eq:leps-lstar}
    \big|\lambda_{\gamma,k} - \lambda_k\big|
    \;\le\;
    \frac{\|\nabla f(\mathbf x_k)\|\,\gamma_k}
    {\|\nabla g(\mathbf x_k)\|\big(\|\nabla g(\mathbf x_k)\|^2 + \gamma_k\big)}.
\end{align}
\end{lemma}
\begin{proof}
By definitions \eqref{eq:tild-lambda1} and \eqref{eq:lambda-eps1},
\[
\lambda_k
= \max\!\left\{
\beta_k - \frac{\nabla g(\mathbf x_k)^\top \nabla f(\mathbf x_k)}{\|\nabla g(\mathbf x_k)\|^2},\,0
\right\},
\qquad
\lambda_{\gamma,k}
= \max\!\left\{
\beta_k - \frac{\nabla g(\mathbf x_k)^\top \nabla f(\mathbf x_k)}
{\|\nabla g(\mathbf x_k)\|^2 + \gamma_k},\,0
\right\}.
\]
Using the nonexpansiveness of the projection onto $\mathbb{R}_+$, i.e.,
$|\max\{a,0\}-\max\{b,0\}|\le |a-b|$, we obtain
\[
\big|\lambda_{\gamma,k} - \lambda_k\big|
\le
\left| \frac{\nabla g(\mathbf x_k)^\top \nabla f(\mathbf x_k)}
{\|\nabla g(\mathbf x_k)\|^2 + \gamma_k} -
\frac{\nabla g(\mathbf x_k)^\top \nabla f(\mathbf x_k)}{\|\nabla g(\mathbf x_k)\|^2}
\right|.
\]
Rewriting the difference yields
\[
\left|
\frac{\nabla g(\mathbf x_k)^\top \nabla f(\mathbf x_k)\,\gamma_k}
{\|\nabla g(\mathbf x_k)\|^2(\|\nabla g(\mathbf x_k)\|^2 + \gamma_k)}
\right|.
\]
Applying the Cauchy--Schwarz inequality,
$|\nabla g(\mathbf x_k)^\top \nabla f(\mathbf x_k)|
\le \|\nabla g(\mathbf x_k)\|\,\|\nabla f(\mathbf x_k)\|$,
we arrive at \eqref{eq:leps-lstar}.
\end{proof}

\begin{lemma}\label{lem:lambda-ineq}
    Consider the expressions of $\lambda_k$ and $\lambda_{\gamma,k}$ in~\eqref{eq:tild-lambda1} and \eqref{eq:lambda-eps1}, respectively. If Assumption~\ref{assum:bounded-stoch-g} is satisfied, then
    \begin{align}\label{eq:lambda1}
        \lambda_k \leq \beta + \frac{G_f}{\|\nabla g(\mathbf x_k)\|}.
    \end{align}
    \begin{align}\label{eq:lambda1-eps}
        \lambda_{\gamma,k} &\leq \beta + \frac{G_f\|\nabla g(\mathbf x_k)\|}{\|\nabla g(\mathbf x_k)\|^2+\gamma_k} . 
    \end{align}
\end{lemma}
\begin{proof}
Using Assumption \ref{assum:bounded-stoch-g} and \eqref{eq:tild-lambda1}, we can show

\begin{align*}
    \lambda_k &\leq \beta + \frac{|\langle\nabla f(\mathbf x_k),\nabla g(\mathbf x_k) \rangle|}{\|\nabla g(\mathbf x_k)\|^2}\nonumber\\
    & \leq \beta + \frac{G_f}{\|\nabla g(\mathbf x_k)\|}. 
\end{align*}
Moreover, using Assumption \ref{assum:bounded-stoch-g} and \eqref{eq:lambda-eps1}, we can show

\begin{align*}
    \lambda_{\gamma,k} &\leq \beta + \frac{|\langle\nabla f(\mathbf x_k),\nabla g(\mathbf x_k) \rangle|}{\|\nabla g(\mathbf x_k)\|^2 + \gamma_k}\nonumber\\
    & \leq \beta + \frac{G_f\|\nabla g(\mathbf x_k)\|}{\|\nabla g(\mathbf x_k)\|^2+\gamma_k}. 
\end{align*}
\end{proof}

\begin{lemma}\label{lem:deltaf-g1}
   Assume that Assumption~\ref{assump:gradf-g-lip} is satisfied, and let the sequence
$\{\mathbf x_k\}$ be generated by Algorithm \ref{alg:alg1}, using a fixed step-size $\eta_k \equiv \eta$ and a constant
parameter $\beta_k \equiv \beta$.
Let
$\Delta f_k \triangleq \mathbb E_k \left[f(\mathbf x_k)-f(\mathbf x_{k+1})\right]$ and
$\Delta g_k \triangleq \mathbb E_k \left[g(\mathbf x_k)-g(\mathbf x_{k+1})\right]$.
Then, the following results hold.
\begin{align}
    &\left(1-L_f \eta \right)\norm*{\mathbf d_k}^2 \nonumber\\
    & \quad\leq \frac{\Delta f_k}{\eta} + \lambda_k \beta \norm*{\nabla g(\mathbf x_k)}^2 + \norm*{\nabla f(\mathbf x_k)} \norm*{\mathbb E_k\left[e_d \right]} + \frac{L_f}{2}\eta \left(\norm*{\mathbb E_k\left[e_d \right]}^2 + \mathbb E_k\left[\norm*{e_d}^2\right] \right),\label{eq: f-bound1} \\
    &\beta\norm*{\nabla g(\mathbf x_k)}^2 \leq \frac{2\Delta g_k}{\eta} +2L_g \eta \norm*{\mathbf d_k}^2 + (\frac{1}{\beta}+L_g\eta)\norm*{\mathbb E_k\left[e_d \right]}^2 + L_g\eta \mathbb E_k\left[\norm*{e_d}^2 \right]. \label{eq:g-bound1}
\end{align}
\end{lemma}
\begin{proof}
From Assumption \ref{assump:gradf-g-lip}, $f$ has an $L_f$-Lipschitz continuous gradient, hence,
\begin{align*}
    f(\mathbf{x}_{k+1})-f(\mathbf{x}_k) \leq -\eta \nabla f(\mathbf{x}_k)^\top \tilde{\mathbf{d}}_k + \frac{L_f}{2}\eta^2\|\tilde{\mathbf d}_k\|^2.
\end{align*}
By adding and subtracting $\mathbf d_k$, we have:
\begin{align*}
    f(\mathbf{x}_{k+1})-f(\mathbf{x}_k) &\leq -\eta \nabla f(\mathbf{x}_k)^\top \left( \tilde{\mathbf{d}}_k-\mathbf d_k + \mathbf d_k \right) + \frac{L_f}{2}\eta^2\|\tilde{\mathbf d}_k - \mathbf d_k + \mathbf d_k\|^2\\
    & = -\eta \nabla f(\mathbf x_k)^\top \mathbf d_k - \eta \nabla f(\mathbf x_k)^\top \left(\tilde{\mathbf d}_k-\mathbf d_k \right) \nonumber\\
    &\quad+ \frac{L_f}{2}\eta^2 \left(\|\tilde{\mathbf d}_k-\mathbf d_k\|^2 + \|\mathbf d_k\|^2 + 2(\tilde{\mathbf d}_k-\mathbf d_k)^\top \mathbf d_k \right).
\end{align*}
Taking expectation given $\mathbf x_k$ from both sides, using Cauchy–Schwarz inequality and defining $e_d \triangleq \tilde{\mathbf d}_k-\mathbf d_k$, we obtain:
\begin{align*}
    \mathbb E_k\left[f(\mathbf{x}_{k+1})-f(\mathbf{x}_k)\right]&\leq -\eta \nabla f(\mathbf{x}_k)^\top \mathbf d_k + \frac{L_f}{2}\eta^2\| \mathbf d_k\|^2  + \eta \|\nabla f(\mathbf x_k)\| \|\mathbb E_k\left[e_d\right]\| \nonumber\\
    &\quad+ \frac{L_f}{2}\eta^2(\mathbb E_k\left[\|e_d\|^2\right] + 2\|\mathbb E_k\left[e_d\right]\| \|\mathbf d_k\| )\\
    & = -\eta \left(\nabla f(\mathbf x_k) - \mathbf d_k \right)^\top \mathbf d_k - \eta \left(1-\frac{L_f}{2}\eta \right)\|\mathbf d_k\|^2 + \eta\|\nabla f(\mathbf x_k)\| \|\mathbb E_k\left[e_d\right]\| \nonumber\\
    &\quad+ \frac{L_f}{2}\eta^2 \left(\mathbb E_k\left[\|e_d\|^2\right]+2\|\mathbb E_k\left[e_d\right]\|\|\mathbf d_k\| \right)\\
    &= \eta \lambda_k \nabla g(\mathbf x_k)^\top \mathbf d_k - \eta \left(1-\frac{L_f}{2}\eta \right)\|\mathbf d_k\|^2  + \frac{L_f}{2}\eta^2\mathbb E_k\left[\|e_d\|^2\right]\nonumber\\
    &\quad + \left(\eta \|\nabla f(\mathbf x_k)\|+L_f\eta^2\|\mathbf d_k\| \right)\|\mathbb E_k\left[e_d\right]\|,
\end{align*}
where in the last equality we used $\nabla f(\mathbf x_k)=\mathbf d_k-\lambda_k\nabla g(\mathbf x_k)$. Since $\mathbf d_k$ is the optimal solution of subproblem \eqref{eq:QP}, with the corresponding optimal dual multiplier $\lambda_k$, the complementarity slackness implies that $\lambda_k(\nabla g(\mathbf x_k)^\top\mathbf d_k - \beta \|\nabla g(\mathbf x_k)\|^2)=0$. Hence, we can show:
\begin{align*}
    &\mathbb E_k\left[f(\mathbf{x}_{k+1})-f(\mathbf{x}_k)\right]\\
    &\quad\leq  - \eta \left(1-\frac{L_f}{2}\eta \right)\|\mathbf d_k\|^2 + \eta \lambda_k\beta\|\nabla g(\mathbf x_k)\|^2 + \left(\eta \|\nabla f(\mathbf x_k)\|+L_f\eta^2\|\mathbf d_k\| \right)\|\mathbb E_k\left[e_d\right]\|\nonumber\\
    &\qquad+ \frac{L_f}{2}\eta^2\mathbb E_k\left[\|e_d\|^2\right].
\end{align*}
By dividing both sides by $\eta$ and letting $\Delta f_k = \mathbb E_k\left[f(\mathbf x_k)-f(\mathbf x_{k+1})\right]$, and rearranging the inequality, we have
\begin{align*}
    &\left(1-\frac{L_f}{2}\eta \right)\|\mathbf d_k\|^2 \nonumber\\
    &\quad\leq \frac{\Delta f_k}{\eta} +  \lambda_k\beta\|\nabla g(\mathbf x_k)\|^2 + \left( \|\nabla f(\mathbf x_k)\|+L_f\eta\|\mathbf d_k\| \right)\|\mathbb E_k\left[e_d\right]\| + \frac{L_f}{2}\eta\mathbb E_k\left[\|e_d\|^2\right].
\end{align*}
Using Young's inequality and rearranging the terms, we obtain
\begin{align*}
    &\left(1-L_f \eta \right)\norm*{\mathbf d_k}^2 \nonumber\\
    &\quad\leq \frac{\Delta f_k}{\eta} + \lambda_k \beta \norm*{\nabla g(\mathbf x_k)}^2 + \norm*{\nabla f(\mathbf x_k)} \norm*{\mathbb E_k\left[e_d \right]} + \frac{L_f}{2}\eta \left(\norm*{\mathbb E_k\left[e_d \right]}^2 + \mathbb E_k\left[\norm*{e_d}^2\right] \right).
\end{align*}
Moreover, from Assumption \ref{assump:gradf-g-lip}, $g$ has an $L_g$-Lipschitz continuous gradient, hence,
\begin{align*}
    g(\mathbf{x}_{k+1})-g(\mathbf{x}_k) \leq -\eta \nabla g(\mathbf{x}_k)^\top \tilde{\mathbf{d}}_k + \frac{L_g}{2}\eta^2\|\tilde{\mathbf d}_k\|^2.
\end{align*}
By adding and subtracting $\mathbf d_k$, we have:
\begin{align*}
    &g(\mathbf{x}_{k+1})-g(\mathbf{x}_k) \nonumber\\
    &\quad\leq -\eta \nabla g(\mathbf x_k)^\top \left(\tilde{\mathbf d}_k - \mathbf d_k + \mathbf d_k \right)+ \frac{L_g}{2}\eta^2\|\tilde{\mathbf d}_k-\mathbf d_k + \mathbf d_k\|^2\\
    &\quad\leq -\eta \beta\|\nabla g(\mathbf x_k)\|^2  + \frac{L_g}{2}\eta^2 \left(\|\tilde{\mathbf d}_k - \mathbf d_k\|^2 + \|\mathbf d_k\|^2 + 2\left(\tilde{\mathbf d}_k - \mathbf d_k \right)^\top \mathbf d_k \right)\nonumber\\
    &\qquad -\eta \nabla g(\mathbf x_k)^\top \left(\tilde{\mathbf d}_k - \mathbf d_k \right),
\end{align*}
where in the last inequality, we use the fact that $-\eta\nabla g(\mathbf x_k)^\top \mathbf d_k \leq -\eta\beta \|\nabla g(\mathbf x_k)\|^2$.\\

Taking expectation given $\mathbf x_k$ from both sides, using Cauchy–Schwarz inequality, and $e_d \triangleq \tilde{\mathbf d}_k-\mathbf d_k$, we obtain
\begin{align*}
    &\mathbb E_k\left[g(\mathbf{x}_{k+1})-g(\mathbf{x}_k)\right]\nonumber\\
    &\quad\leq -\eta \beta\|\nabla g(\mathbf x_k)\|^2 + \frac{L_g}{2}\eta^2\|\mathbf d_k\|^2 + \eta \|\nabla g(\mathbf x_k)\| \|\mathbb E_k\left[e_d \right] \| + L_g\eta^2 \|\mathbf d_k\| \|\mathbb E_k\left[e_d\right]\| \nonumber\\
    &\qquad+ \frac{L_g}{2}\eta^2\mathbb E_k\left[\|e_d\|^2\right].
\end{align*}
By dividing both sides by $\eta$ and letting $\Delta g_k = \mathbb E_k\left[g(\mathbf x_k)-g(\mathbf x_{k+1})\right]$, and rearranging the inequality, we have
\begin{align*}
    \beta\|\nabla g(\mathbf x_k)\|^2 \leq \frac{\Delta g_k}{\eta} + \frac{L_g}{2}\eta\|\mathbf d_k\|^2 + \left( \|\nabla g(\mathbf x_k)\| + L_g\eta \|\mathbf d_k\| \right) \|\mathbb E_k\left[e_d\right]\| + \frac{L_g}{2}\eta\mathbb E_k\left[\|e_d\|^2\right].
\end{align*}
Using Young's inequality and rearranging the terms, we obtain
\begin{align*}
    \beta\norm*{\nabla g(\mathbf x_k)}^2 \leq \frac{\Delta g_k}{\eta} +L_g \eta \norm*{\mathbf d_k}^2 + \norm*{\nabla g(\mathbf x_k)} \norm*{\mathbb E_k \left[e_d \right]} + \frac{L_g \eta}{2}\norm*{\mathbb E_k\left[e_d \right]}^2 + \frac{L_g}{2}\eta \mathbb E_k\left[\norm*{e_d}^2 \right].
\end{align*}
\begin{align*}
    \frac{\beta}{2}\norm*{\nabla g(\mathbf x_k)}^2 \leq \frac{\Delta g_k}{\eta} +L_g \eta \norm*{\mathbf d_k}^2 + \left(\frac{1}{2\beta} + \frac{L_g \eta}{2}\right)\norm*{\mathbb E_k\left[e_d \right]}^2 + \frac{L_g}{2}\eta \mathbb E_k\left[\norm*{e_d}^2 \right],
\end{align*}
which multiplying by two leads to the desired result.
\end{proof}

\begin{lemma}\label{lem:imp-d2}
    Under the update scheme given in~\eqref{eq:x-update2}, suppose that Assumption~\ref{assum:bounded-stoch-g} is satisfied and $\eta \leq \frac{1}{4(L_f +2\beta L_g)}$ and $\beta>0$. Then,

\begin{align*}
    &\norm*{\mathbf d_k}^2 \leq \frac{4(\Delta f_k + 2\beta \Delta g_k)}{\eta} + 4G_f \norm*{\mathbb E_k \left[e_d \right]} + \left(4+4L_g \eta \beta + 2L_f \eta\right)\norm*{\mathbb E_k \left[ e_d\right]}^2 \nonumber\\
    &\quad  + \left(4L_g \eta \beta + 2L_f \eta\right) \mathbb E_k \left[\norm*{e_d}^2 \right] + 4G_f^2 \beta L_g \eta +\frac{2\Delta g_k}{L_g\eta^2}\nonumber\\
    &\quad + 4G_f \sqrt{\left(1+L_g\eta\beta \right)\norm*{\mathbb E_k\left[e_d \right]}^2 + L_g\eta\beta \mathbb E_k\left[\norm*{e_d}^2 \right]}.
\end{align*}
\end{lemma}
\begin{proof}


Combining \eqref{eq:lambda1} with \eqref{eq: f-bound1} and using Assumption \ref{assum:bounded-stoch-g}, we have

\begin{align*}
    &\left(1-L_f \eta \right)\norm*{\mathbf d_k}^2 \nonumber\\
    &\quad\leq \frac{\Delta f_k}{\eta} +  \beta^2\|\nabla g(\mathbf x_k)\|^2 + \beta G_f\|\nabla g(\mathbf x_k)\| + G_f \|\mathbb E_k[e_d]\| + \frac{L_f}{2}\eta \left(\norm*{\mathbb E_k\left[e_d \right]}^2 + \mathbb E_k\left[\norm*{e_d}^2 \right] \right).
\end{align*}
Substituting \eqref{eq:g-bound1}, we have
\begin{align*}
     &\left(1-L_f \eta \right)\norm*{\mathbf d_k}^2 \nonumber\\
    &\quad\leq \frac{\Delta f_k}{\eta} +  \Bigg(2\beta\frac{\Delta g_k}{\eta} + 2\beta L_g \eta \norm*{\mathbf d_k}^2 + \left(1+L_g\eta\beta \right)\norm*{\mathbb E_k\left[e_d \right]}^2 + L_g\eta\beta \mathbb E_k\left[\norm*{e_d}^2 \right] \Bigg)\nonumber\\
    &\qquad+ G_f \|\mathbb E_k[e_d]\| + \frac{L_f}{2}\eta \left(\norm*{\mathbb E_k\left[e_d \right]}^2 + \mathbb E_k\left[\norm*{e_d}^2 \right] \right)\nonumber\\
    &\qquad+ G_f \sqrt{2\beta\frac{\Delta g_k}{\eta} + 2\beta L_g \eta \norm*{\mathbf d_k}^2 +\left(1+L_g\eta \beta \right)\norm*{\mathbb E_k\left[e_d \right]}^2 + L_g\eta \beta \mathbb E_k\left[\norm*{e_d}^2 \right]}.
\end{align*}
Rearranging the terms,
\begin{align}\label{eq:d-boundz1}
    &\left(1-(L_f + 2\beta L_g)\eta \right)\norm*{\mathbf d_k}^2\nonumber\\
    & \quad \leq \frac{\Delta f_k + 2\beta \Delta g_k}{\eta}  + \left(1+L_g \eta \beta + \frac{L_f \eta}{2}\right)\norm*{\mathbb E_k \left[ e_d\right]}^2 + \left(L_g \eta \beta + \frac{L_f \eta}{2}\right) \mathbb E_k \left[\norm*{e_d}^2 \right]\nonumber\\
    &\qquad  + G_f \sqrt{2\beta\frac{\Delta g_k}{\eta} + 2\beta L_g \eta \norm*{\mathbf d_k}^2 +\left(1+L_g\eta\beta \right)\norm*{\mathbb E_k\left[e_d \right]}^2 + L_g\eta\beta \mathbb E_k\left[\norm*{e_d}^2 \right]}\nonumber\\
    & \qquad + G_f \norm*{\mathbb E_k \left[e_d \right]}.
\end{align}
The preceding inequality in \eqref{eq:d-boundz1} is implicit in $\|\mathbf d_k\|$, as the same term also appears on the right-hand side within the square root. To overcome this issue, by using the fact that $\sqrt{u+v}\leq\sqrt{u}+\sqrt{v}$, Young's inequality, and Assumption \ref{assum:bounded-stoch-g}, we have
\begin{align*}
    &\left(1-(L_f + 2\beta L_g)\eta \right)\norm*{\mathbf d_k}^2\nonumber\\
    & \quad \leq \frac{\Delta f_k + 2\beta \Delta g_k}{\eta} + G_f \norm*{\mathbb E_k \left[e_d \right]} + \left(1+L_g \eta \beta + \frac{L_f \eta}{2}\right)\norm*{\mathbb E_k \left[ e_d\right]}^2 \nonumber\\
    &\qquad  + G_f \sqrt{\left(1+L_g\eta \beta \right)\norm*{\mathbb E_k\left[e_d \right]}^2 + L_g\eta \beta \mathbb E_k\left[\norm*{e_d}^2 \right]} \nonumber\\
    &\qquad + \left(L_g \eta \beta + \frac{L_f \eta}{2}\right) \mathbb E_k \left[\norm*{e_d}^2 \right]+ \sqrt{2}G_f\sqrt{\beta\frac{\Delta g_k}{\eta}+\beta L_g \eta \norm*{\mathbf d_k}^2}\nonumber\\
    & \quad \leq \frac{\Delta f_k + 2\beta \Delta g_k}{\eta} + G_f\norm*{\mathbb E_k \left[e_d \right]} + \left(1+L_g \eta \beta + \frac{L_f \eta}{2}\right)\norm*{\mathbb E_k \left[ e_d\right]}^2 + \frac{1}{2}\norm*{\mathbf d_k}^2\nonumber\\
    &\qquad  + G_f \sqrt{\left(1+L_g\eta\beta \right)\norm*{\mathbb E_k\left[e_d \right]}^2 + L_g\eta\beta \mathbb E_k\left[\norm*{e_d}^2 \right]} + G_f^2 \beta L_g \eta \nonumber\\
    &\qquad + \left(L_g \eta \beta + \frac{L_f \eta}{2}\right) \mathbb E_k \left[\norm*{e_d}^2 \right] +\frac{\Delta g_k}{2L_g\eta^2}.
\end{align*}
Rearranging the term, considering $\eta \leq \frac{1}{4(L_f+2\beta L_g)}$, the left-side of the inequality can be lower-bounded by $1/4\norm*{\mathbf d_k}^2$, and by multiplying both sides by $4$, we can show
\begin{align}\label{eq:d-k-f1}
   &\norm*{\mathbf d_k}^2 \leq \frac{4(\Delta f_k + 2\beta \Delta g_k)}{\eta} + 4G_f \norm*{\mathbb E_k \left[e_d \right]} + \left(4+4L_g \eta \beta + 2L_f \eta\right)\norm*{\mathbb E_k \left[ e_d\right]}^2 \nonumber\\
    &\quad  + \left(4L_g \eta \beta + 2L_f \eta\right) \mathbb E_k \left[\norm*{e_d}^2 \right] + 4G_f^2 \beta L_g \eta +\frac{2\Delta g_k}{L_g\eta^2}\nonumber\\
    &\quad + 4G_f \sqrt{\left(1+L_g\eta\beta \right)\norm*{\mathbb E_k\left[e_d \right]}^2 + L_g\eta\beta \mathbb E_k\left[\norm*{e_d}^2 \right]}.
\end{align}

\end{proof}

\begin{lemma}\label{lem:Lip-Phi}
Let $\Phi_{\gamma}:\mathbb R^n\times\mathbb R^n\to\mathbb R^n$ be defined as $\Phi_\gamma(u,v)\triangleq \left[\beta-\frac{u^\top v - c}{\|u\|^2+\gamma}\right]_+u$, for some $\gamma>0,\beta,c\geq 0$. Then, for any $u,v,w,z\in\mathbb R^n$, $\|\Phi_{\gamma}(u,v)-\Phi_{\gamma}(w,z)\|\leq \left(\beta+\frac{4c}{\gamma}+\frac{2\|z\|}{\sqrt{\gamma}}\right)\|u-w\| + \|v-z\|$.
\end{lemma}
\begin{proof}
Using the triangle inequality, we have that 
\begin{align}\label{eq:Phi-bound}
\|\Phi_\gamma(u,v)-\Phi_\gamma(w,z)\|
&\leq \|\Phi_\gamma(u,v)-\Phi_\gamma(u,z)\| + \|\Phi_\gamma(u,z)-\Phi_\gamma(w,z)\|\nonumber\\
&= \|\Phi_\gamma(u,v)-\Phi_\gamma(u,z)\| 
\nonumber\\
&\quad + \left\|\left[\beta-\frac{u^\top z - c}{\|u\|^2+\gamma}\right]_+u - \left[\beta-\frac{w^\top z - c}{\|w\|^2+\gamma}\right]_+w\right\|. 
\end{align}

For the first term on the right-hand side of \eqref{eq:Phi-bound} again using Cauchy-Schwarz inequality and non-expansivity of $[\cdot]_+$ we have
\begin{align}\label{eq:z-lip-bound}
\|\Phi_\gamma(u,v)-\Phi_\gamma(u,z)\|
&=
\left\|
\left[
\beta-\frac{u^\top v - c}{\|u\|^2+\gamma}
\right]_+u
-
\left[
\beta-\frac{u^\top z - c}{\|u\|^2+\gamma}
\right]_+u
\right\|  \nonumber\\
&\le
\left|
\frac{u^\top(v-z)}{\|u\|^2+\gamma}
\right|\|u\| \nonumber\\
&\le
\frac{\|u\|^2}{\|u\|^2+\gamma}\|v-z\| \nonumber\\
&\le
\|v-z\|.
\end{align}

For the second term on the right-hand side of \eqref{eq:Phi-bound}, let us define $b_z(x):=\frac{x^\top z-c}{\|x\|^2+\gamma}$, $a_z(x):=\beta-b_z(x)$, $F_z(x):=[a_z(x)]_+x$, then $\Phi_\gamma(x,z)=F_z(x)$. 
Let $q(t)=w+t(u-w)$, $t\in[0,1]$. Note that the map $t\mapsto F_z(q(t))$ is absolutely
continuous and its derivative for every $t$ can be calculated as,
\[
\frac{d}{dt}F_z(q(t))
=
\frac{d}{dt}[a_z(q(t))]_+\,q(t)
+
[a_z(q(t))]_+(u-w).
\]
Using the chain rule we obtain
\[
\left\|
\frac{d}{dt}F_z(q(t))
\right\|
\leq
\left(
[a_z(q(t))]_+
+
\|q(t)\|\,\|\nabla a_z(q(t))\|
\right)\|u-w\|.
\]
Therefore it suffices to bound the factor in parentheses uniformly in
$x\in\mathbb R^n$.

First,
\begin{align*}
[a_z(x)]_+
=
\left[
\beta-\frac{x^\top z-c}{\|x\|^2+\gamma}
\right]_+                                      
\leq
\beta+
\left|
\frac{x^\top z-c}{\|x\|^2+\gamma}
\right|                                      
&\leq
\beta+
\frac{\|x\|\|z\|}{\|x\|^2+\gamma}
+
\frac{c}{\|x\|^2+\gamma}     \\                 
&\leq
\beta+\frac{\|z\|}{2\sqrt{\gamma}}+\frac{c}{\gamma},
\end{align*}
where we used $\frac{r}{r^2+\gamma}\leq \frac{1}{2\sqrt{\gamma}}$ for any $r\geq 0$.

Next, since $\nabla b_z(x)
=
\frac{(\|x\|^2+\gamma)z-2(x^\top z-c)x}
     {(\|x\|^2+\gamma)^2}$,
and $\nabla a_z(x)=-\nabla b_z(x)$, we have
\begin{align*}
\|x\|\,\|\nabla a_z(x)\|
&=
\|x\|\,\|\nabla b_z(x)\|                                      \\
&\leq
\frac{\|x\|\|z\|}{\|x\|^2+\gamma}
+
\frac{2\|x\|^2|x^\top z-c|}{(\|x\|^2+\gamma)^2}                \\
&\leq
\frac{\|z\|}{2\sqrt{\gamma}}
+
\frac{2\|x\|^3\|z\|}{(\|x\|^2+\gamma)^2}
+
\frac{2c\|x\|^2}{(\|x\|^2+\gamma)^2}.
\end{align*}
Therefore, using the facts that for any $r\geq 0$, 
$\frac{2r^3}{(r^2+\gamma)^2}
=
\frac{2r^2}{r^2+\gamma}\frac{r}{r^2+\gamma}
\leq
\frac{1}{\sqrt{\gamma}}$, $
\frac{2r^2}{(r^2+\gamma)^2}
\leq
\frac{2}{\gamma}$, 
we obtain
\[
\|x\|\,\|\nabla a_z(x)\|
\leq
\frac{3\|z\|}{2\sqrt{\gamma}}+\frac{2c}{\gamma}.
\]
Combining the two estimates gives
\[
[a_z(x)]_+ + \|x\|\,\|\nabla a_z(x)\|
\leq
\beta+\frac{2\|z\|}{\sqrt{\gamma}}+\frac{3c}{\gamma}
\leq
\beta+\frac{2\|z\|}{\sqrt{\gamma}}+\frac{4c}{\gamma}.
\]
Thus,
\[
\left\|
\frac{d}{dt}F_z(q(t))
\right\|
\leq
\left(
\beta+\frac{4c}{\gamma}+\frac{2\|z\|}{\sqrt{\gamma}}
\right)\|u-w\|.
\]
Integrating over $t\in[0,1]$ yields
\begin{align}\label{eq:Phi-uw}
\|\Phi_\gamma(u,z)-\Phi_\gamma(w,z)\|
\leq
\left(
\beta+\frac{4c}{\gamma}+\frac{2\|z\|}{\sqrt{\gamma}}
\right)\|u-w\|.
\end{align}

Substituting the bounds in \eqref{eq:z-lip-bound} and \eqref{eq:Phi-uw} within \eqref{eq:Phi-bound} proves the claim.

\end{proof}
\section{Proof of Lemma \ref{lem:bias-ed-gradg}-part(a)}\label{Appndx5}
\begin{proof}
Recall that $e_d = \tilde{\mathbf d}_k-\mathbf d_k$, hence, we have that 
\begin{align*}
e_d  &=
\left(\nabla \tilde f(\mathbf x_k,\xi_f^k)-\nabla f(\mathbf x_k)\right)
+
\left(
\tilde{\lambda}_{\gamma,k}\nabla \tilde g(\mathbf x_k,\xi_{g}^k)
-
\lambda_k\nabla g(\mathbf x_k)
\right)\\
&=\left(\nabla \tilde f(\mathbf x_k,\xi_f^k)-\nabla f(\mathbf x_k)\right)
+
\left(
\tilde{\lambda}_{\gamma,k}\nabla \tilde g(\mathbf x_k,\xi_{g}^k)
-
\lambda_{\gamma,k}\nabla g(\mathbf x_k)
\right)
+
\left(\lambda_{\gamma,k}-\lambda_k\right)\nabla g(\mathbf x_k)\\
&=\left(\nabla \tilde f(\mathbf x_k,\xi_f^k)-\nabla f(\mathbf x_k)\right)
+\left(\Phi_{\gamma_k}(\nabla \tilde g_k,\nabla \tilde f_k) - \Phi_{\gamma_k}(\nabla g_k,\nabla f_k)\right)+\left(\lambda_{\gamma,k}-\lambda_k\right)\nabla g(\mathbf x_k),
\end{align*}
where $\Phi_\gamma(u,v)\triangleq \left[\beta-\frac{u^\top v}{\|u\|^2+\gamma}\right]_+u$. 
Taking conditional expectation and using the unbiasedness of $\nabla \tilde f(\mathbf x_k,\xi_f)$, we get
\begin{align}
\norm*{\mathbb E_k\left[e_d\right]}
&=\norm*{\mathbb E_k\left[\left(\Phi_{\gamma_k}(\nabla \tilde g_k,\nabla \tilde f_k) - \Phi_{\gamma_k}(\nabla g_k,\nabla f_k)\right)+\left(\lambda_{\gamma,k}-\lambda_k\right)\nabla g(\mathbf x_k)\right]}\nonumber\\
&\leq 
\mathbb E_k\norm*{\Phi_{\gamma_k}(\nabla \tilde g_k,\nabla \tilde f_k) - \Phi_{\gamma_k}(\nabla g_k,\nabla f_k)}
+
\left|\lambda_{\gamma,k}-\lambda_k\right|
\,
\|\nabla g(\mathbf x_k)\|\nonumber\\
&\leq \left(\beta+\frac{2\|\nabla f(\mathbf x_k)\|}{\sqrt{\gamma_k}}\right)\mathbb E_k\left[\|\nabla \tilde g(\mathbf x_k,\xi_g^k)-\nabla g(\mathbf x_k)\|\right]+\mathbb E_k\left[\|\nabla \tilde f(\mathbf x_k,\xi_f^k)-\nabla f(\mathbf x_k)\|\right]
\nonumber\\
&\quad +
\left|\lambda_{\gamma,k}-\lambda_k\right|
\,
\|\nabla g(\mathbf x_k)\|\nonumber\\
&\leq \left(\beta+\frac{2G_f}{\sqrt{\gamma_k}}\right)\frac{\nu_g}{\sqrt{N_g}} + \frac{\nu_f}{\sqrt{N_f}}+\left|\lambda_{\gamma,k}-\lambda_k\right|
\,
\|\nabla g(\mathbf x_k)\|.
\label{eq:ed-bias-1}
\end{align}
where in the penultimate inequality we applied Lemma \ref{lem:Lip-Phi} with $c=0$, and the last inequality follows from Assumption \ref{assum:stoch_grad}. 

To upper bound the last term on the right-hand side of \eqref{eq:ed-bias-1}, in Lemma \ref{lem:lambda_eps_bound}, we have
\begin{align}
&\left|\lambda_{\gamma,k}-\lambda_k\right|\le
\frac{
\gamma_k
\|\nabla f(\mathbf x_k)\|
}{
\|\nabla g(\mathbf x_k)\|
\left(\|\nabla g(\mathbf x_k)\|^2+\gamma_k\right)
}.
\label{eq:ed-bias-2}
\end{align}
Combining \eqref{eq:ed-bias-2} with \eqref{eq:ed-bias-1} and noting that $\|\nabla f(\mathbf x_k)\|\leq G_f$ leads to the desired result.

\end{proof}

\section{Proof of Lemma \ref{lem:bias-ed-gradg}-part(b)}\label{Appndx6}
\begin{proof}
We start upper bounding $\mathbb E_k[\|e_d\|^2]$ by using $
\|a+b+c\|^2\le 3\|a\|^2+3\|b\|^2+3\|c\|^2,$ as follows:
\begin{align}
\mathbb E_k\left[\|e_d\|^2\right]
&\le
3\,
\mathbb E_k\left[
\|\nabla \tilde f(\mathbf x_k,\xi_f^k)-\nabla f(\mathbf x_k)\|^2\right]
\nonumber\\
&\quad
+
3\,
\mathbb E_k\left[
\left\|
\tilde{\lambda}_{\gamma,k}\nabla \tilde g(\mathbf x_k,\xi_{g}^k)
-
\lambda_{\gamma,k}\nabla g(\mathbf x_k)
\right\|^2\right]
\nonumber\\
&\quad
+
3\,
\left|\lambda_{\gamma,k}-\lambda_k\right|^2
\|\nabla g(\mathbf x_k)\|^2.
\label{eq:ed-second-1}
\end{align}

The first term on the right-hand side of \eqref{eq:ed-second-1} is bounded by
\begin{align}
\mathbb E_k\left[
\|\nabla \tilde f(\mathbf x_k,\xi_f^k)-\nabla f(\mathbf x_k)\|^2\right]
\le
\frac{\nu_f^2}{N_f}.
\label{eq:ed-second-2}
\end{align}

For the third term on the right-hand side of \eqref{eq:ed-second-1}, by \eqref{eq:ed-bias-2}, we have
\begin{align}
\left|\lambda_{\gamma,k}-\lambda_k\right|^2
\|\nabla g(\mathbf x_k)\|^2
&\le
\frac{
\gamma_k^2
G_f^2
}{
\left(\|\nabla g(\mathbf x_k)\|^2+\gamma_k\right)^2
}.
\label{eq:ed-second-3}
\end{align}

Next, we focus on the second term on the right-hand side of \eqref{eq:ed-second-1}. Applying Lemma \ref{lem:Lip-Phi} with $c=0$ for $\Phi_{\gamma_k}(\cdot,\cdot)$ at $u=\nabla \tilde g(\mathbf x_k,\xi_g^k)$, $w=\nabla g(\mathbf x_k)$, $v=\nabla \tilde f(\mathbf x_k,\xi_f^k)$, and $z=\nabla f(\mathbf x_k)$ we conclude that 
\begin{align}\label{eq:lambda*grad-variance}
&\mathbb E_k\left[
\left\|
\tilde{\lambda}_{\gamma,k}\nabla \tilde g(\mathbf x_k,\xi_{g}^k)
-
\lambda_{\gamma,k}\nabla g(\mathbf x_k)
\right\|^2\right] \nonumber\\
&\leq 2\left(\beta+\frac{2\|\nabla f(\mathbf x_k)\|}{\sqrt{\gamma_k}}\right)^2\mathbb E_k\left[\|\nabla \tilde g(\mathbf x_k,\xi_g^k)-\nabla g(\mathbf x_k)\|^2\right]+2\mathbb E_k\left[\|\nabla \tilde f(\mathbf x_k,\xi_f^k)-\nabla f(\mathbf x_k)\|^2\right]
\nonumber\\
&\leq 2\left(\beta+\frac{2G_f}{\sqrt{\gamma_k}}\right)^2\frac{\nu_g^2}{N_g} +2\frac{\nu_f^2}{N_f}.
\end{align}
Substituting \eqref{eq:ed-second-2}, \eqref{eq:ed-second-3}, and \eqref{eq:lambda*grad-variance} in \eqref{eq:ed-second-1}, lead to the desired result.
\end{proof}


\section{Proof of Proposition \ref{thm:theorem1}}\label{Appndx9}

\begin{proof}

Start from \eqref{eq:exp-ed1-gradg},
\begin{align*}
\norm*{\mathbb E_k\left[e_d\right]}
\leq \frac{\nu_f}{\sqrt{N_f}}+\left(\beta_k+\frac{2G_f}{\sqrt{\gamma_k}}\right)\frac{\nu_g}{\sqrt{N_g}}  + \frac{
\gamma_k G_f}{\|\nabla g(\mathbf x_k)\|^2+\gamma_k}.
\end{align*}
Assume that $\gamma_k$ is of the same order as
$K^{-r}\|\nabla g(\mathbf x_k)\|^2$, namely,
$\gamma_k = K^{-r}\|\nabla g(\mathbf x_k)\|^2$,
where $r>0$, 
and choose batch sizes as follows
\begin{align}\label{eq:param-Ngp}
    N_f
=
K^{a_f}, \quad N_g
=K^{a_g}\|\nabla g(\mathbf x_k)\|^{-2},
\end{align}
where $a_f,a_g, r> 0$. 

Assuming that $\beta_k\leq \frac{2G_f}{\sqrt{\gamma_k}}$, using \eqref{eq:param-Ngp}, we can simplify the bound as follows 
\begin{align}\label{eq:bias-final-proof}
    \|\mathbb E_k[e_d]\|
&\le \nu_f K^{-a_f/2} + 4G_f\nu_g K^{(r-a_g)/2} + G_f \frac{K^{-r}}{1+K^{-r}}.
\end{align}
Moreover, from \eqref{eq:exp-ed2-gradg}, we have that 
\begin{align*}
\mathbb E_k\left[\|e_d\|^2\right]\leq \frac{9\nu_f^2}{N_f} 
+6\left(\beta_k+\frac{2G_f}{\sqrt{\gamma_k}}\right)^2\frac{\nu_g^2}{N_g} 
+ \frac{3\gamma_k^2G_f^2
}{\left(\|\nabla g(\mathbf x_k)\|^2+\gamma_k\right)^2}.
\end{align*}
which, using similar steps, implies that 
\begin{align}
\mathbb E_k\left[\|e_d\|^2\right]\leq 9\nu_f^2 K^{-a_f} + 48G_f^2\nu_g^2 K^{r-a_g} + 3 G_f^2 \left(\frac{K^{-r}}{1+K^{-r}}\right)^2.
\end{align}

Next, consider \eqref{eq:d-k-f1} and \eqref{eq:g-bound1}, using Assumption \ref{assum:bounded-stoch-g},

\begin{align}
&\norm*{\mathbf d_k}^2 \leq \frac{4(\Delta f_k + 2\beta \Delta g_k)}{\eta} + 4G_f \norm*{\mathbb E_k \left[e_d \right]} + \left(4+4L_g \eta \beta + 2L_f \eta\right)\norm*{\mathbb E_k \left[ e_d\right]}^2 \nonumber\\
    &\quad  + \left(4L_g \eta \beta + 2L_f \eta\right) \mathbb E_k \left[\norm*{e_d}^2 \right] + 4G_f^2 \beta L_g \eta +\frac{2\Delta g_k}{L_g\eta^2}\nonumber\\
    &\quad + 4G_f \sqrt{\left(1+L_g\eta\beta \right)\norm*{\mathbb E_k\left[e_d \right]}^2 + L_g\eta\beta \mathbb E_k\left[\norm*{e_d}^2 \right]}.
\label{eq:dk-start-proof}
\end{align}
\begin{align}\label{eq:g-bound1wcg}
     &\beta\norm*{\nabla g(\mathbf x_k)}^2 \leq \frac{2\Delta g_k}{\eta} +2L_g \eta \norm*{\mathbf d_k}^2 + (\frac{1}{\beta}+L_g\eta)\norm*{\mathbb E_k\left[e_d \right]}^2 + L_g\eta \mathbb E_k\left[\norm*{e_d}^2 \right].
\end{align}

Multiplying \eqref{eq:g-bound1wcg} by $1/(4L_g\eta)$, we obtain
\begin{align*}
\frac{\beta}{4L_g\eta}\norm*{\nabla g(\mathbf x_k)}^2
&\le
\frac{\Delta g_k}{2L_g\eta^2}
+
\frac{1}{2}\norm*{\mathbf d_k}^2
+\left(\frac{1}{4L_g\eta\beta} +\frac{1}{4}\right)\norm*{\mathbb E_k\left[e_d \right]}^2
+
\frac14 \mathbb E_k\left[\norm*{e_d}^2 \right].
\end{align*}
Adding this inequality to \eqref{eq:dk-start-proof}, we obtain
\begin{align*}
\frac{1}{2}\|\mathbf d_k\|^2+\frac{\beta}{4L_g\eta}\|\nabla g(\mathbf x_k)\|^2
&\le
\frac{4\Delta f_k+8\beta\Delta g_k}{\eta}
+\frac{5\Delta g_k}{2L_g\eta^2}
+4G_f^2\beta L_g\eta\\
&\quad +\left(4G_f\right)\norm*{\mathbb E_k \left[e_d \right]} \\
&\quad +\Big(4L_g\eta\beta+2L_f\eta+\frac{1}{4L_g\eta \beta}+\frac{17}{4}\Big)\norm*{\mathbb E_k \left[e_d \right]}^2 \\
&\quad +\Big(4L_g\eta\beta+2L_f\eta+\tfrac14\Big)\mathbb E_k\left[\norm*{e_d}^2 \right] \\
&\quad +4G_f \sqrt{\left(1+L_g\eta\beta \right)\norm*{\mathbb E_k\left[e_d \right]}^2 + L_g\eta\beta \mathbb E_k\left[\norm*{e_d}^2 \right]}.
\end{align*}

Define the potential function
$\mathcal G_k \;\triangleq\; \frac{1}{2}\norm*{\mathbf d_k}^2 + \frac{\beta}{4L_g\eta}\norm*{\nabla g(\mathbf x_k)}^2.$ 
Taking the average of the preceding inequality over $k=0,\ldots,K-1$, using $\sum_{k=0}^{K-1}\Delta f_k=f(\mathbf x_0)-f(\mathbf x_K)
\le f(\mathbf x_0)-\inf f = \Delta f$, and $\sum_{k=0}^{K-1}\Delta g_k
= g(\mathbf x_0)-g(\mathbf x_K) \le g(\mathbf x_0)-g^* =
\Delta g$, and taking the total expectation, we arrive at
\begin{align}\label{eq:bound-G-error}
\frac{1}{K}\sum_{k=0}^{K-1}\mathbb E[\mathcal G_k]
&\le
\frac{4\Delta f+8\beta\Delta g}{\eta K}
+\frac{5\Delta g}{2L_g\eta^2 K}
+4G_f^2\beta L_g\eta \nonumber\\
&\quad +
\left(4G_f\right)
\frac{1}{K}\sum_{k=0}^{K-1}\mathbb E\left[\norm*{\mathbb E_k \left[e_d \right]}\right] \nonumber\\
&\quad +
\Big(4L_g\eta\beta+2L_f\eta+\frac{1}{4L_g\eta\beta}+\frac{17}{4}\Big)
\frac{1}{K}\sum_{k=0}^{K-1} \mathbb E\left[\norm*{\mathbb E_k \left[e_d \right]}^2\right] \nonumber \\
&\quad +
\Big(4L_g\eta\beta+2L_f\eta+\tfrac14\Big)
\frac{1}{K}\sum_{k=0}^{K-1}\mathbb E\left[\norm*{e_d}^2 \right] \nonumber \\
&\quad +
\frac{4G_f}{K}
\sum_{k=0}^{K-1}
\sqrt{\left(1+L_g\eta\beta \right)\mathbb E\left[\norm*{\mathbb E_k\left[e_d \right]}\right]^2 + L_g\eta\beta \mathbb E\left[\norm*{e_d}^2 \right]}.
\end{align}

Now, let $\eta=K^{-s}\in (0,1)$, $\beta=K^{-q}\in (0,1)$, and define $\alpha=\min\{a_f/2,(a_g-r)/2,r\}\geq 0$. Replacing the above estimates into the bound for $\frac1K\sum_{k=0}^{K-1}\mathbb E[\mathcal G_k]$, and keeping the largest coefficient of error terms, we obtain
\begin{align*}
&\frac{1}{K}\sum_{k=0}^{K-1}\mathbb E[\mathcal G_k]\\
&\quad =
\mathcal O\Big(K^{s-1}+K^{s-q-1}+K^{2s-1}+K^{-(q+s)}+K^{-\alpha}+K^{s+q-2\alpha}+K^{-2\alpha}+K^{-(q+s+2\alpha)/2}\Big)\\
&\quad =\mathcal O\Big(K^{2s-1}+K^{-(q+s)}+K^{-\alpha}+K^{s+q-2\alpha}\Big)
\end{align*}

Now, considering the parameter choices $s=\frac{1}{2(p+1)}$, $q=\frac{p}{2(p+1)}$, $\alpha=\frac{3p+1}{4(p+1)}$ for any $p\geq 0$, we conclude that $\frac{1}{K}\sum_{k=0}^{K-1}\mathbb E[\mathcal G_k]=\mathcal O(K^{-\mu_1})$ where $\mu_1\triangleq \min\{\frac{p}{p+1},\frac{3p+1}{4(p+1)},\frac{1}{2}\}$. By the definition $\mathcal{G}_{k} = \frac{\norm*{\mathbf d_{k}}^2}{2} + \frac{\beta \norm*{\nabla g(\mathbf x_{k})}^2}{2L_g \eta}$, this implies $\norm*{\mathbf d_{k}}^2 \leq 2\mathcal{G}_{k}$, and $\norm*{\nabla g(\mathbf x_{k})}^2 \leq \frac{2L_g \eta}{\beta}\mathcal{G}_{k}$. Therefore, defining $\mu_2\triangleq \min\{\frac12,\frac{p+3}{4(p+1)},\frac{1}{p+1}\}$, 
\begin{align}\label{eq:d-gradg-bnd-p}
    &\frac{1}{K}\sum_{k=0}^{K-1}\mathbb E\left[\norm*{\mathbf d_{k}}^2\right] = \mathcal{O}\left(K^{-\mu_1} \right),\quad \frac{1}{K}\sum_{k=0}^{K-1}\mathbb E\left[\norm*{\nabla g(\mathbf x_{k})}^2\right] = \mathcal{O}\left(K^{-\mu_2} \right).
\end{align}

Moreover, substituting the chosen parameters into \eqref{eq:param-Ngp}, we have that $a_f=2\alpha$ and $a_g=3\alpha$, hence, the per-iteration oracle costs are $N_f=K^{\frac{3p+1}{2(p+1)}}$ and $N_g=K^{\frac{9p+3}{4(p+1)}}\|\nabla g(\mathbf x_k)\|^{-2}$. Therefore, to achieve $(\epsilon_f,\epsilon_g)$-stationary solution, we need $K= \mathcal O(\max\{\epsilon_f^{-\frac{1}{\mu_1}},\epsilon_g^{-\frac{1}{\mu_2}}\})$ iterations. Let $\epsilon=\min\{\epsilon_f,\epsilon_g\}$ and $\mu\triangleq \min\{\mu_1,\mu_2\}$, this implies that $K= \mathcal O(\epsilon^{-\frac{1}{\mu}})$. Therefore, the required number of sample gradients for the upper-level function can be computed as  $\sum_{k=0}^{K-1}N_f=K^{\frac{5p+3}{2(p+1)}}= \mathcal O(\epsilon^{-\frac{1}{\mu}\frac{5p+3}{2(p+1)}})$, and for the lower-level function we have 
\begin{align}\label{eq:lower-Ng}
\sum_{k=0}^{K-1}N_g^{(k)} = K^{\frac{9p+3}{4(p+1)}}\sum_{k=0}^{K-1}\|\nabla g(\mathbf x_k)\|^{-2}. 
\end{align}

Minimizing the exponents in upper-level sample complexity $\frac{1}{\mu}\frac{5p+3}{2(p+1)}$ 
with respect to $p$, we conclude that the unique optimal minimizer is $p=1$. Therefore, the convergence rate simplifies to  
\begin{align}\label{eq:d-gradg-bnd-1}
    &\frac{1}{K}\sum_{k=0}^{K-1}\mathbb E\left[\norm*{\mathbf d_{k}}^2\right] = \mathcal{O}\left(K^{-\frac{1}{2}} \right),\quad \frac{1}{K}\sum_{k=0}^{K-1}\mathbb E\left[\norm*{\nabla g(\mathbf x_{k})}^2\right] = \mathcal{O}\left(K^{-\frac{1}{2}} \right).
\end{align}
Moreover, the required number of sample gradients for the upper-level function is $\sum_{k=0}^{K-1}N_f=K^{2}= \mathcal O(\epsilon^{-4})$, and for the lower-level function, we obtain $\sum_{k=0}^{K-1}N_g^{(k)}= K^{3/2}\sum_{k=0}^{K-1}\|\nabla g(\mathbf x_k)\|^{-2}$.

\end{proof}

Before, proving the result of Theorem \ref{thm:theorem2}, we state the modified version of Algorithm \ref{alg:alg1} as stated in the statement of Theorem \ref{thm:theorem2}. 

\begin{algorithm}[tbh!]
\caption{Modified SDBPG Method}
\label{alg:alg1-modified}
\begin{algorithmic}[1]
\State \textbf{Input:} step sizes $\{\eta_k\}_{k\ge 0}$, barrier parameters $\{\beta_k\}_{k\ge 0}$, perturbation sequence $\{\gamma_k\}_{k\ge 0}$, thresholds $\theta,\tau\geq 0$, and sample sizes $N_f,N_g$
\State \textbf{Initialization:} $\mathbf{x}_0 \in \mathbb R^n$
\For{$k = 0,1,2,\dots,K-1$}
    \State Draw independent mini-batch samples $\hat\xi_f^k,\hat\xi_{g}^k$ with size $N_f$ and $N_g$, respectively.
    \State Compute $\nabla \hat{f}_k = \frac{1}{N_f}\sum_{i=1}^{N_f} \nabla \tilde f(\mathbf{x}_k, \hat\xi_{f,i}^k)$, $\nabla \hat{g}_k = \frac{1}{N_g}\sum_{i=1}^{N_g} \nabla \tilde g(\mathbf{x}_k, \hat\xi_{g,i}^k)$.
    \State Draw independent mini-batch samples $\xi_f^k,\xi_{g}^k$ with size $N_f$ and $N_g$, respectively. 
    
    \State Compute $\nabla \tilde{f}_k = \frac{1}{N_f}\sum_{i=1}^{N_f} \nabla \tilde f(\mathbf{x}_k, \xi_{f,i}^k)$, $\nabla \tilde{g}_k = \frac{1}{N_g}\sum_{i=1}^{N_g} \nabla \tilde g(\mathbf{x}_k, \xi_{g,i}^k)$. 

    \If{$ \norm{\nabla \hat g_k}^2\le \tau/2$ and $\nabla \hat g_k^\top\nabla \hat f_k< -\theta\norm{\nabla \hat g_k}^2$}
     \State $\tilde\lambda_{\gamma,k} \gets 0$
    
    \Else
    \State $\tilde\lambda_{\gamma,k} \gets \max\big\{\beta_k - \tfrac{\nabla \tilde g_k^\top \nabla\tilde f_k}{\|\nabla \tilde g_k\|^2+\gamma_k}, 0\big\}$
    \EndIf
    \State $\tilde{\mathbf d}_k
    \gets
    \nabla \tilde f_k
    +
    \tilde{\lambda}_{\gamma,k}
    \nabla \tilde g_k$
    \State 
    $\mathbf{x}_{k+1}
    \gets
    \mathbf{x}_k
    -
    \eta_k \tilde{\mathbf d}_k$
\EndFor
\end{algorithmic}
\end{algorithm}

\section{Proof of Theorem \ref{thm:theorem2}}\label{Appndx10}
\begin{lemma}[Perturbation error outside the bad region]
\label{lem:lambda-bad-region-bound}
Let $\lambda_k$ and $\lambda_{\gamma,k}$ be defined by
\eqref{eq:tild-lambda1} and \eqref{eq:lambda-eps1}, respectively, and use
the convention that
$|\lambda_{\gamma,k}-\lambda_k|\|\nabla g(\mathbf x_k)\|=0$ whenever
$\nabla g(\mathbf x_k)=0$. if $\mathbf x_k\notin \mathcal R_k(\tau,\theta)$, then
\[
|\lambda_{\gamma,k}-\lambda_k|\,\|\nabla g(\mathbf x_k)\|
\le
\frac{\gamma_kG_f}{\tau+\gamma_k}
+
(\beta_k+\theta)\sqrt{\tau}.
\]
\end{lemma}
\begin{proof}
Let
$a_k\triangleq \|\nabla g(\mathbf x_k)\|^2,\  
b_k\triangleq \nabla g(\mathbf x_k)^\top\nabla f(\mathbf x_k)
$, and suppose $\mathbf x_k\notin \mathcal R_k(\tau,\theta)$. If
$a_k>\tau$, then from Lemma \ref{lem:lambda_eps_bound} we have $|\lambda_{\gamma,k}-\lambda_k|\|\nabla g(\mathbf x_k)\|
\le
\frac{\gamma_kG_f}{\tau+\gamma_k}$. 
It remains to consider $a_k\le \tau$. Since
$\mathbf x_k\notin\mathcal R_k(\tau,\theta)$, we must have
$b_k\ge -\theta a_k$.

If $b_k<0$, then $\lambda_k=\beta_k-\frac{b_k}{a_k}$ and $\lambda_{\gamma,k}
=
\beta_k-\frac{b_k}{a_k+\gamma_k}$, hence, 
\[
|\lambda_{\gamma,k}-\lambda_k|\sqrt{a_k}
=
\frac{\gamma_k(-b_k)}{a_k(a_k+\gamma_k)}\sqrt{a_k}
\le
\frac{\gamma_k\theta a_k}{a_k(a_k+\gamma_k)}\sqrt{a_k}
\le
\theta\sqrt{a_k}
\le
\theta\sqrt{\tau}.
\]
If $b_k\ge 0$, then $\lambda_{\gamma,k}\le \beta_k$ and
$\lambda_k\ge 0$, while $\lambda_{\gamma,k}\ge\lambda_k$. Therefore,
\[
|\lambda_{\gamma,k}-\lambda_k|\sqrt{a_k}
\le
\lambda_{\gamma,k}\sqrt{a_k}
\le
\beta_k\sqrt{\tau}.
\]
Combining the cases gives
\[
|\lambda_{\gamma,k}-\lambda_k|\|\nabla g(\mathbf x_k)\|
\le
\frac{\gamma_kG_f}{\tau+\gamma_k}
+
(\beta_k+\theta)\sqrt{\tau}.
\]
\end{proof}

Now, we are ready to present the proof of the convergence rate and complexity of the proposed method.

\begin{proof}[Proof of Theorem \ref{thm:theorem2}]
Let $\mathcal F_k$ denote the history after drawing independent mini-batch samples $\hat\xi_f^k,\hat\xi_g^k$ and before the independent direction mini-batch is drawn. We partition the underlying probability space at each iteration $k$ into the four disjoint events
\begin{align*}
    &E_k^{(0)}\triangleq \{\mathbf x_k\notin \mathcal{R}_k(\tau,\theta)\}\cap \{\mathbf x_k\notin \widehat{\mathcal{R}}_k(\tau/2,\theta)\},\\
    &E_k^{(1)}\triangleq \{\mathbf x_k\notin \mathcal{R}_k(\tau,\theta)\}\cap \{\mathbf x_k\in \widehat{\mathcal{R}}_k(\tau/2,\theta)\}\cap \{\norm*{\nabla g(\mathbf x_k)}^2> \tau\},\\
    &E_k^{(2)}\triangleq \{\mathbf x_k\notin \mathcal{R}_k(\tau,\theta)\}\cap \{\mathbf x_k\in \widehat{\mathcal{R}}_k(\tau/2,\theta)\}\cap \{\norm*{\nabla g(\mathbf x_k)}^2\leq \tau\},\\
    &E_k^{(3)}\triangleq \{\mathbf x_k\in \mathcal{R}_k(\tau,\theta)\}.
\end{align*}
With the above convention, each $E_k^{(j)}$ is $\mathcal F_k$-measurable.

\textbf{Case $E_k^{(0)}$.} On this event the multiplier is given by \eqref{eq:tild-lambda-eps-2} and $\mathbf x_k\notin \mathcal R_k(\tau,\theta)$, hence Lemma \ref{lem:lambda-bad-region-bound} applies verbatim. From the inequality in \eqref{eq:ed-bias-1} and Lemma \ref{lem:lambda-bad-region-bound} we conclude that
$$\|\mathbb E_k[e_d]\|\mathbf{1}_{E_k^{(0)}}\leq \frac{\nu_f}{\sqrt{N_f}} +\left(\beta_k+\frac{2G_f}{\sqrt{\gamma_k}}\right)\frac{\nu_g}{\sqrt{N_g}} +  \frac{
\gamma_k G_f}{\tau+\gamma_k}+(\beta_k+\theta)\sqrt{\tau} ,$$
where $\mathbf{1}_{E_k^{(0)}}$ denotes the indicator function on the set $E_k^{(0)}$, and from the proof of Lemma \ref{lem:bias-ed-gradg}-part(b) we obtain
$$\mathbb E_k\left[\|e_d\|^2\right]\mathbf{1}_{E_k^{(0)}}\leq \frac{12\nu_f^2}{N_f}+8\left(\beta_k+\frac{2G_f}{\sqrt{\gamma_k}}\right)^2\frac{\nu_g^2}{N_g}+\frac{4\gamma_k^2G_f^2}{(\tau+\gamma_k)^2}+4(\beta_k+\theta)^2\tau.$$

\textbf{Case $E_k^{(1)}$.} On this event we have $\norm*{\nabla \hat g_k}^2\leq \tau/2$ while $\norm*{\nabla g(\mathbf x_k)}^2> \tau$, hence, the triangle inequality implies that $\norm*{\nabla \hat g_k-\nabla g(\mathbf x_k)}\geq \sqrt{\tau}-\sqrt{\tau/2}=(1-\tfrac{1}{\sqrt{2}})\sqrt{\tau}$. Therefore, from Chebyshev inequality together with Assumption \ref{assum:stoch_grad} we obtain
$$\mathbb P\left(E_k^{(1)}\right)\leq \frac{\mathbb E\left[\norm*{\nabla \hat g_k-\nabla g(\mathbf x_k)}^2\right]}{(1-\tfrac{1}{\sqrt{2}})^2\tau}\leq \frac{\nu_g^2}{(1-\tfrac{1}{\sqrt{2}})^2N_g\tau}.$$
Note that this event need not be rare, and its contribution is controlled through the detector mini-batch size $N_g$ rather than through Assumption \ref{assump:rare-visit}.

\textbf{Case $E_k^{(2)}$.} Here the detector is active, hence the method sets $\tilde\lambda_{\gamma,k}=0$ and $\tilde{\mathbf d}_k=\nabla \tilde f_k$, while $\mathbf d_k=\nabla f(\mathbf x_k)+\lambda_k\nabla g(\mathbf x_k)$. Since $\mathbf x_k\notin \mathcal R_k(\tau,\theta)$ and $\norm*{\nabla g(\mathbf x_k)}^2\leq \tau$, we must have $\nabla g(\mathbf x_k)^\top \nabla f(\mathbf x_k)\geq -\theta\norm*{\nabla g(\mathbf x_k)}^2$, which together with \eqref{eq:tild-lambda1} implies that $\lambda_k\leq \beta_k+\theta$ and therefore
$$\lambda_k\norm*{\nabla g(\mathbf x_k)}\leq (\beta_k+\theta)\sqrt{\tau}.$$
Consequently, since $e_d=\nabla \tilde f_k-\nabla f(\mathbf x_k)-\lambda_k\nabla g(\mathbf x_k)$ on this event and $\nabla \tilde f_k$ is an unbiased estimator of $\nabla f(\mathbf x_k)$, we conclude that
$$\|\mathbb E_k[e_d]\|\mathbf 1_{E_k^{(2)}}\leq (\beta_k+\theta)\sqrt{\tau},\qquad \mathbb E_k\left[\|e_d\|^2\right]\mathbf 1_{E_k^{(2)}}\leq \frac{2\nu_f^2}{N_f}+2(\beta_k+\theta)^2\tau.$$
These bounds involve only the terms already present in Case $E_k^{(0)}$ up to absolute constants, hence no smallness of $\mathbb P(E_k^{(2)})$ is required.

\textbf{Case $E_k^{(3)}$.} On the other hand, from Assumption~\ref{assump:stoch-bound}, one can verify that when $\tilde{\lambda}_{\gamma,k}>0$, we have
$\|\mathbb E_k[e_d]\| \le 2(\beta_k+G_f)$; moreover, when $\tilde{\lambda}_{\gamma,k}=0$, we get $\|\mathbb E_k[e_d]\| =\|\lambda_{\gamma,k}\nabla g(\mathbf x_k)\|\leq \beta C_g+G_f$. Combining the two cases we obtain a uniform bound as $\|\mathbb E_k[e_d]\| \le (2+C_g)\beta_k +2G_f$. Hence, 
we obtain $\mathbb E\!\left[\|\mathbb E_k[e_d]\|\mathbf 1_{E_k^{(3)}}\right]
\le
((2+C_g)\beta_k+2G_f)\mathbb P(E_k^{(3)})$. 
Therefore, using the law of total expectation over the four disjoint events, we conclude that 
\begin{align*}
    \mathbb E\left[\|\mathbb E_k[e_d]\|\right]&=\sum_{j=0}^{3}\mathbb E\left[\|\mathbb E_k[e_d]\|\mathbf{1}_{E_k^{(j)}}\right]\\
    &\leq  \frac{\nu_f}{\sqrt{N_f}}+\left(\beta_k+\frac{2G_f}{\sqrt{\gamma_k}}\right)\frac{\nu_g}{\sqrt{N_g}}  + \frac{
\gamma_k G_f}{\tau+\gamma_k}+2(\beta_k+\theta)\sqrt{\tau} \\
&\quad +((2+C_g)\beta_k+2G_f)\left(\mathbb P(E_k^{(3)})+\frac{\nu_g^2}{(1-\tfrac{1}{\sqrt{2}})^2N_g\tau}\right).
\end{align*}
Let $p_K\triangleq \frac{1}{K}\sum_{k=0}^{K-1}\left(\mathbb P(E_k^{(3)})+\frac{\nu_g^2}{(1-\tfrac{1}{\sqrt{2}})^2N_g\tau}\right)$, and note that $p_K=\mathcal O(K^{-\varsigma}+\nu_g^2N_g^{-1}\tau^{-1})$ from Assumption \ref{assump:rare-visit}.  
Taking average over iterations along with parameter selection $\gamma_k=\gamma$ and $\beta_k=\beta$ we obtain
\begin{align*}
    \frac{1}{K}\sum_{k=0}^{K-1}\mathbb E\left[\|\mathbb E_k[e_d]\|\right]&\leq \frac{\nu_f}{\sqrt{N_f}} +\left(\beta+\frac{2G_f}{\sqrt{\gamma}}\right)\frac{\nu_g}{\sqrt{N_g}}  + \frac{
\gamma G_f}{\tau+\gamma} \\
&\quad +2(\beta+\theta)\sqrt{\tau}+2((2+C_g)\beta_k+2G_f)p_K.
\end{align*}

Using the same argument, for the second moment bound, from the proof of Lemma \ref{lem:bias-ed-gradg}-part(b), we obtain
\begin{align*}
\mathbb E\left[\mathbb E_k\left[\|e_d\|^2\right]\right]&\leq \frac{14\nu_f^2}{N_f} 
+8\left(\beta_k+\frac{2G_f}{\sqrt{\gamma_k}}\right)^2\frac{\nu_g^2}{N_g} 
+ \frac{4\gamma_k^2G_f^2
}{\left(\tau+\gamma_k\right)^2}+6(\beta_k+\theta)^2{\tau} \nonumber\\
&\quad + 4((2+C_g)\beta_k+2G_f)^2\left(\mathbb P(E_k^{(3)})+\frac{\nu_g^2}{(1-\tfrac{1}{\sqrt{2}})^2N_g\tau}\right).
\end{align*}
Taking average over iterations and selecting $\gamma_k=\gamma$ and $\beta_k=\beta$, we obtain
\begin{align*}
\frac{1}{K}\sum_{k=0}^{K-1}\mathbb E\left[\mathbb E_k\left[\|e_d\|^2\right]\right]&\leq \frac{14\nu_f^2}{N_f} 
+8\left(\beta+\frac{2G_f}{\sqrt{\gamma}}\right)^2\frac{\nu_g^2}{N_g} 
+ \frac{4\gamma^2G_f^2
}{\left(\tau+\gamma\right)^2}\nonumber\\
&\quad +6(\beta+\theta)^2{\tau} + 4((2+C_g)\beta+2G_f)^2p_K.
\end{align*}

Following the same line of proof as in Theorem \ref{thm:theorem1}, from \eqref{eq:bound-G-error}, we have that
\begin{align*}
\frac{1}{K}\sum_{k=0}^{K-1}\mathbb E[\mathcal G_k]
&\le
\frac{4\Delta f+8\beta\Delta g}{\eta K}
+\frac{5\Delta g}{2L_g\eta^2 K}
+4G_f^2\beta L_g\eta \nonumber\\
&\quad +
\left(4G_f\right)
\frac{1}{K}\sum_{k=0}^{K-1}\mathbb E\left[\norm*{\mathbb E_k \left[e_d \right]}\right] \nonumber\\
&\quad +
\Big(4L_g\eta\beta+2L_f\eta+\frac{1}{4L_g\eta\beta}+\frac{17}{4}\Big)
\frac{1}{K}\sum_{k=0}^{K-1}\mathbb E\left[\norm*{\mathbb E_k \left[e_d \right]}^2\right] \nonumber \\
&\quad +
\Big(4L_g\eta\beta+2L_f\eta+\tfrac14\Big)
\frac{1}{K}\sum_{k=0}^{K-1}\mathbb E\left[\norm*{e_d}^2 \right] \nonumber \\
&\quad +
\frac{4G_f}{K}
\sum_{k=0}^{K-1}
\sqrt{\left(1+L_g\eta\beta \right)\mathbb E\left[\norm*{\mathbb E_k\left[e_d \right]}^2\right] + L_g\eta\beta \mathbb E\left[\norm*{e_d}^2 \right]}.
\end{align*}
Selecting the parameters as $N_f=K^{a_f}$, $N_g=K^{a_g}\tau^{-1}$, $\eta=K^{-s}\in (0,1)$, $\gamma_k=K^{-r}\tau\in (0,1)$, $\theta=K^{-q}$, $\beta=K^{-q}\in (0,1)$, defining $\alpha=\min\{a_f/2,(a_g-r)/2,r\}\geq 0$, and noting that $p_K=\mathcal O(K^{-\varsigma}+K^{-a_g})=\mathcal O(K^{-\varsigma})$ whenever $a_g\geq \varsigma$, and replacing them into \eqref{eq:bound-G-error} we obtain
\begin{align*}
&\frac{1}{K}\sum_{k=0}^{K-1}\mathbb E[\mathcal G_k]=\mathcal O\Big(K^{2s-1}+K^{-(q+s)}+K^{-\alpha}+K^{s+q-2\alpha} + K^{-r}+K^{-q-r/2}+K^{-\varsigma}+K^{s+q-\varsigma}\Big)
\end{align*}
where the last term arises from the coefficient $\frac{1}{4L_g\eta\beta}=\mathcal O(K^{s+q})$ multiplying the averaged squared bias in \eqref{eq:bound-G-error}, whose bad-region contribution is of order $p_K$.
Therefore, letting $\tau=K^{-1/2}$, $s=q=\frac{1}{4}$, $r=\frac{1}{2}$, $a_f=1$, and $a_g=\frac{3}{2}$, we get $\alpha=\frac{1}{2}$, hence,
\begin{align*}
    &\frac{1}{K}\sum_{k=0}^{K-1}\mathbb E\left[\norm*{\mathbf d_{k}}^2\right] = \mathcal{O}\left(K^{-\frac{1}{2}} +K^{\frac{1}{2}-\varsigma}\right),\quad \frac{1}{K}\sum_{k=0}^{K-1}\mathbb E\left[\norm*{\nabla g(\mathbf x_{k})}^2\right] = \mathcal{O}\left(K^{-\frac{1}{2}} +K^{\frac{1}{2}-\varsigma} \right).
\end{align*}
Moreover, for $\varsigma\in (\frac{1}{2},1]$ the term $K^{\frac{1}{2}-\varsigma}$ dominates $K^{-\frac{1}{2}}$ and the above rate reduces to $\mathcal O(K^{-\delta})$ with $\delta=\varsigma-\frac{1}{2}$. Hence, to achieve $(\epsilon,\epsilon)$-stationary solution the required number of sample gradients for the upper-level function can be computed as $\sum_{k=0}^{K-1}N_f=K^{2}= \mathcal O(\epsilon^{-2/\delta})$, and for the lower-level function, we obtain $\sum_{k=0}^{K-1}N_g=K^3= \mathcal O(\epsilon^{-3/\delta})$. Finally, if $\mathcal R_k(\tau,\theta)$ is empty then Assumption \ref{assump:rare-visit} holds with $\varsigma=1$ and $\delta=\frac{1}{2}$, which implies that the number of sample gradients for the upper-level and lower-level functions are $\mathcal O(\epsilon^{-4})$ and $\mathcal O(\epsilon^{-6})$, respectively, otherwise the rate depends on $\varsigma$.
\end{proof}

\section{Oracle Complexity Analysis of Inexact Stochastic Approach}\label{section:inexact}

In this section, we analyze the application of two stochastic primal-dual algorithms for solving subproblem \eqref{eq:QP}, which leads to an inexact stochastic framework. Building on the error analysis of the gap function in \eqref{eq:bound-G-error}, together with the inequality $\|\mathbb  E_k[e_d]\|\leq \sqrt{\mathbb E_k[\|e_d\|^2]}$, we obtain
\begin{align}
\frac{1}{K}\sum_{k=0}^{K-1}\mathbb E[\mathcal G_k]
&\le
\frac{4\Delta f+8\beta\Delta g}{\eta K}
+\frac{5\Delta g}{2L_g\eta^2 K}
+4G_f^2\beta L_g\eta + \frac{4G_f}{K}\sum_{k=0}^{K-1}\sqrt{\mathbb E \left[\norm*{e_d}^2\right]} \nonumber\\
&\quad +
\Big(8L_g\eta\beta+4L_f\eta+\tfrac92+\frac{1}{4L_g\eta\beta}\Big)
\frac{1}{K}\sum_{k=0}^{K-1}\mathbb E\left[\norm*{e_d}^2 \right] \nonumber \\
&\quad +
\frac{4G_f}{K}
\sum_{k=0}^{K-1}
\sqrt{(1+2L_g\eta\beta) \mathbb E\left[\norm*{e_d}^2 \right]}.
\end{align}
By selecting the parameters $\eta=\beta=K^{-1/4}$, it follows that achieving the convergence rate $\mathcal O(K^{-1/2})$ requires the average error accumulation to satisfy $\frac{1}{K}\sum_{k=0}^{K-1}\mathbb E_k[\|e_d\|^2]=\mathcal O(K^{-1})$. 

This result highlights that any subroutine used to solve \eqref{eq:QP} must guarantee $\mathbb E_k[\|e_d\|^2]=\mathcal O(K^{-1})$ in order to ensure $\frac{1}{K}\sum_{k=0}^{K-1}\mathbb E[\mathcal G_k] = \mathcal O(K^{-1/2})$. To the best of our knowledge, the literature on stochastic constrained optimization with strongly convex objectives remains limited, with only two first-order methods currently available \cite{zhao2022accelerated,boob2023stochastic}. In what follows, we briefly review these methods and then specialize them to the subproblem in \eqref{eq:saddle-point}. This allows us to analyze the overall oracle complexity of the resulting inexact approach.

\paragraph{Strongly Convex Minimization with Functional Constrained Optimization.}

In \citep{boob2023stochastic}, the following composite functional constrained problem is considered:
\begin{equation}
    \min_{\mathbf x \in X} \ \psi_0(\mathbf x) := f_0(\mathbf x)+\chi_0(\mathbf x)
    \qquad
    \mathrm{s.t.} \qquad
    \psi_i(\mathbf x) := f_i(\mathbf x)+\chi_i(\mathbf x) \leq 0, \quad i=1,\ldots,m,
    \label{eq:fc_problem}
\end{equation}
where \(X\subseteq \mathbb{R}^n\) is compact and convex with diameter $D_X>0$, \(f_i\) are convex
possibly nonsmooth functions, and \(\chi_i\) are convex simple functions. The objective is assumed to be strongly convex with modulus \(\alpha_0>0\). Both the objective and the constraints are in expectation form and can be accessed through stochastic first-order information. 

Specializing their method to single constraint $m=1$ and smooth setting $\chi_0(\mathbf x)=\chi_i(\mathbf x)=0$, given \(\mathbf x\in X\), the oracle returns
\(G_0(\mathbf x,\xi)\), \(G_1(\mathbf x,\xi)\), and \(F(\mathbf x,\xi)\) satisfying
\begin{align*}
&\mathbb{E}[G_0(\mathbf x,\xi)] = \nabla f_0(\mathbf x), \qquad\qquad\quad
    \mathbb{E}[G_1(\mathbf x,\xi)] = \nabla f_1(\mathbf x), \\
&\mathbb{E}[F(\mathbf x,\xi)] = f_1(\mathbf x),
    \qquad\qquad\qquad
    \mathbb{E}\!\left[\|G_0(\mathbf x,\xi)-f_0'(\mathbf x)\|_*^2\right] \leq \sigma_0^2,\\
&\mathbb{E}\!\left[\|G_1(\mathbf x,\xi)-f_1'(\mathbf x)\|_*^2\right] \leq \sigma^2,
    \qquad
    \mathbb{E}\!\left[\|F(\mathbf x,\xi)-f_1(\mathbf x)\|_2^2\right] \leq \sigma_f^2 .
\end{align*}
Moreover, $f_0,f_1$ have Lipschitz gradients with constants $L_0,L_f$, respectively. 

In this paper, Constraint Extrapolation (ConEx) method is proposed. Let $
    \widehat{F}_t
    :=
    F(\mathbf x_{t-1},\widehat{\xi}_{t-1})
    +
    G(\mathbf x_{t-1},\widehat{\xi}_{t-1})^\top(\mathbf x_t-\mathbf x_{t-1})$ 
denote a stochastic linearization of the constraint vector \(f(\mathbf x_t)\). Starting from
\((\mathbf x_0,\mathbf y_0)\), with \(\mathbf x_{-1}=\mathbf x_0\), ConEx performs, for
\(t=0,\ldots,T-1\),
\begin{align*}
    s_t
    &=
    (1+\theta_t)\bigl(\widehat{F}_t\bigr)
    -
    \theta_t\bigl(\widehat{F}_{t-1}\bigr),
    \\
    \mathbf y_{t+1}
    &=
    \arg\min_{y\geq 0}
    \left\{
        -\langle s_t,\mathbf y\rangle
        +
        \frac{\tau_t}{2}\|\mathbf y-\mathbf y_t\|_2^2
    \right\}
    =
    \left[\mathbf y_t+\frac{1}{\tau_t}s_t\right]_+,
    \\
    \mathbf x_{t+1}
    &= \mathbf x_t - \eta \left(G_0(\mathbf x_t,\xi_t) + \mathbf y_{t+1} G_1(\mathbf x_t,\xi_t)\right).
\end{align*}

When the objective function is a strongly convex objective, Theorem~1 in their paper establishes the following convergence rate result 
\begin{equation}\label{eq:rate-conex}
\mathbb{E}\!\left[\|\mathbf x^\star-\mathbf x_T\|^2\right] \leq \mathcal O\left(\frac{t_0^2D_X^2+t_0|\mathbf y^*|^2}{T^2}+\frac{t_0 |\mathbf y^*|^2\sigma^2}{T}+\frac{\sqrt{t_0}\sigma_X |\mathbf y^*|^2}{B\sqrt{T}}\right)
\end{equation}
where $t_0=\frac{4(L_0+BL_f)}{\alpha_0}$, $\sigma_X=(\sigma_f^2+D_X^2\sigma^2)^{1/2}$, $\mathbf y^*$ denotes the dual optimal solution, and $B$ denotes the dual optimal bound, i.e., $|\mathbf y^*|\leq B$.

\paragraph{Complexity analysis of ConEx as a subroutine of DBGD.}

We now specialize the previous result to the subproblem \eqref{eq:QP}. In this setting, the dual optimal solution satisfies $\mathbf y^*=\lambda_k$, and the problem-dependent quantities can be identified as 
\begin{align*}
B=\beta+G_f\|\nabla g(\mathbf x_k)\|^{-1},\quad \sigma_f^2=\nu_g^2\|\mathbf d\|^2+2\beta_k^2\nu_g^2(\|\nabla g(\mathbf x_k)\|^2+\mathbb E_k[\|\nabla \tilde g(\mathbf x_k,\xi_g)\|^2]),\quad \sigma^2=\nu_g^2.
\end{align*}
Focusing on the dominant term in \eqref{eq:rate-conex}, we characterize the number of inner iterations required to control the inexactness error and derive the resulting overall sample complexity.

\begin{theorem}
Consider the algorithm update $\mathbf x_{k+1}=\mathbf x_k - \eta_k \tilde{\mathbf d}_k$ where $\tilde{\mathbf d}_k$ is an approximate solution of subproblem \eqref{eq:saddle-point} by implementing ConEx method with $T_k=\mathcal O(\|\nabla g(\mathbf x_k)\|^{-3} K^2)$ inner iterations. Let $\eta_k=\eta=K^{-1/4}$ and $\beta_k=\beta=K^{-1/4}$, then 
$$\frac{1}{K}\sum_{k=0}^{K-1}\mathbb E\left[\norm*{\mathbf d_{k}}^2\right] =\mathcal O\left(K^{-\frac{1}{2}}\right)
    ,\quad \frac{1}{K}\sum_{k=0}^{K-1}\mathbb E\left[\norm*{\nabla g(\mathbf x_{k})}^2\right] = \mathcal{O}\left(K^{-\frac{1}{2}}\right).$$
Moreover, after $K$ iterations overall sample complexity of the inexact stochastic primal-dual method using ConEx method is $\mathcal O(K^2\sum_{k=0}^{K-1}\|\nabla g(\mathbf x_k)\|^{-3})$. 
\end{theorem}
\begin{proof}
Based on the convergence rate of ConEx in \eqref{eq:rate-conex}, we have that at each iteration $k$, $\mathbb E[\|e_d\|^2]=\mathbb E[\|\mathbf d_k - \tilde{\mathbf d}_k\|^2]=\mathcal O(K^{-1})$ requires $T_k=\mathcal O(B^3 K^2)=\mathcal O(\|\nabla g(\mathbf x_k)\|^{-3} K^2)$ inner iterations. Therefore, the total sample gradients of the inexact stochastic primal-dual method using ConEx after $K$ iterations can be computed as $\sum_{k=0}^{K-1}T_k=\mathcal O(K^2\sum_{k=0}^{K-1}\|\nabla g(\mathbf x_k)\|^{-3})$. 
\end{proof}

\paragraph{Stochastic Restart for Strongly Convex Convex-Concave Saddle-Point Problems \cite{zhao2022accelerated}.}

In \cite{zhao2022accelerated}, the following stochastic convex--concave saddle-point problem is studied:
\[
    \min_{\mathbf x\in X}\max_{\mathbf y\in Y}
    \left\{
        S(\mathbf x,\mathbf y) := f(\mathbf x)+g(\mathbf x)+\Phi(\mathbf x,\mathbf y)-h(\mathbf y)
    \right\},
\]
where $X$ and $Y$ are convex and compact sets with diameters $\Omega_x,\Omega_y$, respectively, $f$ is $L$-smooth and
$\mu$-strongly convex with $\mu>0$, $g$ and $h$ admit tractable
Bregman proximal projections, and $\Phi$ is smooth convex--concave
with Lipschitz parameters $L_{xx},L_{yx},L_{yy}$. The algorithm has
access only to unbiased stochastic first-order oracles for
$\nabla f$, $\nabla_x\Phi$, and $\nabla_y\Phi$, with variance bounds $\sigma_{x,f}^2$, $\sigma_{x,\Phi}^2$, and $\sigma_{y,\Phi}^2$.
The goal is to find a primal--dual pair with a small duality gap  $G(\mathbf x,\mathbf y)
    \triangleq 
    \sup_{\mathbf y'\in Y} S(\mathbf x,\mathbf y') - \inf_{\mathbf x'\in X} S(\mathbf x',\mathbf y)$.

The proposed method is a multi-stage stochastic restart scheme,
Algorithm~2S in the paper. At each stage $k$, a localized variant of the stochastic primal-dual hybrid gradient subroutine, Algorithm~1R, is run on the restricted primal set $X_k \triangleq \{\mathbf x\in X:\|\mathbf x-\mathbf x_k\|\le R_k/2\}$, using a rescaled Bregman distance. After $J_k$ iterations, the method sets $\mathbf x_{k+1}$ to the output primal iterate and shrinks the radius as $R_{k+1} = \frac{R_k}{\sqrt{2}}$. 

From Theorem~4.3 in their paper and the discussions therein, we can conclude Stochastic Restart Primal-dual (SRPD) method has the following convergence rate guarantee:
\begin{align}\label{eq:restart}
\mathbb{E}\!\left[\|\mathbf x^\star-\mathbf x_T\|^2\right] \leq \mathcal O\left(\frac{L_{yx}\Omega_y^2}{T^2}+\frac{L_{yy}{\Omega_{y}^2}}{T}+\frac{(\sigma_{x,f}+\sigma_{x,\Phi})^2}{\mu T}+\frac{\sigma_{y,\Phi}\Omega_{y}\log(T)}{\sqrt{T}}\right).
\end{align}

\paragraph{Complexity analysis of SRPD as a subroutine of DBGD.} 
We now specialize the previous result to the subproblem \eqref{eq:saddle-point}. In this setting, $g\equiv 0$, $h\equiv 0$, $f=\frac{1}{2}\|d-\nabla f(\mathbf x_k)\|^2$, and $\Phi(\mathbf x,\mathbf y)=-\lambda(\nabla g(\mathbf x_k)^\top \mathbf d - \phi(\mathbf x_k))$ then the problem-dependent quantities can be identified as 
\begin{align*}
&\Omega_y=\beta+G_f\|\nabla g(\mathbf x_k)\|^{-1},\quad \sigma_{x,\Phi}^2=\lambda^2\nu_g^2,\\ &\sigma_{y,\Phi}^2=\nu_g^2\|\mathbf d\|^2+2\beta_k^2\nu_g^2(\|\nabla g(\mathbf x_k)\|^2+\mathbb E_k[\|\nabla \tilde g(\mathbf x_k,\xi_g)\|^2]).
\end{align*}
Focusing on the dominant term in \eqref{eq:restart}, we characterize the number of inner iterations required to control the inexactness error and derive the resulting overall sample complexity.

\begin{theorem}
Consider the algorithm update $\mathbf x_{k+1}=\mathbf x_k - \eta_k \tilde{\mathbf d}_k$ where $\tilde{\mathbf d}_k$ is an approximate solution of subproblem \eqref{eq:saddle-point} by implementing SRPD method with $T_k=\mathcal O(\|\nabla g(\mathbf x_k)\|^{-2} K^2)$ inner iterations. Let $\eta_k=\eta=K^{-1/4}$ and $\beta_k=\beta=K^{-1/4}$, then 
$$\frac{1}{K}\sum_{k=0}^{K-1}\mathbb E\left[\norm*{\mathbf d_{k}}^2\right] =\mathcal O\left(K^{-\frac{1}{2}}\right)
    ,\quad \frac{1}{K}\sum_{k=0}^{K-1}\mathbb E\left[\norm*{\nabla g(\mathbf x_{k})}^2\right] = \mathcal{O}\left(K^{-\frac{1}{2}}\right).$$
Moreover, after $K$ iterations overall sample complexity of the inexact stochastic primal-dual method using SRPD method is $\mathcal O(K^2\sum_{k=0}^{K-1}\|\nabla g(\mathbf x_k)\|^{-2})$.
\end{theorem}
\begin{proof}
Based on the convergence rate of SRPD in \eqref{eq:restart}, we have that at each iteration $k$, $\mathbb E[\|e_d\|^2]=\mathbb E[\|\mathbf d_k - \tilde{\mathbf d}_k\|^2]=\mathcal O(K^{-1})$ requires $T_k=\mathcal O(\|\nabla g(\mathbf x_k)\|^{-2} K^2)$ inner iterations. Therefore, the total sample gradients for upper- and lower-level objectives of the inexact stochastic method can be computed as $\sum_{k=0}^{K-1}T_k=\mathcal O(K^2\sum_{k=0}^{K-1}\|\nabla g(\mathbf x_k)\|^{-2})$. 
\end{proof}

\begin{remark}
Comparing our proposed method in Algorithm \ref{alg:alg1} with the inexact stochastic approaches, we observe several notable advantages. First, our method employs a simpler update rule, operating within a single-loop structure, in contrast to the nested procedures required by inexact stochastic counterparts. 
More importantly, our method achieves substantially improved sample complexity. Specifically, it reduces the best sample complexities of inexact approaches from $\mathcal O(K^2\sum_{k=0}^{K-1}\|\nabla g(\mathbf x_k)\|^{-2})$ for both the upper- and lower-level problems to $\mathcal O(K^2)$ and $\mathcal O(K^{3/2}\sum_{k=0}^{K-1}\|\nabla g(\mathbf x_k)\|^{-2})$, respectively.
\end{remark}

\section{Proofs for Penalty-Regularized SDBPG}\label{sec:pr-sdbpg-proof}

\begin{proof}[Proof of Lemma~\ref{lem:pr-lambda-bounded}]
The numerator of $\hat{\lambda}_k$ satisfies $|\beta\sigma_k^2 - \nabla g(\mathbf{x}_k)^\top \nabla f(\mathbf{x}_k)| \leq \beta\sigma_k^2 + G_f\|\nabla g(\mathbf{x}_k)\|$, where $\sigma_k^2 = \|\nabla g(\mathbf{x}_k)\|^2 + \gamma$. Splitting:
\[
\frac{\mu\beta\sigma_k^2}{(1+\mu)\|\nabla g(\mathbf{x}_k)\|^2 + \gamma} \leq \frac{\mu\beta(\|\nabla g(\mathbf{x}_k)\|^2 + \gamma)}{(1+\mu)\|\nabla g(\mathbf{x}_k)\|^2 + \gamma} \leq \mu\beta,
\]
and by AM-GM, $(1+\mu)\|\nabla g(\mathbf{x}_k)\|^2 + \gamma \geq 2\sqrt{(1+\mu)\gamma}\|\nabla g(\mathbf{x}_k)\|$, so
\[
\frac{\mu G_f\|\nabla g(\mathbf{x}_k)\|}{(1+\mu)\|\nabla g(\mathbf{x}_k)\|^2 + \gamma} \leq \frac{\mu G_f}{2\sqrt{(1+\mu)\gamma}}.
\]
Combining gives $|\hat{\lambda}_k| \leq \mu\beta + \frac{\mu G_f}{2\sqrt{(1+\mu)\gamma}} = C_\Lambda$.
\end{proof}

We now present the key auxiliary lemmas required for the proof of Theorem~\ref{thm:pr-sdbpg}. Due to the uniform boundedness of the regularized multiplier (Lemma~\ref{lem:pr-lambda-bounded}), all bounds hold \emph{unconditionally} without requiring Assumption~\ref{assump:rare-visit}.

\begin{lemma}\label{lem:G1}
Define
$\tilde\beta:=\frac{\mu\beta}{2(1+\mu)}$, and $\Delta_0:=\frac{(1+\mu)G_f^2}{2\mu\beta}$.
Then, for all \(k\),
\[
\nabla g(\mathbf x_k)^\top \mathbf d_k
\ge
\tilde\beta\|\nabla g(\mathbf x_k)\|^2-\Delta_0.
\]
\end{lemma}

\begin{proof}
For simplicity, write
$f_k:=\nabla f(\mathbf x_k),
g_k:=\nabla g(\mathbf x_k),
s_k:=\|g_k\|^2,
a_k:=g_k^\top f_k$.
The deterministic PR-SDBPG direction is
$\mathbf d_k=f_k+\hat\lambda_k g_k$,
where
$\hat\lambda_k
=
\left[
\frac{
\mu\left(\beta(s_k+\gamma)-a_k\right)
}{
(1+\mu)s_k+\gamma
}
\right]_+$ .
We consider two cases.

\textbf{Case 1: clipping is active.}
In this case \(\hat\lambda_k=0\), so \(\mathbf d_k=f_k\). The clipping condition implies
$a_k\ge \beta(s_k+\gamma)\ge \beta s_k$.
Therefore,
\[
g_k^\top \mathbf d_k
=
a_k
\ge
\beta s_k
\ge
\tilde\beta s_k-\Delta_0,
\]
because \(\beta\ge \tilde\beta\) and \(\Delta_0\ge 0\).

\textbf{Case 2: clipping is inactive.}
In this case, 
$\hat\lambda_k
=
\frac{
\mu\left(\beta(s_k+\gamma)-a_k\right)
}{
(1+\mu)s_k+\gamma
}$.
Therefore,
\[
g_k^\top \mathbf d_k
=
a_k+\hat\lambda_k s_k.
\]
Substituting the expression for \(\hat\lambda_k\), and rearranging, we obtain
\[
g_k^\top \mathbf d_k
=
\frac{(s_k+\gamma)(a_k+\mu\beta s_k)}
{(1+\mu)s_k+\gamma}.
\]
Using Cauchy--Schwarz and the boundedness of \(\nabla f\), we have
$a_k\ge -G_f\sqrt{s_k}$.
Thus,
\[
g_k^\top \mathbf d_k
\ge
\frac{(s_k+\gamma)(\mu\beta s_k-G_f\sqrt{s_k})}
{(1+\mu)s_k+\gamma}.
\]
Now use
$\frac{s_k+\gamma}{(1+\mu)s_k+\gamma}\ge \frac{1}{1+\mu}$
for the positive term \(\mu\beta s_k\), and
$\frac{s_k+\gamma}{(1+\mu)s_k+\gamma}\le 1$
for the negative term \(-G_f\sqrt{s_k}\). This gives
\[
g_k^\top \mathbf d_k
\ge
\frac{\mu\beta}{1+\mu}s_k
-
G_f\sqrt{s_k}.
\]
By Young's inequality, we have
\[
g_k^\top \mathbf d_k
\ge
\frac{\mu\beta}{2(1+\mu)}s_k
-
\frac{(1+\mu)G_f^2}{2\mu\beta}
=
\tilde\beta s_k-\Delta_0.
\]
This proves the claim.
\end{proof}
\begin{lemma}[Bias of the PR-SDBPG direction]\label{lem:G2}
Suppose Assumptions~\ref{assum:stoch_grad} and~\ref{assump:stoch-bound} hold, and assume
$0<\bar\gamma:=\frac{\gamma}{1+\mu}<1$.
Let
$\mathbf d_k=\nabla f(\mathbf x_k)+\hat\lambda_k\nabla g(\mathbf x_k)$, and
 $\tilde{\mathbf d}_k=\nabla \tilde f_k+\tilde\lambda_k\nabla \tilde g_k$,
where \(\hat\lambda_k\) and \(\tilde\lambda_k\) are defined by the penalty-regularized multiplier in
Algorithm~\ref{alg:pr-sdbpg}. Let \(e_d=\tilde{\mathbf d}_k-\mathbf d_k\). Then
\[
\|\mathbb E_k[e_d]\|
\le
\frac{\mu}{1+\mu}\frac{\nu_f}{\sqrt{N_f}}
+
\frac{\mu}{1+\mu}
\left(
\beta+4\beta\mu+2G_f\sqrt{\frac{1+\mu}{\gamma}}
\right)
\frac{\nu_g}{\sqrt{N_g}} .
\]
\end{lemma}
\begin{proof}
For simplicity, write
$u=\nabla \tilde g_k, v=\nabla \tilde f_k,
w=\nabla g(\mathbf x_k), z=\nabla f(\mathbf x_k)$.
Recall that the stochastic PR-SDBPG multiplier is
$\tilde\lambda_k
=
\left[
\frac{
\mu\left(\beta(\|u\|^2+\gamma)-u^\top v\right)
}{
(1+\mu)\|u\|^2+\gamma
}
\right]_+$ .
Define
$\bar\gamma=\frac{\gamma}{1+\mu}$, and
$c=\beta\gamma-\beta\bar\gamma
=
\frac{\beta\mu\gamma}{1+\mu}$.
Then, we can show
\[
\tilde\lambda_k 
=
\frac{\mu}{1+\mu}
\left[
\beta
-
\frac{u^\top v-c}{\|u\|^2+\bar\gamma}
\right]_+.
\]
Multiplying both sides by $u$ and defining $\Phi_{\bar\gamma}(u,v)
=
\left[
\beta
-
\frac{u^\top v-c}{\|u\|^2+\bar\gamma}
\right]_+u$, we have
\[
\tilde\lambda_k u
=
\frac{\mu}{1+\mu}
\Phi_{\bar\gamma}(u,v),
\]
Similarly,
$\hat\lambda_k w
=
\frac{\mu}{1+\mu}
\Phi_{\bar\gamma}(w,z)$.

Now decompose
$e_d
=
\nabla \tilde f_k-\nabla f(\mathbf x_k)
+
\tilde\lambda_k\nabla \tilde g_k
-
\hat\lambda_k\nabla g(\mathbf x_k)$.
Taking conditional expectation with respect to \(\mathcal F_k\), and using the unbiasedness of
\(\nabla \tilde f_k\), we obtain
\[
\|\mathbb E_k[e_d]\|
\le
\frac{\mu}{1+\mu}
\mathbb E_k
\left[
\|\Phi_{\bar\gamma}(u,v)-\Phi_{\bar\gamma}(w,z)\|
\right].
\]

By Lemma~\ref{lem:Lip-Phi} and Assumptions~\ref{assump:stoch-bound} and \ref{assum:stoch_grad}, 

\[
\|\mathbb E_k[e_d]\|
\le
\frac{\mu}{1+\mu}\frac{\nu_f}{\sqrt{N_f}}
+
\frac{\mu}{1+\mu}
\left(
\beta+4\beta\mu
+
2G_f\sqrt{\frac{1+\mu}{\gamma}}
\right)
\frac{\nu_g}{\sqrt{N_g}}.
\]
This proves the desired result.
\end{proof}
\begin{lemma}[Second moment of the PR-SDBPG direction error]\label{lem:G3}
Suppose Assumptions~\ref{assum:stoch_grad}, \ref{assum:bounded-stoch-g}, and~\ref{assump:stoch-bound} hold, and assume
$0<\bar\gamma:=\frac{\gamma}{1+\mu}<1$.
Let
$\mathbf d_k=\nabla f(\mathbf x_k)+\hat\lambda_k\nabla g(\mathbf x_k)$ and $
\tilde{\mathbf d}_k=\nabla \tilde f_k+\tilde\lambda_k\nabla \tilde g_k$,
and define \(e_d=\tilde{\mathbf d}_k- \mathbf d_k\). Let
$A_\gamma
=
\beta+4\beta\mu+2G_f\sqrt{\frac{1+\mu}{\gamma}}$.
Then,
\[
\mathbb E_k[\|e_d\|^2]
\le
\left(
2+\frac{4\mu^2}{(1+\mu)^2}
\right)\frac{\nu_f^2}{N_f}
+
\frac{4\mu^2}{(1+\mu)^2}
A_\gamma^2
\frac{\nu_g^2}{N_g}.
\]
\end{lemma}
\begin{proof}
For simplicity, write
$u=\nabla \tilde g_k,
v=\nabla \tilde f_k,
w=\nabla g(\mathbf x_k),
z=\nabla f(\mathbf x_k)$.
As in Lemma~\ref{lem:G2}, define
$\bar\gamma=\frac{\gamma}{1+\mu}$ and $
c=\beta\gamma-\beta\bar\gamma
=
\frac{\beta\mu\gamma}{1+\mu}$.
Then, defining $\Phi_{\bar\gamma}(u,v)
=
\left[
\beta-
\frac{u^\top v-c}{\|u\|^2+\bar\gamma}
\right]_+u $, the penalty-regularized dual-gradient product can be written as
\[
\tilde\lambda_k u
=
\frac{\mu}{1+\mu}\Phi_{\bar\gamma}(u,v),
\qquad
\hat\lambda_k w
=
\frac{\mu}{1+\mu}\Phi_{\bar\gamma}(w,z),
\]

Therefore,
\[
e_d
=
v-z
+
\frac{\mu}{1+\mu}
\left(
\Phi_{\bar\gamma}(u,v)-\Phi_{\bar\gamma}(w,z)
\right).
\]
Using \(\|a+b\|^2\le 2\|a\|^2+2\|b\|^2\), we obtain
\[
\mathbb E_k[\|e_d\|^2]
\le
2\mathbb E_k[\|v-z\|^2]
+
2\left(\frac{\mu}{1+\mu}\right)^2
\mathbb E_k
\left[
\|\Phi_{\bar\gamma}(u,v)-\Phi_{\bar\gamma}(w,z)\|^2
\right].
\]

Next, using \((a+b)^2\le 2a^2+2b^2\) on Lemma~\ref{lem:Lip-Phi}, and taking conditional expectation along with Assumptions \ref{assum:stoch_grad} and \ref{assum:bounded-stoch-g}, we can show


\[
\mathbb E_k
\left[
\|\Phi_{\bar\gamma}(u,v)-\Phi_{\bar\gamma}(w,z)\|^2
\right]
\le
2A_\gamma^2\frac{\nu_g^2}{N_g}
+
2\frac{\nu_f^2}{N_f},
\]
where $A_\gamma \triangleq\beta+4\beta\mu
+
2G_f\sqrt{\frac{1+\mu}{\gamma}}$. Substituting this into the bound for \(\mathbb E_k[\|e_d\|^2]\), we obtain
\[
\mathbb E_k[\|e_d\|^2]
\le
\left(
2+\frac{4\mu^2}{(1+\mu)^2}
\right)
\frac{\nu_f^2}{N_f}
+
\frac{4\mu^2}{(1+\mu)^2}
A_\gamma^2
\frac{\nu_g^2}{N_g}.
\]
This proves the claim.
\end{proof}

\begin{proof}[Proof of Theorem~\ref{thm:pr-sdbpg}]
We prove the result in two steps. First, we show that the Lyapunov descent controls the
regularized DBGD direction \(\mathbf d_k\). Second, we use a separate descent argument on \(g\)
to control \(\|\nabla g(\mathbf x_k)\|^2\).

Throughout the proof, we denote
$f_k:=\nabla f(\mathbf  x_k), g_k:=\nabla g(\mathbf x_k),
s_k:=\|g_k\|^2, a_k:=g_k^\top f_k$.

From Lemma~\ref{lem:pr-lambda-bounded}, the multiplier is uniformly bounded, i.e, $0\le \hat\lambda_k\le C_\Lambda$.
Define
$c:=C_\Lambda+1$.
Then
$c-\hat\lambda_k\ge 1$.
We also define the Lyapunov function
$\Phi_k:=f(\mathbf x_k)+c\,g(\mathbf x_k)$,
and the constants
$L_\Phi:=L_f+cL_g$ and $
G_\Phi:=G_f+cC_g$.

By Assumption~\ref{assum:bounded-stoch-g} and Jensen's inequality,
\[
\|\nabla f(\mathbf x_k)+c\nabla g(\mathbf x_k)\|\le G_\Phi.
\]

We first prove a deterministic cross-term bound. Let
$y_k:=g_k^\top \mathbf d_k$.
We claim that
\[
(c-\hat\lambda_k)y_k
\ge
-\delta_c,
\quad \text{where}\quad
\delta_c
:=
\left(C_\Lambda+1-\mu\beta\right)\frac{G_f^2}{4\mu\beta}.
\]
Indeed, if \(y_k\ge 0\), then since \(c-\hat\lambda_k\ge 0\), we immediately have
$(c-\hat\lambda_k)y_k\ge 0$.
Now suppose \(y_k<0\). In this case, clipping cannot be active. To see this, if
\(\hat\lambda_k=0\), then \(\mathbf d_k=f_k\), and the clipping condition implies
\[
a_k=g_k^\top f_k>\beta(s_k+\gamma)\ge 0.
\]
Thus \(y_k=g_k^\top \mathbf d_k=a_k>0\), which contradicts \(y_k<0\). Therefore,
\(\hat\lambda_k=\lambda_k^{\rm reg}\). 

Using \(\mathbf d_k=f_k+\hat\lambda_k g_k\), we obtain
\[
y_k
=
g_k^\top \mathbf d_k
=
a_k+\hat\lambda_k s_k.
\]
Substituting the expression for \(\hat\lambda_k\), and rearranging, we get
\[
y_k
=
\frac{
(s_k+\gamma)(a_k+\mu\beta s_k)
}{
(1+\mu)s_k+\gamma
}.
\]
Since \(y_k<0\), it follows that
$a_k+\mu\beta s_k<0$.
Using this inequality in the expression for \(\hat\lambda_k\), we have
\[
\hat\lambda_k
=
\frac{
\mu\left(\beta(s_k+\gamma)-a_k\right)
}{
(1+\mu)s_k+\gamma
}
>
\frac{
\mu\left(\beta(s_k+\gamma)+\mu\beta s_k\right)
}{
(1+\mu)s_k+\gamma
}.
\]
Since
$\beta(s_k+\gamma)+\mu\beta s_k
=
\beta\left((1+\mu)s_k+\gamma\right)$,
we conclude that
$\hat\lambda_k>\mu\beta$.
Therefore, in the case \(y_k<0\),
\[
0\le c-\hat\lambda_k\le c-\mu\beta
=
C_\Lambda+1-\mu\beta.
\]
Next, using \(a_k\ge -G_f\sqrt{s_k}\), we have
\[
y_k
=
\frac{
(s_k+\gamma)(a_k+\mu\beta s_k)
}{
(1+\mu)s_k+\gamma
}
\ge
a_k+\mu\beta s_k
\ge
-G_f\sqrt{s_k}+\mu\beta s_k,
\]
where we used
$\frac{s_k+\gamma}{(1+\mu)s_k+\gamma}\le 1$
together with \(a_k+\mu\beta s_k<0\). Therefore,
\[
-y_k
\le
G_f\sqrt{s_k}-\mu\beta s_k
\le
\max_{r\ge 0}\left\{G_f r-\mu\beta r^2\right\}
=
\frac{G_f^2}{4\mu\beta}.
\]
Combining the previous two bounds gives
\[
(c-\hat\lambda_k)y_k
\ge
-\left(C_\Lambda+1-\mu\beta\right)\frac{G_f^2}{4\mu\beta}
=
-\delta_c.
\]

We now turn to the Lyapunov descent. Since \(f\) and \(g\) have Lipschitz gradients,
$\Phi(\mathbf x)=f(\mathbf x)+c g(\mathbf x)$
has \(L_\Phi=L_f+cL_g\)-Lipschitz gradient. Using
$\mathbf x_{k+1}=\mathbf x_k-\eta\tilde{\mathbf d}_k=\mathbf x_k-\eta(\mathbf d_k+e_d)$ and \(\|\nabla\Phi(\mathbf x_k)\|\le G_\Phi\),
we obtain
\[
\mathbb E_k[\Phi_k-\Phi_{k+1}]
\ge
\eta \nabla\Phi(\mathbf x_k)^\top \mathbf d_k
-
\eta G_\Phi\|\mathbb E_k[e_d]\|
-
\frac{L_\Phi\eta^2}{2}
\mathbb E_k\|\mathbf d_k+e_d\|^2.
\]
Moreover,
$\nabla\Phi(\mathbf x_k)^\top \mathbf d_k
=
(f_k+c g_k)^\top \mathbf d_k$.
Therefore,
\[
(f_k+c g_k)^\top \mathbf d_k
=
(\mathbf d_k-\hat\lambda_k g_k+c g_k)^\top \mathbf d_k
=
\|\mathbf d_k\|^2+(c-\hat\lambda_k)g_k^\top \mathbf d_k.
\]
By the cross-term bound above,
\[
(f_k+c g_k)^\top \mathbf d_k
\ge
\|\mathbf d_k\|^2-\delta_c.
\]
and,
\[
\mathbb E_k\|\mathbf d_k+e_d\|^2
\le
2\|\mathbf d_k\|^2+2\mathbb E_k\|e_d\|^2.
\]
Thus,
\[
\mathbb E_k[\Phi_k-\Phi_{k+1}]
\ge
\eta\|\mathbf d_k\|^2
-\eta\delta_c
-\eta G_\Phi\|\mathbb E_k[e_d]\|
-
L_\Phi\eta^2\|\mathbf d_k\|^2
-
L_\Phi\eta^2\mathbb E_k\|e_d\|^2.
\]
Choosing $\eta\le \frac{1}{2L_\Phi}$, and rearranging gives
\[
\|\mathbf d_k\|^2
\le
\frac{2}{\eta}\mathbb E_k[\Phi_k-\Phi_{k+1}]
+
2\delta_c
+
2G_\Phi\|\mathbb E_k[e_d]\|
+
2L_\Phi\eta\mathbb E_k\|e_d\|^2.
\]

Taking total expectation, summing over \(k=0,\ldots,K-1\), dividing by \(K\), and defining 
$B_{\max}:=\max_{0\le k\le K-1}\|\mathbb E_k[e_d]\|$ and $
V_{\max}:=\max_{0\le k\le K-1}\mathbb E_k\|e_d\|^2$, and let $\Delta_\Phi:=\Phi_0-\Phi^\star$, where $\Phi^\star:=\inf_{\mathbf x} \Phi(\mathbf x)$, we obtain
\[
\frac{1}{K}\sum_{k=0}^{K-1}\mathbb E\|\mathbf d_k\|^2
\le
\frac{2\Delta_\Phi}{\eta K}
+
2\delta_c
+
2G_\Phi B_{\max}
+
2L_\Phi\eta V_{\max}.
\]


By Lemmas~\ref{lem:G2} and~\ref{lem:G3}, there exist absolute constants
\(\bar C_1,\bar C_2,\bar C_3,\bar C_4>0\) such that
\[
B_{\max}
\le
\bar C_1\frac{\nu_f}{\sqrt{N_f}}
+
\bar C_2 A_\gamma \frac{\nu_g}{\sqrt{N_g}},\quad
\text{and}\quad
V_{\max}
\le
\bar C_3\frac{\nu_f^2}{N_f}
+
\bar C_4 A_\gamma^2\frac{\nu_g^2}{N_g}.
\]

Now choose
$\mu=1, \gamma=1, \beta=\frac{4G_f^3}{\epsilon}$.
Then,
$A_\gamma=\mathcal{O}\left(\frac{G_f^3}{\epsilon}\right),
C_\Lambda=\mathcal{O}\left(\frac{G_f^3}{\epsilon}\right),
c=\mathcal{O}\left(\frac{G_f^3}{\epsilon}\right)$.
Moreover,
$G_\Phi=G_f+cC_g
=
\mathcal{O}\left(\frac{G_f^3C_g}{\epsilon}\right)$ and $
L_\Phi=L_f+cL_g$.
The cross-term constant satisfies
\[
\delta_c
=
\left(C_\Lambda+1-\mu\beta\right)\frac{G_f^2}{4\mu\beta}
=
\mathcal{O}(\epsilon).
\]
Choose the batch sizes so that
$G_\Phi B_{\max}\le \mathcal{O}(\epsilon)$.
It is sufficient to take
$N_f
=
\mathcal{O}\left(
\frac{\nu_f^2G_\Phi^2}{\epsilon^2}
\right)$ and $
N_g
=
\mathcal{O}\left(
\frac{A_\gamma^2\nu_g^2G_\Phi^2}{\epsilon^2}
\right)$.
Under the above parameter choices, this becomes
\[
N_f
=
\mathcal{O}\left(
\frac{\nu_f^2G_f^6C_g^2}{\epsilon^4}
\right),
\qquad
N_g
=
\mathcal{O}\left(
\frac{G_f^{12}\nu_g^2C_g^2}{\epsilon^6}
\right).
\]
With these choices,
$B_{\max}=\mathcal{O}\left(\frac{\epsilon}{G_\Phi}\right)$ and 
$V_{\max}
=
\mathcal{O}\left(\frac{\epsilon^2}{G_\Phi^2}\right)$. Since \(\eta\le 1/(2L_\Phi)\), thus,
$L_\Phi\eta V_{\max}=\mathcal{O}(\epsilon)$.
Therefore,
\[
\frac{1}{K}\sum_{k=0}^{K-1}\mathbb E\|\mathbf d_k\|^2
\le
\frac{2\Delta_\Phi}{\eta K}+\mathcal{O}(\epsilon).
\]
Choosing
$K
=
\mathcal{O}\left(\frac{\Delta_\Phi}{\eta\epsilon}\right)$, for \(k^*\sim \mathrm{Unif}\{0,\ldots,K-1\}\),

\[
\mathbb E\|\mathbf d_{k^*}\|^2\le \mathcal{O}(\epsilon).
\]

It remains to control \(\|\nabla g(\mathbf x_k)\|^2\). By Assumption \ref{assump:gradf-g-lip},
\[
\mathbb E_k[g(\mathbf x_{k+1})]
\le
g(\mathbf x_k)
-\eta g_k^\top \mathbf d_k
-\eta g_k^\top \mathbb E_k[e_d]
+
\frac{L_g\eta^2}{2}\mathbb E_k\|\mathbf d_k+e_d\|^2.
\]
Using Lemma~\ref{lem:G1},
and applying
\[
|g_k^\top \mathbb E_k[e_d]|
\le
C_g\|\mathbb E_k[e_d]\|
\le
C_g B_{\max}.
\]
and,
\[
\mathbb E_k\|\mathbf d_k+e_d\|^2
\le
2\|\mathbf d_k\|^2+2\mathbb E_k\|e_d\|^2.
\]
Therefore, rearranging gives
\[
\tilde\beta\|g_k\|^2
\le
\frac{1}{\eta}\mathbb E_k[g(\mathbf x_k)-g(\mathbf x_{k+1})]
+
\Delta_0
+
C_gB_{\max}
+
L_g\eta\|\mathbf d_k\|^2
+
L_g\eta\mathbb E_k\|e_d\|^2.
\]
Taking total expectation, summing over \(k=0,\ldots,K-1\), dividing by \(K\tilde\beta\), and defining $\Delta_g:=g(\mathbf x_0)-g^\star$ we obtain
\[
\frac{1}{K}\sum_{k=0}^{K-1}\mathbb E\|g_k\|^2
\le
\frac{\Delta_g}{\eta K\tilde\beta}
+
\frac{\Delta_0}{\tilde\beta}
+
\frac{C_gB_{\max}}{\tilde\beta}
+
\frac{L_g\eta}{\tilde\beta}
\left(
\frac{1}{K}\sum_{k=0}^{K-1}\mathbb E\|\mathbf d_k\|^2
+
V_{\max}
\right).
\]

With \(\mu=1\), we have
$\tilde\beta=\frac{\beta}{4}=\mathcal{O}(\frac{G_f^3}{\epsilon})$.
Moreover,
$\Delta_0=\frac{G_f^2}{\beta}=\mathcal{O}(\epsilon)$,
and therefore
$\frac{\Delta_0}{\tilde\beta}=\mathcal{O}(\epsilon^2)$.
Using the chosen batch sizes,
$\frac{C_gB_{\max}}{\tilde\beta}=\mathcal{O}(\epsilon)$.
Also, since
$\frac{1}{K}\sum_{k=0}^{K-1}\mathbb E\|\mathbf d_k\|^2\le \mathcal{O}(\epsilon)$,
 and \(V_{\max}\) is controlled by the chosen batch sizes, we get
\[
\frac{L_g\eta}{\tilde\beta}
\left(
\frac{1}{K}\sum_{k=0}^{K-1}\mathbb E\|\mathbf d_k\|^2
+
V_{\max}
\right)
\le
\mathcal{O}(\epsilon).
\]
Finally, choosing
$K
=
\mathcal{O}\left(\frac{\Delta_g}{\eta\tilde\beta\epsilon}\right)$
ensures
$\frac{\Delta_g}{\eta K\tilde\beta}\le \mathcal{O}(\epsilon)$.
Hence,
\[
\frac{1}{K}\sum_{k=0}^{K-1}\mathbb E\|\nabla g(\mathbf x_k)\|^2
\le
\mathcal{O}(\epsilon).
\]
Therefore, for \(k^*\sim\mathrm{Unif}\{0,\ldots,K-1\}\),
\[
\mathbb E\|\nabla g(\mathbf x_{k^*})\|^2\le \mathcal{O}(\epsilon).
\]

Combining the two bounds, we obtain
\[
\mathbb E\|\mathbf d_{k^*}\|^2\le \mathcal{O}(\epsilon),
\qquad
\mathbb E\|\nabla g(\mathbf x_{k^*})\|^2\le \mathcal{O}(\epsilon).
\]

The required number of iterations is
\[
K
=
\mathcal{O}\left(
\frac{\Delta_\Phi}{\eta\epsilon}
+
\frac{\Delta_g}{\eta\tilde\beta\epsilon}
\right).
\]
Using $\eta=\frac{\epsilon}{5G_f^3L_\Phi}$, $
\tilde\beta=\frac{G_f^3}{\epsilon}$,
$L_\Phi=L_f+cL_g$, $\Delta_\Phi=\mathcal O(c)$, and $
c=\mathcal{O}\left(\frac{G_f^3}{\epsilon}\right)$,
we obtain, up to problem-dependent constants,
$K=\mathcal{O}(G_f^9L_g\epsilon^{-4})$.
Therefore, 
\[
f-\text{total sample complexity}:K N_f
=
\mathcal{O}(\epsilon^{-4})\cdot \mathcal{O}(\epsilon^{-4})
=
\mathcal{O}(\epsilon^{-8})
\]
\[
g-\text{total sample complexity}:K N_g
=
\mathcal{O}(\epsilon^{-4})\cdot \mathcal{O}(\epsilon^{-6})
=
\mathcal{O}(\epsilon^{-10}).
\]
\end{proof}

\section{Proofs for Variance-Reduced Penalty-Regularized SDBPG}\label{app:vr_pr_sdbpg}


The convergence guarantee of PR-SDBPG in Theorem~\ref{thm:pr-sdbpg} uses large mini-batches to
control the stochastic direction error. In particular, the dominant condition is
$
G_\Phi B_{\max}\lesssim \mathcal{O}(\epsilon),
$,
where \(G_\Phi=G_f+cC_g\) and \(c=C_\Lambda+1\). Since \(c=\mathcal{O}(\epsilon^{-1})\),
the stochastic error is amplified by the Lyapunov weight, leading to the
per-iteration lower-level batch size \(N_g=\mathcal{O}(\epsilon^{-6})\).

In this section, we show that this dependence can be improved by incorporating
STORM-type gradient tracking~\cite{cutkosky2019momentum}. The key reason variance
reduction is effective for PR-SDBPG is that the penalty-regularized direction map
is uniformly Lipschitz: the denominator
$
(1+\mu)\|\nabla g\|^2+\gamma
$
is always lower bounded by \(\gamma>0\). This avoids the degeneracy that appears in
the original SDBPG direction near lower-level stationary points.

\paragraph{Additional assumption.}
We use the following standard mean-square smoothness assumption for stochastic
gradients.

\begin{assumption}[Mean-square smooth stochastic gradients]\label{ass:ms_smooth}
For \(h\in\{f,g\}\), there exists \(\mathcal L_h>0\) such that, for all \(\mathbf x,\mathbf y\),
\[
\mathbb E_\xi
\left[
\|\nabla \tilde h(\mathbf x;\xi)-\nabla \tilde h(\mathbf y;\xi)\|^2
\right]
\le
\mathcal L_h^2\|\mathbf x-\mathbf y\|^2.
\]
\end{assumption}


\begin{algorithm}[t!]
\caption{Variance-Reduced PR-SDBPG}
\label{alg:vr_pr_sdbpg}
\begin{algorithmic}[1]
\State Input: step size \(\eta\), parameters \(\beta,\mu,\gamma\), batch sizes \(N_f,N_g\), STORM parameters \(\alpha_f,\alpha_g\).
\State Initialize \(\mathbf x_0\in\mathbb R^n\).
\State Draw initial batches and set
\[
v_f^0=\frac1{N_f}\sum_{i=1}^{N_f}\nabla \tilde f(\mathbf x_0;\xi_{f,i}^0),
\qquad
v_g^0=\frac1{N_g}\sum_{i=1}^{N_g}\nabla \tilde g(\mathbf x_0;\xi_{g,i}^0).
\]
\For{\(k=0,1,\ldots,K-1\)}
\State Compute the regularized multiplier
\[
\tilde\lambda_k
=
\left[
\frac{
\mu\left(\beta(\|v_g^k\|^2+\gamma)-(v_g^k)^\top v_f^k\right)
}{
(1+\mu)\|v_g^k\|^2+\gamma
}
\right]_+ .
\]
\State Set
\[
\tilde{\mathbf d}_k=v_f^k+\tilde\lambda_k v_g^k.
\]
\State Update
\[
\mathbf x_{k+1}=\mathbf x_k-\eta\tilde{\mathbf d}_k.
\]
\State For each \(h\in\{f,g\}\), draw a fresh independent batch \(\{\zeta_{h,i}^k\}_{i=1}^{N_h}\) and update
\[
v_h^{k+1}
=
\frac1{N_h}\sum_{i=1}^{N_h}\nabla \tilde h(\mathbf x_{k+1};\zeta_{h,i}^k)
+
(1-\alpha_h)
\left[
v_h^k
-
\frac1{N_h}\sum_{i=1}^{N_h}\nabla \tilde h(\mathbf x_k;\zeta_{h,i}^k)
\right].
\]
\EndFor
\end{algorithmic}
\end{algorithm}
The same mini-batch \(\{\zeta_{h,i}^k\}_{i=1}^{N_h}\) is used to evaluate both
\(\nabla \tilde h(\mathbf x_{k+1};\zeta_{h,i}^k)\) and
\(\nabla \tilde h(\mathbf x_k;\zeta_{h,i}^k)\); this shared-sample difference is essential
for the STORM tracking bound.
The update in Line~8 is the standard STORM recursion. Conditional on the past,
\[
\mathbb E_k[v_h^{k+1}-\nabla h(\mathbf x_{k+1})]
=
(1-\alpha_h)(v_h^k-\nabla h(\mathbf x_k)).
\]
This contraction property is used in the tracking analysis below.

\subsection{Lipschitz property of the regularized direction map}

Define, for \(p,q\in\mathbb R^n\),
\[
\Lambda(p,q)
:=
\left[
\frac{
\mu\left(\beta(\|q\|^2+\gamma)-q^\top p\right)
}{
(1+\mu)\|q\|^2+\gamma
}
\right]_+,
\qquad
\Psi(p,q):=p+\Lambda(p,q)q.
\]
Thus, the deterministic direction and the variance-reduced stochastic direction are
\[
\mathbf d_k=\Psi(\nabla f(\mathbf x_k),\nabla g(\mathbf x_k)),
\qquad
\tilde{\mathbf d}_k=\Psi(v_f^k,v_g^k).
\]

\begin{lemma}[Lipschitz property of the PR-SDBPG direction map]\label{lem:psi_lip}
Suppose \(0<\bar\gamma:=\gamma/(1+\mu)<1\). Define
\[
\Lambda(p,q)
:=
\left[
\frac{
\mu\left(\beta(\|q\|^2+\gamma)-q^\top p\right)
}{
(1+\mu)\|q\|^2+\gamma
}
\right]_+,
\qquad
\Psi(p,q):=p+\Lambda(p,q)q.
\]
Assume \(\|p_2\|\le G_f\). Then, for any \(p_1,p_2,q_1,q_2\),
\[
\|\Psi(p_1,q_1)-\Psi(p_2,q_2)\|
\le
L_\Psi^f\|p_1-p_2\|
+
L_\Psi^g\|q_1-q_2\|,
\]
where
$
L_\Psi^f:=1+\frac{\mu}{1+\mu},
$
and
$
L_\Psi^g
:=
\frac{\mu}{1+\mu}
\left(
\beta+4\beta\mu
+
2G_f\sqrt{\frac{1+\mu}{\gamma}}
\right).
$
In particular, for \(\mu=1\) and \(\gamma=1\),
\(
L_\Psi^f=\Theta(1),
L_\Psi^g=\Theta(\beta).
\)
\end{lemma}

\begin{proof}
Let
\(
\bar\gamma=\frac{\gamma}{1+\mu},
c=\beta\gamma-\beta\bar\gamma
=
\frac{\beta\mu\gamma}{1+\mu}.
\)
Then
\[
\Lambda(p,q)
=
\frac{\mu}{1+\mu}
\left[
\beta-
\frac{q^\top p-c}{\|q\|^2+\bar\gamma}
\right]_+.
\]
Multiply both side by $q$, and defining $\Phi_{\bar\gamma}(q,p)
=
\left[
\beta-
\frac{q^\top p-c}{\|q\|^2+\bar\gamma}
\right]_+q$, we obtain,
\[
\Lambda(p,q)q
=
\frac{\mu}{1+\mu}\Phi_{\bar\gamma}(q,p),
\]
Now, we can rewrite $\Psi$ as,
$\Psi(p,q)
=
p+\frac{\mu}{1+\mu}\Phi_{\bar\gamma}(q,p)$.
Therefore,
\[
\begin{aligned}
\|\Psi(p_1,q_1)-\Psi(p_2,q_2)\|
&\le
\|p_1-p_2\|
+
\frac{\mu}{1+\mu}
\|\Phi_{\bar\gamma}(q_1,p_1)-\Phi_{\bar\gamma}(q_2,p_2)\|.
\end{aligned}
\]
By Lemma~\ref{lem:Lip-Phi}, using \(\|p_2\|\le G_f\),
\[
\|\Phi_{\bar\gamma}(q_1,p_1)-\Phi_{\bar\gamma}(q_2,p_2)\|
\le
\left(
\beta+\frac{4c}{\bar\gamma}
+
\frac{2G_f}{\sqrt{\bar\gamma}}
\right)
\|q_1-q_2\|
+
\|p_1-p_2\|.
\]
Since
\(
\frac{c}{\bar\gamma}=\beta\mu,
\frac{1}{\sqrt{\bar\gamma}}
=
\sqrt{\frac{1+\mu}{\gamma}},
\)
we obtain
\[
\|\Phi_{\bar\gamma}(q_1,p_1)-\Phi_{\bar\gamma}(q_2,p_2)\|
\le
\left(
\beta+4\beta\mu
+
2G_f\sqrt{\frac{1+\mu}{\gamma}}
\right)
\|q_1-q_2\|
+
\|p_1-p_2\|.
\]
Substituting this bound into the previous inequality gives
\[
\|\Psi(p_1,q_1)-\Psi(p_2,q_2)\|
\le
\left(1+\frac{\mu}{1+\mu}\right)\|p_1-p_2\|
+
\frac{\mu}{1+\mu}
\left(
\beta+4\beta\mu
+
2G_f\sqrt{\frac{1+\mu}{\gamma}}
\right)
\|q_1-q_2\|.
\]
This proves the claim.
\end{proof}
\subsection{STORM tracking error}

Define the tracking errors 
$e_{h,k}:=v_h^k-\nabla h(\mathbf x_k),
\qquad h\in\{f,g\}$. 
We also define 
$\bar\sigma_h^2
:=
\frac1K\sum_{k=0}^{K-1}\mathbb E\|e_{h,k}\|^2.$

For the variance-reduced method, we let \(\mathcal F_k\) denote the history up to
iteration \(k\), including \(\mathbf x_k,v_f^k,v_g^k\). Conditional expectations
\(\mathbb E_k[\cdot]\) in this section are taken with respect to this filtration.
\begin{lemma}[Average STORM tracking error]\label{lem:storm_tracking}
Suppose Assumption~\ref{ass:ms_smooth} and the bounded-variance oracle assumption hold.
Let \(v_h^k\) be generated by Algorithm~\ref{alg:vr_pr_sdbpg}, and suppose
\(
0<\alpha_h\le 1.
\)
Then
\[
\bar\sigma_h^2
\le
\frac{2\nu_h^2}{\alpha_h K N_h}
+
\frac{4\alpha_h\nu_h^2}{N_h}
+
\frac{4\mathcal L_h^2\eta^2}{\alpha_h N_h}
\left(
\frac1K\sum_{k=0}^{K-1}\mathbb E\|\tilde{\mathbf d}_k\|^2
\right).
\]
In particular, if
\(
\alpha_h=\frac{\mathcal L_h\eta}{\nu_h}\le 1,
\)
then
\[
\bar\sigma_h^2
\le
\frac{2\nu_h^3}{\mathcal L_h\eta K N_h}
+
\frac{4\mathcal L_h\eta\nu_h}{N_h}
+
\frac{4\mathcal L_h\eta\nu_h}{N_h}
\left(
\frac1K\sum_{k=0}^{K-1}\mathbb E\|\tilde{\mathbf d}_k\|^2
\right).
\]
\end{lemma}
\begin{proof}
Let \(h_k:=\nabla h(\mathbf x_k)\). From the STORM update,
\[
v_h^{k+1}
=
\frac1{N_h}\sum_{i=1}^{N_h}\nabla \tilde h(\mathbf x_{k+1};\zeta_{h,i}^k)
+
(1-\alpha_h)
\left[
v_h^k
-
\frac1{N_h}\sum_{i=1}^{N_h}\nabla \tilde h(\mathbf x_k;\zeta_{h,i}^k)
\right].
\]
Subtracting \(h_{k+1}\), we obtain
\[
e_{h,k+1}
=
(1-\alpha_h)e_{h,k}
+
\Delta_{h,k}
+
\alpha_h\xi_{h,k},
\]
where
\[
\Delta_{h,k}
:=
\frac1{N_h}\sum_{i=1}^{N_h}
\left[
\nabla \tilde h(\mathbf x_{k+1};\zeta_{h,i}^k)
-
\nabla \tilde h(\mathbf x_k;\zeta_{h,i}^k)
-
(h_{k+1}-h_k)
\right],
\]
and
\[
\xi_{h,k}
:=
\frac1{N_h}\sum_{i=1}^{N_h}
\left[
\nabla \tilde h(\mathbf x_k;\zeta_{h,i}^k)-h_k
\right].
\]

Both \(\Delta_{h,k}\) and \(\xi_{h,k}\) have conditional mean zero. Hence,
$\mathbb E_k\left[
e_{h,k}^\top(\Delta_{h,k}+\alpha_h\xi_{h,k})
\right]=0$.
Therefore,
\[
\mathbb E_k\|e_{h,k+1}\|^2
=
(1-\alpha_h)^2\|e_{h,k}\|^2
+
\mathbb E_k\|\Delta_{h,k}+\alpha_h\xi_{h,k}\|^2.
\]
Using \((1-\alpha_h)^2\le 1-\alpha_h\) and
$\|\Delta_{h,k}+\alpha_h\xi_{h,k}\|^2
\le
2\|\Delta_{h,k}\|^2+2\alpha_h^2\|\xi_{h,k}\|^2$,
we obtain
\[
\mathbb E_k\|e_{h,k+1}\|^2
\le
(1-\alpha_h)\|e_{h,k}\|^2
+
2\mathbb E_k\|\Delta_{h,k}\|^2
+
2\alpha_h^2\mathbb E_k\|\xi_{h,k}\|^2.
\]
By bounded variance,
$\mathbb E_k\|\xi_{h,k}\|^2\le \frac{\nu_h^2}{N_h}$, and
by Assumption~\ref{ass:ms_smooth},
\[
\mathbb E_k\|\Delta_{h,k}\|^2
\le
\frac{\mathcal L_h^2}{N_h}\|\mathbf x_{k+1}-\mathbf x_k\|^2
=
\frac{\mathcal L_h^2\eta^2}{N_h}\|\tilde{\mathbf d}_k\|^2.
\]
Thus,
\[
\mathbb E_k\|e_{h,k+1}\|^2
\le
(1-\alpha_h)\|e_{h,k}\|^2
+
\frac{2\alpha_h^2\nu_h^2}{N_h}
+
\frac{2\mathcal L_h^2\eta^2}{N_h}\|\tilde{\mathbf d}_k\|^2.
\]
Summing over \(k=0,\ldots,K-1\), using
\(
\mathbb E\|e_{h,0}\|^2\le \frac{\nu_h^2}{N_h},
\)
and dividing by \(\alpha_h K\), gives the desired bound after enlarging constants.
\end{proof}

\begin{proof}[Proof of Theorem~\ref{rem:vr-pr-sdbpg}]
Let
$\mathbf  d_k:=\Psi(\nabla f(\mathbf  x_k),\nabla g(\mathbf  x_k))$ and $
\tilde {\mathbf d}_k:=\Psi(v_f^k,v_g^k)$,
and define the direction error
$e_k:=\tilde {\mathbf d}_k-\mathbf d_k$.

We first control the averaged exact direction norm
$D:=\frac1K\sum_{k=0}^{K-1}\mathbb E\|\mathbf  d_k\|^2$,
and then use a separate descent argument on \(g\) to control
$\frac1K\sum_{k=0}^{K-1}\mathbb E\|\nabla g(\mathbf  x_k)\|^2$.

By Lemma~\ref{lem:psi_lip}, with
$p_1=v_f^k, q_1=v_g^k,
p_2=\nabla f(\mathbf x_k), q_2=\nabla g(\mathbf x_k)$,
and using Assumption \ref{assump:gradf-g-lip}, we have
\[
\|e_k\|^2
\le
2(L_\Psi^f)^2\|v_f^k-\nabla f(\mathbf  x_k)\|^2
+
2(L_\Psi^g)^2\|v_g^k-\nabla g(\mathbf  x_k)\|^2.
\]
Define
$\bar E^2
:=
\frac1K\sum_{k=0}^{K-1}\mathbb E\|e_k\|^2$.
Then, by the definition of \(\bar\sigma_f^2\) and \(\bar\sigma_g^2\),
\[
\bar E^2
\le
2(L_\Psi^f)^2\bar\sigma_f^2
+
2(L_\Psi^g)^2\bar\sigma_g^2.
\]

We now derive a Lyapunov descent bound. Since
\(
\Phi(\mathbf  x):=f(\mathbf  x)+c g(\mathbf  x)
\)
has \(L_\Phi=L_f+cL_g\)--Lipschitz gradient, and since
$\mathbf  x_{k+1}=\mathbf  x_k-\eta\tilde {\mathbf d}_k=\mathbf x_k-\eta(\mathbf d_k+e_k)$,
we have
\[
\Phi_k-\Phi_{k+1}
\ge
\eta\nabla\Phi(\mathbf  x_k)^\top \mathbf d_k
+
\eta\nabla\Phi(\mathbf  x_k)^\top e_k
-
\frac{L_\Phi\eta^2}{2}\|\tilde {\mathbf d}_k\|^2.
\]
Using
$\nabla\Phi(\mathbf x_k)^\top \mathbf d_k
=
\|\mathbf  d_k\|^2
+
(c-\hat\lambda_k)\nabla g(\mathbf  x_k)^\top \mathbf  d_k$,
and the cross-term bound from the proof of Theorem~\ref{thm:pr-sdbpg},
\[
(c-\hat\lambda_k)\nabla g(\mathbf  x_k)^\top \mathbf  d_k
\ge
-\delta_c,
\]
where
\(
\delta_c
:=
(C_\Lambda+1-\mu\beta)\frac{G_f^2}{4\mu\beta}
=
\mathcal{O}(\epsilon),
\)
we obtain
\[
\nabla\Phi(\mathbf  x_k)^\top \mathbf  d_k
\ge
\|\mathbf  d_k\|^2-\delta_c.
\]
Moreover, since
\(
\|\nabla\Phi(\mathbf  x_k)\|\le G_\Phi
\)
and
\(
\|\tilde {\mathbf d}_k\|^2
=
\|\mathbf d_k+e_k\|^2
\le
2\|\mathbf  d_k\|^2+2\|e_k\|^2,
\)
we get
\[
\Phi_k-\Phi_{k+1}
\ge
\eta\|\mathbf d_k\|^2
-\eta\delta_c
-\eta G_\Phi\|e_k\|
-
L_\Phi\eta^2\|\mathbf d_k\|^2
-
L_\Phi\eta^2\|e_k\|^2.
\]
Since \(\eta\le 1/(2L_\Phi)\), this implies
\[
\Phi_k-\Phi_{k+1}
\ge
\frac{\eta}{2}\|\mathbf d_k\|^2
-
\eta\delta_c
-
\eta G_\Phi\|e_k\|
-
L_\Phi\eta^2\|e_k\|^2.
\]
Taking expectation, summing over \(k=0,\ldots,K-1\), and dividing by \(K\), we obtain
\[
D
\le
\frac{2\Delta_\Phi}{\eta K}
+
2\delta_c
+
2G_\Phi\sqrt{\bar E^2}
+
2L_\Phi\eta\bar E^2.
\]
In addition, if we define
$M
:=
\frac1K\sum_{k=0}^{K-1}\mathbb E\|\tilde {\mathbf  d}_k\|^2$,
then,
$M
\le
2D+2\bar E^2$.

Next, we control \(\bar E^2\) using the STORM tracking bound. By Lemma~\ref{lem:storm_tracking}, the choices of \(N_f,N_g\), and the choice of \(K\), the transient term satisfies
\[
\frac{\nu_h^3}{\mathcal L_h\eta K N_h}
\le
\frac{\mathcal L_h\eta\nu_h}{N_h},
\qquad h\in\{f,g\}.
\]
Therefore, for sufficiently large constants \(C_1,C_2\), Lemma~\ref{lem:storm_tracking} gives
\[
\bar E^2
\le
\frac{C\epsilon^2}{G_\Phi^2}(1+M),
\]
for an absolute constant \(C>0\). Let
$a:=\frac{C\epsilon^2}{G_\Phi^2}$.
For sufficiently small \(\epsilon\), we have \(a\le 1/8\). Since
$M\le 2D+2\bar E^2$,
we have
$\bar E^2
\le
a(1+2D+2\bar E^2)$.

Rearranging and using \(a\le 1/8\) yields
\[
\bar E^2
\le
2a(1+2D).
\]
Consequently, there exists an absolute constant \(C'>0\) such that
\[
G_\Phi\sqrt{\bar E^2}
\le
C'\epsilon\sqrt{1+D}.
\]

Substituting this bound into the inequality for \(D\), and using
$\frac{2\Delta_\Phi}{\eta K}=\mathcal{O}(\epsilon)$ and $
\delta_c=\mathcal{O}(\epsilon)$,
we obtain
\[
D
\le
\mathcal{O}(\epsilon)
+
C'\epsilon\sqrt{1+D}
+
2L_\Phi\eta\bar E^2.
\]
Since
$L_\Phi\eta=\mathcal{O}(\epsilon)$
and
$\bar E^2\le 2a(1+2D)$,
the last term is bounded by \(\mathcal{O}(\epsilon)(1+D)\). Hence,
\[
D
\le
\mathcal{O}(\epsilon)
+
\mathcal{O}(\epsilon)\sqrt{1+D}
+
\mathcal{O}(\epsilon)(1+D).
\]
Using \(\sqrt{1+D}\le 1+D\), we get
$D
\le
\mathcal{O}(\epsilon)+\mathcal{O}(\epsilon)D$.

For sufficiently small \(\epsilon\), the \(\mathcal{O}(\epsilon)D\) term can be absorbed into the left-hand side. Therefore,
$D=\mathcal{O}(\epsilon)$.
It follows that
$\bar E^2
=
\mathcal{O}\left(\frac{\epsilon^2}{G_\Phi^2}\right)$ and $
M
=
\mathcal{O}(\epsilon)+\mathcal{O}\left(\frac{\epsilon^2}{G_\Phi^2}\right)$.
In particular,
\[
\frac1K\sum_{k=0}^{K-1}\mathbb E\|\mathbf d_k\|^2
=
D
\le
\mathcal{O}(\epsilon).
\]

It remains to control the lower-level stationarity measure. By \(L_g\)-smoothness of \(g\),
\[
g(\mathbf  x_k)-g(\mathbf  x_{k+1})
\ge
\eta\nabla g(\mathbf  x_k)^\top \mathbf d_k
+
\eta\nabla g(\mathbf  x_k)^\top e_k
-
\frac{L_g\eta^2}{2}\|\tilde {\mathbf  d}_k\|^2.
\]
Using Lemma~\ref{lem:G1}, and using
\[
\nabla g(\mathbf  x_k)^\top \mathbf d_k
\ge
\tilde\beta\|\nabla g(\mathbf x_k)\|^2-\Delta_0,
\]
where
\[
\tilde\beta:=\frac{\mu\beta}{2(1+\mu)},
\qquad
\Delta_0:=\frac{(1+\mu)G_f^2}{2\mu\beta}.
\]
and,
\[
|\nabla g(\mathbf x_k)^\top e_k|
\le
C_g\|e_k\|.
\]
Taking expectation, summing over \(k=0,\ldots,K-1\), and dividing by \(K\tilde\beta\), we obtain
\[
\frac1K\sum_{k=0}^{K-1}
\mathbb E\|\nabla g(\mathbf x_k)\|^2
\le
\frac{\Delta_g}{\eta K\tilde\beta}
+
\frac{\Delta_0}{\tilde\beta}
+
\frac{C_g\sqrt{\bar E^2}}{\tilde\beta}
+
\frac{L_g\eta}{2\tilde\beta}M.
\]

We now bound each term on the right-hand side. By the choice of \(K\),
$\frac{\Delta_g}{\eta K\tilde\beta}
\le
\mathcal{O}(\epsilon)$.
Moreover, since \(\mu=1\), we have
$\tilde\beta=\frac{\beta}{4}=\frac{G_f^3}{\epsilon}$, $
\Delta_0=\frac{G_f^2}{\beta}=\mathcal{O}(\epsilon)$,
and hence
\(
\frac{\Delta_0}{\tilde\beta}
=
\mathcal{O}(\epsilon^2).
\)
Furthermore, since
\(
\sqrt{\bar E^2}
=
\mathcal{O}\left(\frac{\epsilon}{G_\Phi}\right),
\)
we get
\(
\frac{C_g\sqrt{\bar E^2}}{\tilde\beta}
\le
\mathcal{O}(\epsilon).
\)
Finally,
$\frac{L_g\eta}{\tilde\beta}
=
\mathcal{O}(\epsilon^2)$ and $
M=\mathcal{O}(\epsilon)+\mathcal{O}\left(\frac{\epsilon^2}{G_\Phi^2}\right)$.
Therefore,
$\frac{L_g\eta}{2\tilde\beta}M
\le
\mathcal{O}(\epsilon)$.

Combining these bounds gives 
$\frac1K\sum_{k=0}^{K-1}
\mathbb E\|\nabla g(\mathbf x_k)\|^2
\le
\mathcal{O}(\epsilon)$. 

Sampling \(k^*\sim\mathrm{Unif}\{0,\ldots,K-1\}\), we have
\[
\mathbb E\|\mathbf  d_{k^*}\|^2
=
\frac1K\sum_{k=0}^{K-1}\mathbb E\|\mathbf  d_k\|^2
\le
\mathcal{O}(\epsilon),\quad
\mathbb E\|\nabla g(\mathbf x_{k^*})\|^2
=
\frac1K\sum_{k=0}^{K-1}\mathbb E\|\nabla g(\mathbf x_k)\|^2
\le
\mathcal{O}(\epsilon).
\]

It remains to verify the claimed sample complexities. Under the chosen parameters,
$c=\mathcal{O}(\epsilon^{-1}),
L_\Phi=\mathcal{O}(\epsilon^{-1}),
\Delta_\Phi\le \Delta_f+c\Delta_g$.
Therefore,
$K=\mathcal{O}(\epsilon^{-4})$
up to problem-dependent constants. Moreover,
$L_\Psi^f=\mathcal{O}(1),
L_\Psi^g=\mathcal{O}(\epsilon^{-1}),
G_\Phi=\mathcal{O}(\epsilon^{-1}),
\eta=\mathcal{O}(\epsilon^2)$.
Thus,
$N_f=\mathcal{O}(\epsilon^{-2})$ and $
N_g=\mathcal{O}(\epsilon^{-4})$.
Consequently, the total sample complexities are 
$K N_f=\mathcal{O}(\epsilon^{-6})$ and $K N_g=\mathcal{O}(\epsilon^{-8})$.
\end{proof}

\vspace{-5mm}
\section{Detailed Experimental Setup}
\label{app:exp_details}

We provide complete details of the experimental setup used in the numerical evaluation, including datasets, model architectures, loss functions, and implementation specifics.

We consider a stochastic nonconvex bilevel formulation arising in large language model (LLM) unlearning. The problem is given by
\begin{align*}
    \min_{\theta \in \Theta} \quad & f(\theta) := \mathbb{E}_{\xi_f}[\ell_{\mathrm{ret}}(y \mid x; \theta)]
    \quad \text{s.t.} \quad 
    \Theta = \arg\min_{\theta} \; g(\theta) := \mathbb{E}_{\xi_g}[\ell_{\mathrm{for}}(y \mid x; \theta)],
\end{align*}
where $\xi_f$ and $\xi_g$ are sampled from the retain and forget datasets, respectively. The retain objective $\ell_{\mathrm{ret}}$ is implemented as the standard cross-entropy loss ($\ell_{ret}(y \mid x; \theta) = -\log \pi(y \mid x; \theta)$). The forget objective is instantiated using the negative preference optimization (NPO) loss defined as
\[
\ell_{\mathrm{NPO},\beta'}(y \mid x; \theta)
=
\frac{2}{\beta'}\log\!\left(1+\left(\frac{\pi(y \mid x; \theta)}{\pi(y \mid x; \theta_0)}\right)^{\beta'}\right),
\]
where $\theta_0$ denotes a fixed reference (pre-trained) model and $\beta' \ge 0$ controls the strength of the forgetting penalty. Both objectives are nonconvex due to the underlying neural network architecture.

All experiments are conducted on Google Colab using cloud-based NVIDIA G4 GPUs. The code is developed from a personal computer running macOS Sonoma on an Apple M2 Pro chip with 8GB memory. We conduct experiments on two benchmark public datasets obtained through the Hugging Face platform. The first dataset is TOFU\footnote{\url{https://huggingface.co/datasets/locuslab/TOFU}}\citep{maini2024tofu} (Task of Fictitious Unlearning), which contains $4000$ question-answer pairs generated from biographies of $200$ fictitious authors. Each sample is of the form $(x, y)$, where $x$ is a natural language query and $y$ is the corresponding answer. We use the \texttt{forget05} configuration, where $5\%$ of the data is designated as the forget set, resulting in $200$ forget samples and $3800$ retain samples.

The second dataset is MUSE-News\footnote{\url{https://huggingface.co/datasets/muse-bench/MUSE-News}}\citep{shi2024muse} (Machine Unlearning Six-way Evaluation), which consists of news-based question-answer pairs designed for evaluating unlearning performance. We use the \texttt{knowmem} subset, which contains two disjoint splits: \texttt{forget\_qa} with $100$ samples and \texttt{retain\_qa} with $100$ samples.

We use publicly available pre-trained models obtained through the Hugging Face platform, which are accessed under their corresponding Hugging Face license and usage agreement. Access to these models requires acceptance of the model provider’s terms of use. For the TOFU dataset, we use the instruction-tuned model \texttt{Llama-3.2-1B-Instruct}\footnote{\url{https://huggingface.co/meta-llama/Llama-3.2-1B-Instruct}}\citep{dorna2025openunlearning}, a $1.23$B-parameter autoregressive transformer. For the MUSE-News dataset, we use \texttt{Llama-2-7B}\footnote{\url{https://huggingface.co/meta-llama/Llama-2-7b}}\citep{shi2024muse}, a $7$B-parameter autoregressive transformer trained on large-scale corpora. We employ Low-Rank Adaptation (LoRA) to fine-tune the pre-trained model, updating only a small set of low-rank adaptation parameters while keeping the base model frozen, which substantially reduces the number of trainable parameters and improves computational efficiency.

In the TOFU experiment, the stochastic implementations (SDBPG, PR-SDBPG, VR-PR-SDBPG) use mini-batches of size $4$ from both the retain and forget datasets, resulting in $8$ samples per iteration. The algorithm is run for $1500$ iterations (total $12{,}000$ samples). The deterministic variant (DDBPG) uses a full batch of size $200$ from the forget dataset and a fixed subset of size $400$ (out of $3800$) from the retain dataset, resulting in $600$ samples per iteration, and is run for $60$ iterations. A fully deterministic implementation using the entire retain dataset ($3800$ samples) would be computationally prohibitive in this setting.

In the MUSE-News experiment, the stochastic methods use mini-batches of size $2$ from both datasets, resulting in $4$ samples per iteration, and is run for $1500$ iterations (total $6000$ samples). The deterministic method uses full-batch gradients over all $100$ retain, and $100$ forget samples, resulting in $200$ samples per iteration, and is run for $80$ iterations. Although full-batch gradients are feasible in this case, they remain significantly more expensive per iteration.


The hyperparameters in both experiments are set to
$\eta_k=\eta = 5 \times 10^{-6}, 
\beta_k=\beta = 0.2, 
\gamma_k=\gamma = 10^{-4}$, $\mu=1$, $\alpha_{STORM}=0.2$, $\beta' = 0.05$, which are tuned to achieve the best empirical performance for the considered methods.


\vspace{-5mm}

\paragraph{Additional comparison versus iterations.}
For completeness, we also report the convergence behavior of the methods as a function of the number of iterations (see Fig.~\ref{fig:iter}). While the deterministic method appears to converge in fewer iterations, this observation is expected and does not reflect a true computational advantage. In particular, each deterministic iteration uses substantially larger batches compared to the stochastic methods ($600$ samples vs.\ $8$ samples in TOFU, and 200 samples vs.\ $4$ samples in MUSE-News). As a result, iteration counts are not directly comparable across methods. This further motivates our primary evaluation based on the number of processed samples and wall-clock time, which provides a more meaningful measure of computational efficiency.

Finally, we would like to highlight that in this experiment, lower-level stationarity is achieved by the end of the algorithm run. This indicates that the iterates enter the bad region $\widetilde{R}_k$ only for a small fraction of iterations, during which upper-level stationarity is also attained; hence, the observed behavior is consistent with Assumption~\ref{assump:rare-visit}. Therefore, the improved performance of SDBPG is consistent with the theoretical guarantee in Theorem~\ref{thm:theorem2}.

\begin{figure}[t]
\centering

\begin{subfigure}{0.24\linewidth}
    \centering
    \includegraphics[width=\linewidth]{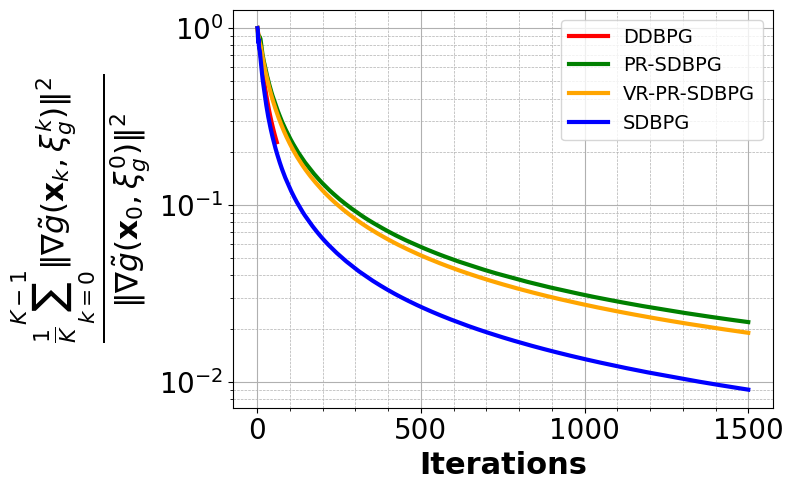}
    \caption{}
\end{subfigure}
\hfill
\begin{subfigure}{0.24\linewidth}
    \centering
    \includegraphics[width=\linewidth]{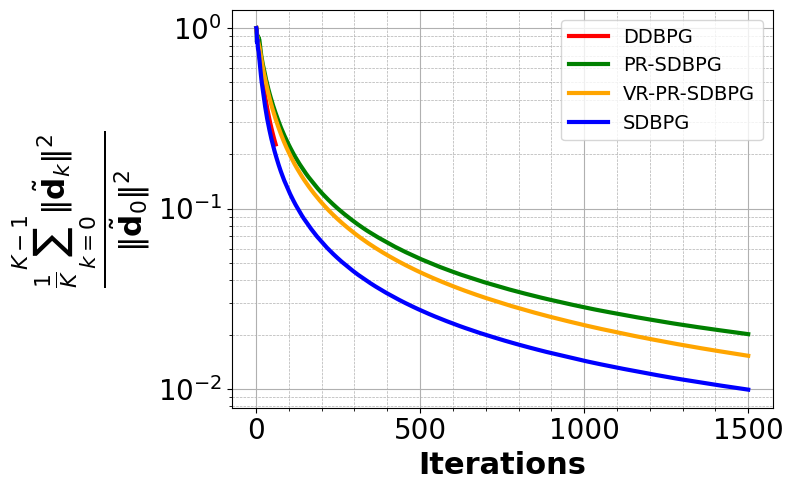}
    \caption{}
\end{subfigure}
\begin{subfigure}{0.24\linewidth}
    \centering
    \includegraphics[width=\linewidth]{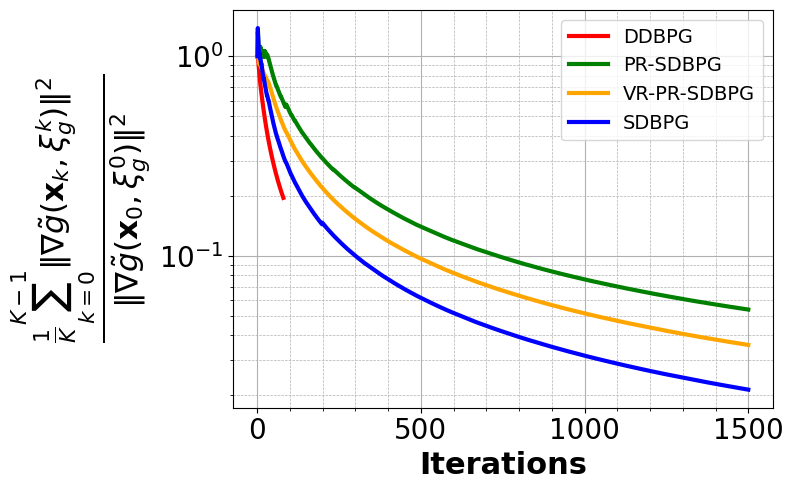}
    \caption{}
\end{subfigure}
\hfill
\begin{subfigure}{0.24\linewidth}
    \centering
    \includegraphics[width=\linewidth]{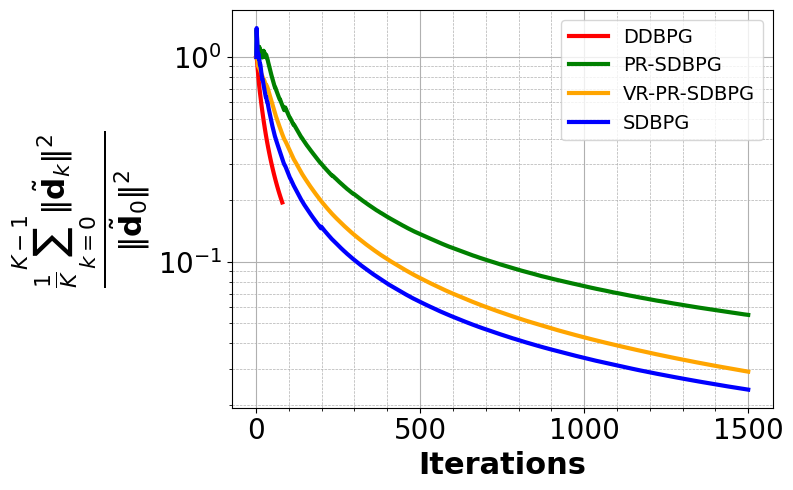}
    \caption{}
\end{subfigure}
\vspace{-2mm}
\caption{Comparison of SDBPG (blue), PR-SDBPG (green), and VR-PR-SDBPG (orange) with the deterministic counterpart DDBPG (red) versus iterations on
TOFU (a-b) and MUSE-News (c-d).
}
\label{fig:iter}
\end{figure}

\end{document}